\documentclass{amsart}

\RequirePackage{amsthm,amsmath,amsfonts,amssymb}
\RequirePackage[numbers]{natbib}

\usepackage{bbm}

\usepackage{amssymb,amsmath}
\usepackage{graphicx} 

\usepackage{comment}

\usepackage{booktabs} 
\usepackage{array} 
\usepackage{paralist} 
\usepackage{verbatim} 
\usepackage{subfig} 

\usepackage{amsthm}
\usepackage{fancybox,ascmac}
\usepackage{wrapfig}
\usepackage{pifont}

\usepackage{xifthen}
\usepackage{subfig}
 \usepackage{pkgfile}

\numberwithin{equation}{section}
\theoremstyle{plain}
 \newtheorem{lemma}[thm]{Lemma}
 \newtheorem{Prop}[thm]{Proposition}
\theoremstyle{remark}
 \newtheorem{asmp}[thm]{Assumption}
 \newtheorem{defn}[thm]{Definition}
\def \be{\begin{equation}}
\def \ee{\end{equation}}
\def \bt{\begin{theorem}} 
\def \et{\end{theorem}}
\def \bl{\begin{lemma}} 
\def \el{\end{lemma}}
\def \bea{\begin{eqnarray}}
\def \eea{\end{eqnarray}}
\def \bas{\begin{eqnarray*}}
\def \eas{\end{eqnarray*}}

\renewcommand{\eqref}[1]{(\ref{#1})} 

\def \La{\Lambda}

\def \Om{\Omega}

\newcommand{\bC}{\mathbb{C}}
\newcommand{\bD}{\mathbb{D}}
\newcommand{\bE}{\mathbb{E}}

\newcommand{\bP}{\mathbb{P}}

\newcommand{\bR}{\mathbb{R}}

\newcommand{\mt}{\mathbb{T}}

\newcommand{\bZ}{\mathbb{Z}}

\newcommand{\GG}{\mathcal{G}}

\newcommand{\TT}{\mathcal{T}}

\newcommand{\sfA}{\mathsf {A}}
\newcommand{\sfB}{\mathsf {B}}

\newcommand{\sfD}{\mathsf {D}}

\newcommand{\sfG}{\mathsf {G}}
\newcommand{\sfH}{\mathsf {H}}
\newcommand{\sfI}{\mathsf {I}}

\newcommand{\sfO}{\mathsf {O}}

\def\b1{\mathbf 1}

\def \H{{\mathbf H}}

\def \P{{\mathbf P}}

\newcommand{\Log}{{\rm Log}}

\begin{document}

\title{Spectral measure for uniform $d$-regular digraphs}

 \author{Arka Adhikari}
    \address{Arka Adhikari\hfill\break
    Department of Mathematics, University of Maryland-College Park, College Park, MD, USA}
    \email{arkaa@umd.edu}

\author{Amir Dembo}
\address{Amir Dembo \hfil \break Department of Mathematics and Statistics, Stanford University, Stanford, CA, USA}
\email{adembo@stanford.edu}

\begin{abstract}
Consider the matrix $\sfA_\GG$ chosen uniformly at random from the finite 
set of all $N$-dimensional matrices of zero main-diagonal and binary entries, 
having each row and column of $\sfA_\GG$ sum to $d$. 
That is, the adjacency matrix for the uniformly random 
$d$-regular simple digraph $\GG$. Fixing $d \ge 3$, it has long been conjectured 
that as $N \to \infty$ the corresponding empirical eigenvalue distributions converge
weakly, in probability, to an explicit non-random limit,
given by the Brown measure of the free sum of $d$ Haar unitary operators. 
We reduce this conjecture to bounding the decay in $N$ of the probability that
the minimal singular value of the shifted matrix $\sfA(w) = \sfA_\GG - w \sfI$
is very small. While the latter remains a challenging task, the required bound is 
comparable to the recently established control on the singularity of $\sfA_\GG$.
The reduction is achieved here by sharp estimates 
on the behavior at large $N$, near the real line, of the Green's function (aka resolvent) 
of the Hermitization of $\sfA(w)$, which is of independent interest.
\end{abstract}



\maketitle

\section{Introduction}\label{sec-intro}

The method of moments and the Stieltjes transform approach 
provide rather precise information on asymptotics of the 
Empirical Spectral Distribution (in short \abbr{esd}), 
for many {\em Hermitian} random matrix models.
In contrast, both methods fail for {\em non-Hermitian} matrix models, 
and the only available general scheme for finding the limiting spectral 
distribution in such cases is the one proposed by Girko (in \cite{Girko84}). 
It is extremely challenging to rigorously justify this scheme, even 
for the matrix model consisting of i.i.d. entries (of zero mean and finite variance). 
Indeed, the {\em circular law} conjecture, for the i.i.d. case, was established
in full generality only in 2010, after rather long series of partial results 
(see \cite{TV10} and historical references in \cite{BC12}). Beyond this 
simple model, only a few results have been rigorously proved 
in the non-Hermitian regime. Our focus here is on the long standing conjecture 
about the limit, as $N \to \infty$, of the \abbr{esd} of the adjacency matrix 
$\sfA_\GG$ in case $\GG$ is a uniformly random $d$-regular {\em directed graph} (aka digraph) 
of $N$ vertices.
Specifically, in this context it is conjectured (see \cite{BC12}), that, fixing $d \ge 3$, 
such \abbr{esd} 
converge weakly, in probability, to {the oriented Kesten-McKay law. Namely, the}
measure $\mu_d$ on the complex plane, 
whose density with respect to Lebesgue measure on $\bC$ is  
\beq\label{eq:dreg}
h_d(w) := \f{1}{\pi} \f{d^2(d-1)}{(d^2 - |w|^2)^2}\bI_{\{|w| \le \sqrt{d}\}} 
\eeq
{(c.f. \cite{Cook,MNR} for more insights on this conjecture).
Related to this conjecture, \cite{Coste21} applies the high trace method to study 
the spectral gap of such matrices, proving in particular that the limiting \abbr{esd} 
support must be the disk of radius $\sqrt{d}$ with no outliers apart from the trivial eigenvalue. 
This is also one of the consequences of \cite{CLZ} work on the asymptotic of the 
characteristic polynomial of such $\sfA_\GG$ 
outside the disk of radius $\sqrt{d}$ (for more on this direction, see \cite{CLZ} and the
references therein). Note that as $d \to \infty$, the oriented Kesten-McKay law $\mu_d$, 
rescaled by $\sqrt{d}$, converges to the circular law. This suggests that the circular law should 
hold for the limiting \abbr{esd} of $\frac{1}{\sqrt{d}} \sfA_\GG$ whenever $d \wedge (N-d) \to \infty$ 
as $N \to \infty$, as indeed shown by \cite{Cook,LLTTY} (see also \cite{BCZ14} for this result in 
the related permutation model, with at least poly-log growing degree).}

From a graph-theoretic point of view, it is most natural to consider the
{\em simple digraph} (\abbr{sd}) version of this problem,  whereby we restrict $\GG$ to be a 
simple digraph, or alternatively choose $\sfA_\GG$ uniformly from the ensemble of 
all matrices of $\{0,1\}$-valued entries with zero main diagonal and having each of 
its row and column sum to $d$. {Taking advantage of the contiguity results of \cite{Janson95}, 
one} may alternatively aim at proving this conjecture for the 
{\em configuration model} (\abbr{cm}), where 
$\sfA_\GG$ is uniformly chosen from among all matrices of non-negative integer 
entries with all rows and columns sum to $d$ (that is, $\GG$ is still uniform,
directed and $d$-regular, but now allowing also for self and multiple edges).
The preceding conjecture is best appreciated in the context of yet another 
version, the so called {\em permutation model} (\abbr{pm}), where $\sfA_\GG$ is 
taken as the sum of $d$ uniformly and independently chosen 
$N$-dimensional permutation matrices (that is, with $\GG$ a $d$-regular
multi-graph, now of a slightly non-uniform law). 

Indeed, permutation matrices form a 
subgroup of the group of orthogonal matrices, and it was shown in \cite{BD13} that 
as $N \to \infty$, the \abbr{esd} for the sum of $d$ independently chosen, \emph{Haar 
distributed}
$N$-dimensional, 
orthogonal matrices, converge weakly, 
in probability, to the measure $\mu_d$
whose density is $h_d(\cdot)$ of \eqref{eq:dreg}. This result of \cite{BD13} 
brings us back to the rationale behind the conjecture
(which should hold for all three versions of uniform $d$-regular digraph 
we have mentioned). Namely, the observation that $\mu_d$ is the {\em Brown measure} of the
{\em free sum} of $d \ge 2$ {\em Haar unitary operators} (see \cite[Example 5.5]{HL00}). 
To understand the plausible route and challenges in affirming this conjecture, recall the 
definition of Brown measure for a bounded operator (see \cite[Page 333]{HL00}),
where $\langle \Log, \mu \rangle := \int_{\bR} \log |x| d\mu(x)$
for any probability measure $\mu$ on $\bR$ (for which such integral is well defined).
\begin{dfn}
Let $(\cA, \tau)$ be a non-commutative $W^*$-probability space, i.e. 
a von Neumann algebra $\cA$ with a normal faithful tracial 
state $\tau$ (see \cite[Defn. 5.2.26]{AGZ10}). 
For $h$ a positive element in $\cA$, let $\mu_h$ denote the unique 
probability measure on $\bR^+$ such that 
$\tau(h^n) = \int t^n d \mu_h(t)$ for all $n \in \bZ^+$. 
The Brown measure $\mu_a$ associated with each bounded $a \in \cA$, 
is the Riesz measure corresponding to the 
$[-\infty,\infty)$-valued sub-harmonic function 
$w \mapsto  \langle \Log, \mu_{|a-w|} \rangle$ on $\bC$.   
That is, $\mu_a$ is the unique Borel probability measure on $\bC$ such that   
\beq\label{def:bmeas}
d \mu_a (w) = \f{1}{2 \pi} \Delta_w
\langle \Log, \mu_{|a-w|} \rangle 
\, dw
, 
\eeq
where $\Delta_w$ denotes the two-dimensional Laplacian operator 
(with respect to $w \in \bC$), and the identity (\ref{def:bmeas}) 
holds in distribution sense (i.e. when integrated against 
any test function $\psi \in C^\infty_c(\bC)$).
\end{dfn}

\noindent
Recall that as $N \to \infty$, independent Haar distributed $N$-dimensional orthogonal
matrices converge in $\star$-moments (see \cite{Sniady02} for a definition), to the 
collection $\{u_i\}_{i =1}^d$ of $\star$-{\em free} Haar unitary operators 
(see \cite[Thm.~5.4.10]{AGZ10}). The challenge, even in the context of \cite{BD13},
stems from the fact that the convergence of $\star$-moments (or even the stronger
convergence in distribution of traffics, see \cite{Male11}), does not imply convergence 
of the corresponding Brown measures\footnote{The Brown measure of a matrix 
is its \abbr{ESD} (see \cite[Prop.~1]{Sniady02})} (see \cite[\textsection 2.6]{Sniady02}). 
Moreover, while \cite[Thm.~6]{Sniady02} shows that if the original matrices 
are perturbed by adding small Gaussian (of {\em unknown variance}), then the Brown measures do converge,
removing the Gaussian, or merely identifying the variance needed, are often hard tasks (e.g.
\cite[Prop.~7 \& Cor.~8]{GWZ14} provide an example of an ensemble where no
Gaussian matrix of polynomially vanishing variance can regularize the Brown measures, in this sense). 
In contrast, for Toeplitz matrices, many spectral features of regularization
by adding a matrix of polynomially small norm,  are by now fairly well understood
(c.f. \cite{SV} for Gaussian perturbations, or \cite{BPZ20} for non-Gaussian ones).

Thus, instead of relying on the preceding free probability intuition, we shall utilize Girko's approach, which
we now summarize. First, one associates to any $N$-dimensional (non-Hermitian) 
matrix $\sfA_N$ and every $w \in \bC$ the $2N$-dimensional 
{\em Hermitian} matrix 
\beq\label{eq:herm-def}
\sfH^{\sfA_N} (0,w) := 
 \begin{bmatrix}
  0 & \sfA_N -w \sfI_N \\
 \sfA_N^* - \overline{w} \sfI_N & 0   
 \end{bmatrix} \,.
\eeq  
Applying Green's formula to the characteristic polynomial 
of a matrix $\sfA_N$ (whose \abbr{esd} we 
denote hereafter by $L_{\sfA_N}$), results with 
\begin{align*}
\int_\bC \psi(w) dL_{\sfA_N}(w) =  & 
\f{1}{2 \pi N} \int_\bC \Delta \psi(w) \log | \det  (w \sfI_N-\sfA_N)| dw 
, \qquad \forall \psi \in C_c^2(\bC)\,. 
\end{align*} 
The eigenvalues of $\sfH^{\sfA_N}(0,w)$ are merely
$\pm 1$ times the singular values of $w \sfI_N- \sfA_N$. Hence, denoting the \abbr{esd} 
of $\sfH^{\sfA_N}(0,w)$ by $\nu_N^w$, we have that
\beq
\f{1}{N} \log |\det (w \sfI_N-\sfA_N)| = \f{1}{2N} \log |\det \sfH^{\sfA_N}(0,w)| 
= \langle \Log, \nu_N^w \rangle \, , \notag 
\eeq
out of which we deduce Girko's formula,
\beq
\int_\bC \psi(w)dL_{\sfA_N}(w) = \f{1}{2 \pi} 
\int_\bC \Delta \psi(w) \langle \Log, \nu_N^w \rangle dw \,, \qquad \forall \psi \in C_c^2(\bC)\,.
\label{eq:girko_key_identity} 
\eeq
The preceding identity suggests the following recipe for proving 
convergence of $L_{\sfA_N}$ per given family of non-Hermitian random matrices $\{\sfA_N\}$ 
(to which we referred already as Girko's method).
\vskip5 pt

\noindent
 {\bf Step 1}: Show that for (Lebesgue almost) 
every $w \in \bC$, as $N \ra \infty$ the measures 
$\nu_N^w$ converge weakly, in probability, to 
some measure $\nu^w$. 

\vskip5pt

\noindent
{\bf Step 2}: Justify that 
$\langle \Log, \nu_N^w \rangle \to \langle \Log, \nu^w \rangle$ 
in probability, for almost every $w \in \bC$
(which is the main technical challenge of this approach).

\vskip5pt

\noindent
{\bf Step 3}: A uniform integrability argument allows one to convert 
the $w$-a.e. convergence of $\langle \Log, \nu_N^w \rangle$ to 
the corresponding convergence for a suitable collection 
$\cS \subseteq C_c^2(\bC)$ of (smooth) test functions. 
It then follows from (\ref{eq:girko_key_identity}) that
for each fixed $\psi \in \cS$, we have the convergence in probability
\begin{align}
\int_\bC \psi(w) dL_{\sfA_N}(w) \ra & \f{1}{2 \pi} \int_\bC \Delta \psi(w) 
\langle \Log, \nu^w \rangle dw \nonumber \\ 
&= \int_\bC \psi(w) h(w) dw\,, \quad \text{for} \quad h(w):=\frac{1}{2\pi} \Delta \langle \Log, \nu^w \rangle \,,
\label{eq:step3}
\end{align}
provided $w \mapsto \langle \Log, \nu^w \rangle$ is smooth enough to justify the integration 
by parts. For $\cS$ large enough, this implies the convergence in probability of the \abbr{esd}-s 
$L_{\sfA_N}$ to a limit whose density
is $h(w)$.

For example, \cite{BD13} follows this approach, inductively over $d \ge 2$, and 
for the core induction step they adapt the arguments of \cite{GKZ11}  
(who for Haar distributed orthogonal $\sfO_N$,
independent of the uniformly bounded, non-negative definite, diagonal $\sfD_N$, show that the 
weak convergence of $L_{\sfD_N}$, in probability, implies the same for 
$L_{\sfO_N \sfD_N}$, whose limit is 
the Brown measure of a certain 
operator). Utilizing this, it suffices in our case to only consider {\bf Step 2}. Indeed, 
for any adjacency matrix $\sfA_N$ of a $d$-regular digraph we have that $d^{-2} \sfA_N^* \sfA_N$
is a (Markov) probability transition matrix, hence of spectral norm one. Any 
such $\sfA_N$ must then have singular values bounded by $d$ and a Hilbert-Schmidt norm 
$\|\sfA_N\|_2 \le d \sqrt{N}$. The same applies for any sum of 
$d$ unitary matrices, and in particular to the sum $\sfB_N$ of $d$ independent 
Haar distributed unitary matrices. Recall, from \cite[Prop.~1.4]{BD13}, 
that $L_{\sfB_N}$ converges weakly, in probability, to the measure $\mu_d$ 
(of density $h_d(w)$ as in \eqref{eq:dreg}), and from \cite[Lemma 3.2]{BD13}, 
that for almost every $w \in \bC$, as $N \to \infty$,
\[
\frac{1}{N} \log |\det(w \sfI_N-\sfB_N)| \to \langle \Log, \nu^w \rangle \,.
\]
It then follows from applying the \textit{replacement principle} of \cite[Thm. 2.1]{TV10}
for ensembles $(\sfA_N)$ and $(\sfB_N)$, that \textbf{Step 2} already yields the 
conjectured convergence for any such $L_{\sfA_N}$.  

The following theorem, which is our main result, addresses the latter convergence,  {while also providing 
the corresponding local limit law}.
\begin{thm}\label{thm:main}
Let $\bP_{N,d}$ denote the uniform distribution over all $d$-regular, simple digraphs on $N$ vertices,
with $\sfA_N=A_\GG$ the adjacency matrix of such a digraph, and $s_1(w)$ the minimal singular 
value of $\sfA_N-w \sfI_N$.
Suppose for $\mathfrak{a}=0$,  any fixed $\kappa>0$ and a.e $w \in \bC$,
\begin{equation}\label{eq:lowesteigbnd}
\lim_{N \to \infty} N^{\mathfrak{a}} \,  \bP_{N,d} (- \log s_1 (w) \ge N^{\kappa}) = 0 \,.
\end{equation}
(a).  Girko's method {\bf Step 2} holds for our $\nu_N^w$
(i.e. $\langle \Log, \nu_N^w \rangle \to \langle \Log, \nu^w \rangle$, in probability),
and the \abbr{esd} $\{L_{\sfA_N}\}$ converge weakly, in probability 
to $\mu_d$, whose density is $h_d(w)$ of \eqref{eq:dreg}.
\newline
(b).  Further,  if \eqref{eq:lowesteigbnd} holds with $\mathfrak{a}(\kappa)>0$,  then for some 
$\mathfrak{a}_o>0$ any $w_o \in \bC$ and $\psi \in C_c^2(\bC)$,  setting $\psi_r (\cdot) := r^{2} \psi(r \, \cdot)$,  
we have that
\begin{equation}\label{eq:local-lim}
\lim_{N \to \infty} \bP_{N,d} \Big(  \big| \langle \psi_{N^a} (\cdot-w_o),{L_{{\sfA}_N}} - \mu_d \rangle \big|  
> \eta \Big)  = 0 \,, \qquad \forall \eta>0,  a\in (0,\mathfrak{a}_o) \,.
\end{equation}
\end{thm}
There is a long history regarding bounds on the smallest singular value of random matrix ensembles
which are of the type of \eqref{eq:lowesteigbnd} (see, e.g. \cite{TV10,RV14,BCZ14,LLTTY18,LT22,SSS}). 
Of particular note are the paper \cite{Huang21}, which shows that 
$\bP_{N,d}(s_1(0)=0) \to 0$ as $N \to \infty$ {(c.f. \cite{NP} for better decay rate),
and \cite{SSS}, which establishes 
\eqref{eq:lowesteigbnd} at $\kappa=1$ for the adjacency matrices of Erdos-Renyi($d/N$) digraphs.}
Unfortunately, \cite{Huang21,NP} rely on considering digraphs over finite fields, which does not seem applicable for an estimate such as \eqref{eq:lowesteigbnd}, {whereas independence of the
entries of $\sfA_N$ seems to be key to the success of the approach taken in \cite{SSS}.}

As we shortly show, Theorem \ref{thm:main} is an immediate consequence of the 
following local law for the Green's function $\sfG_N(z,w):=[\sfH^{\sfA_N}(0,w)-z \sfI_{2N}]^{-1}$ 
of our Hermitized matrices. 
\begin{thm}\label{thm:trace}
Let $m_\star (z,w) :=\int_{\bR} (\lambda-z)^{-1} d\nu^w (\lambda)$ (which is well defined
for $\Im(z)>0$).
For any $0 < \mathfrak{c} < 1$ and $d \ge 3$, there exist $\varepsilon =\varepsilon(\mathfrak{c},d)>0$  and
$\delta = \delta(\mathfrak{c},d) >0$
so that 
$\bP_{N,d}(\Omega'_N) \ge 1-O(N^{-\mathfrak{c}})$ for some
sets $\Omega'_N$ of $d$-regular, simple digraphs and for any compact $\mathbb{D} \subset \mathbb{C}$, 
as $N \to \infty$, 
\begin{equation}\label{eq:trG-conv}
\sup_{\GG \in \Omega_N'} \sup_{w \in \mathbb{D}, \Im(z) \ge N^{-\varepsilon}}
\Big| \frac{1}{2N} {\rm trace}(\sfG_N(z,w)) - m_\star (z,w) \Big| \le N^{-\delta} \,.
\end{equation}
\end{thm}

\begin{rmk} 
The values of $\delta(\mathfrak{c},d)$ and $\varepsilon(\mathfrak{c},d)$ in our proof are sub-optimal. 
While plausible to expect our method to yield $\varepsilon(\mathfrak{c},d) \to 1$ 
as $d \to \infty$ and $c \to 0$, this would entail substantially more effort.  We further note that 
near the edge of the support of  $\mu_d$ (namely,  at $|w| =\sqrt{d}$),  the singularity of the 
self-consistency equations causes weaker estimates on the Green's function,  but 
having small $\varepsilon$ and $\delta$ allows us to get the uniform bound 
in \eqref{eq:trG-conv},  avoiding effects due to such singularity near the edge.
\end{rmk}

\begin{rmk} We have stated Theorems \ref{thm:main} and \ref{thm:trace} in terms of uniformly 
random, \emph{simple} $d$-regular digraphs, namely, taken from the \abbr{sd} model. However,
our work yields precisely the same results when $\GG$ is taken from the configuration model.
\end{rmk}

\begin{rmk}  
Combining Theorem \ref{thm:trace} with \cite[Theorem 2.1]{BPZ19} (at $\alpha=0$),
one has, for the standard 
complex Ginibre matrix $\sfG_N$, that under $\bP_{N,d}$ the \abbr{esd} of 
$\sfA_N+N^{-\beta-\frac{1}{2}} \sfG_N$ converges weakly, in probability, 
to the oriented Kesten-McKay law $\mu_d$. Utilizing \cite[Theorem 1.8]{BPZ20}, 
the same applies for $\sfA_N+N^{-\beta} \sfO_N$, without requiring \eqref{eq:lowesteigbnd} to hold. 
\end{rmk}

Many recent papers deal with establishing local laws for graph models.  For example,  \cite{RHY19,RKY17,HuangYau} derive local laws via switching estimates that generate useful 
self-consistent equations.  The most relevant here are 
\cite{RHY19,HuangYau}, which introduce the strategy of preserving local neighborhoods and performing switches along the boundary of the local neighborhood.  In the graphs considered in those papers the local neighborhoods of most vertices match those of the root of an infinite $d$-regular tree,  so one applies perturbation theory and averaging estimates to compare the Green's functions of the graph with the one at the root of the infinite $d$-regular tree.  
The key for doing so are estimates on the Green's function on the tree,  which thanks to its homogeneous 
structure depend only on the distance between the two vertices.  This homogeneity of the tree Green's function
helps one establish the contractive property of the linearized self-consistency equations for the perturbation,
which is essential to the success of this approach.  

While we follow the same philosophy in our proof of Theorem \ref{thm:trace},  a major new challenge arises when considering digraphs,  whereby the limit object 
with respect to which we now perturb is the infinite $d$-regular directed tree (see Definition \ref{def:digraphtree}).
Due to having separate in-edges and out-edges in the directed tree,  homogeneity is lost,  with the
Green's function estimates now depending on the precise path taken between the two vertices on the tree.
This adds complexity to the perturbation estimates we need (see Section \ref{sec:Tree-like}),  with   
much extra care now required for comparing the Green's functions of the original and the switched 
graphs (specifically,  for controlling the effect that applying Schur complement 
formula for the removal of large neighborhoods has on the Green's functions of the switched graphs).  Lastly and
most importantly,  as further detailed at the start of Section \ref{sec:sec2},  in the digraph case the limiting (non-perturbative) Green's function evolution along some edges is expansive (instead of contractive),  and a very 
precise averaging is crucial to get the stability of the self-consistency equations.

Deferring for awhile the construction and the proof of all these Green's function estimates,
we proceed with a standard,  short derivation of Theorem \ref{thm:main} out of 
Theorem \ref{thm:trace}.  The remainder of this paper is devoted to the proof 
of Theorem \ref{thm:trace},  and other results which we get along the way 
(e.g. about the asymptotic behavior of the off-diagonal entries of $\sfG_N$).

\subsection{Proof of Theorem \ref{thm:main}}
Fixing $d \ge 3$,  $0<\mathfrak{c} <1$ and $\kappa < \varepsilon(\mathfrak{c},d)$ 
we restrict ourselves hereafter \abbr{wlog} to digraphs in $\Omega_N'$ from Theorem \ref{thm:trace},  where the estimate \eqref{eq:trG-conv} holds and denote by 
$\Omega_N(w)$ those digraphs for which $s_1(w)$,  the minimal singular value of $\sfA_\GG-w \sfI_N$
exceeds $\delta_N := e^{-N^{\kappa}}$.  
\newline
(a).  Fixing $w \in \bC$ where \eqref{eq:lowesteigbnd} holds (namely,  $\bP_{N,d}(\Omega_N(w)) \to 1$), 
recall that for any $z \in \bC_+$ and $N$, 
\begin{equation}\label{dfn:Stiel}
\frac{1}{2N} {\rm trace}(\sfG_N(z,w)) = \int_{\bR} (\lambda-z)^{-1} d\nu_N^w(\lambda) \,.
\end{equation}
In particular, \eqref{eq:trG-conv} implies the convergence as $N \to \infty$ 
of the Stieltjes transform of $\nu_N^w$ to that of $\nu^w$, at any $z \in \bC_+$, uniformly over 
$\Omega'_N$.  This of course implies that $|\langle f, \nu_N^w \rangle - \langle f, \nu^w \rangle| \to 0$
for any $f \in C_b(\bC)$ (which is {\bf Step 1} of Girko's method), also uniformly over 
digraphs in $\Omega_N'$.  Further, with
$\langle \Log, \mu \rangle_a^b := \int 1_{|x| \in [a,b]} \log |x| d\mu(x)$ for $a<b$,
and recalling that the spectral norm of $\sfA_\GG(w)$ is uniformly 
bounded (by $d+|w|$), 
we have for $\eta>0$ fixed, uniformly over $\Omega_N'$,  that as $N \to \infty$,
\[
|\langle \Log, \nu_N^w \rangle_\eta^\infty - \langle \Log, \nu^w \rangle_\eta^\infty| \to 0 \,.
\]
By definition,  on $\Omega_N(w)$ we have that $\nu_N^w(-(\delta_N,\delta_N))=0$.
Further,  from \cite[(4.18)]{BD13} we have for a.e.~$w$,
\begin{equation}\label{eq:tail-nu}
\Big| \langle \Log, \nu^w \rangle_0^\eta \Big|  \le C \eta | \log \eta| \,, \qquad \forall \, 0 \le \eta \le 1/2 \,,
\end{equation} 
Thus, it remains only to show that  
$\langle \Log, \nu_N^w \rangle_{\delta_N}^\eta \to 0$ uniformly on $\Omega_N$, when $N \to \infty$
followed by $\eta \to 0$. To this end,  note that \eqref{eq:trG-conv} also implies that for some 
$C<\infty$, any $N$, all digraphs in $\Omega'_N$ and $y>0$, 
\[
\nu_N^w ((-y,y)) \le C (y \vee N^{-\varepsilon})  
\]
(see \cite[(4.15)]{BD13} or \cite[Lemma 15]{GKZ11}). Combining this bound with integration by parts we find that 
\begin{equation}\label{bd:near-zero}
\langle \Log, \nu_N^w \rangle_{\delta_N}^\eta \le C \int_{\delta_N}^\eta \frac{1}{y} (y \vee N^{-\varepsilon}) dy  \le C ( N^{-\varepsilon} |\log \delta_N| + \eta ) = C (N^{\kappa-\varepsilon} + \eta) \,,
\end{equation}
which goes to zero when $N \to \infty$ followed by $\eta \to 0$,  thereby establishing {\bf Step 2} of 
Girko's method.
\newline
(b).
Next,  fixing  $\eta>0$ and $a<\mathfrak{a}_o$,  since the Lipschitz norm of $\psi_r(\cdot)$ is 
at most $r^3  \| \nabla \psi \|_\infty$,   
we can and shall replace $\langle \psi_{N^a} (\cdot-w_o),{L_{{\sfA}_N}} \rangle$  
in \eqref{eq:local-lim},  by $\langle \widetilde{\psi}_{N^a} (\cdot-w_o),{L_{{\sfA}_N}} \rangle$,  where
$\widetilde{\psi}_r (\cdot) := \bE_{U_N} [\psi_r(\cdot - U_N)]$ 
for $U_N$ uniformly distributed on ${\mathbb B}(0,N^{-3 \mathfrak{a}_o})$ independently of $\sfA_N$.
Hence,  we let 
\[
\Gamma_N(w) := \bE_{U_N} \big[ \langle \Log, \nu_N^{w+  U_N} \rangle \big] - \langle \Log, \nu^w \rangle \,,
\]
and combine Girko's formula \eqref{eq:girko_key_identity}
for $\widetilde{\psi}_{N^a}(\cdot)$ and $\psi_{N^a}(\cdot)$,  with Fubini's theorem,  to get
\eqref{eq:local-lim} out of 
\[
\lim_{N \to \infty} \bP_{N,d} \Big( \big| \int_\bC \Delta \psi_{N^a} (w) \Gamma_N(w_o+w) d w \big| > \eta \Big) = 0 \,.
\]

We will now obtain estimates on the differences of the empirical densities of $\nu^w_N$ and $\nu^w$ via the Helffer-Sjostrand formula. First, we remark that the Green's function $m^d_{\mathcal{T}}(z,w)$ is bounded by a constant via the computations of  Lemma \ref{lem:minftybnd}. Thus, the estimate  \eqref{eq:trG-conv} would also imply that $\frac{1}{2N} \text{trace}(G_N(z,w))$ is bounded by a constant independent of $N, z$ and $w$ as long as we consider $z$ with $\Im(z)\ge N^{-\varepsilon}$. Now, since the functions 
$\Im(z) \Im(m^d_{\mathcal{T}}(z,w))$ and $\Im(z)\frac{1}{2N} \Im(\text{trace}(G_N(z,w)))$ are both monotone increasing in $\Im(z)$. Thus, we have that,
\begin{equation}
\Im(m^d_{\mathcal{T}}(z,w)) \le \frac{C N^{-\varepsilon}}{\Im(z)}, \frac{1}{2N} \Im(\text{trace}(G_N(z,w))) \le \frac{C N^{-\varepsilon}}{\Im(z)}
\end{equation}
for any $z$ with $\Im(z) \le N^{-\varepsilon}$.

In addition, we have that,
\begin{equation}
\left|m^d_{\mathcal{T}}(z,w) - \frac{1}{2N} \Im(\text{trace}(G_N(z,w)))\right| \le \frac{N^{-\delta }}{\Im(z)}, \forall z: 1 \ge  \Im(z) \ge N^{-\varepsilon}.
\end{equation}

If we define $\rho = \nu^w_N - \nu^w$ and let $f_{\hat{\eta}}$ be a positive twice differentiable function that is $1$ on the region $(-E,E)$, 0 on $(-E - \hat{\eta}, E +\hat{\eta})$ and has first derivative bounded by $\frac{1}{\hat{\eta}}$ and second derivative bounded by $\frac{1}{\hat{\eta}^2}$, we can apply the Hellfer-Sjostrand formula as in \cite[Lemma 11.2]{EY17} 
to derive that,
\begin{equation}
\begin{aligned}
&\nu^w_N(-E,E) - \nu^w(-E,E) = \int_{-E-\hat{\eta}}^{E+ \hat{\eta}} f_{\hat{\eta}}(x) \text{d}\rho(x) - \int_E^{E+\hat{\eta}} f_{\hat{\eta}}(x) \text{d}\nu^w_N(x)\\&
\hspace{1 cm}- \int_{-E-\hat{\eta}}^{-E}f_{\hat{\eta}}(x) \text{d}\nu^w_N(x) + \int_E^{E+\hat{\eta}} f_{\hat{\eta}}(x) \text{d}\nu^w(x) +\int_{E}^{E+\hat{\eta}}f_{\hat{\eta}}(x) \text{d}\nu^w(x)\\
&\le C N^{-\varepsilon} \log N + N^{-\delta } \frac{N^{-2\varepsilon}}{\hat{\eta}^2} + C \hat{\eta}
\end{aligned}
\end{equation}
We used the fact that $ - \int_E^{E+\hat{\eta}} f_{\hat{\eta}}(x)$ is negative and that the density of $\nu^w$ is bounded (again via Lemma \ref{lem:minftybnd})  to control  $\int_{E}^{E+\hat{\eta}}f_{\hat{\eta}}(x) \text{d}\nu^w(x)$ for the last inequality. We can derive a lower bound for $\nu^w_N(-E,E) - \nu^w(-E,E)$ in a similar way. 

If we set $E$ to be any value $\gg N^{-\varepsilon}$, we can choose $\hat{\eta} = N^{-\varepsilon} \log N$ to assert that

\begin{equation}
|\nu^w_N(-y,y) - \nu^w(-y,y)| \le C[ N^{-\delta} + N^{-\varepsilon} \log N], \text{ } \forall \text{ } y \gg N^{-\varepsilon}.
\end{equation}
Using this estimate, we find that for $\chi = N^{-\varepsilon +\tau}$ for some small $\tau >0$ and $\eta = d^2 + |w|$, we have that
\begin{equation}
|\langle \Log, \nu^{w}_N \rangle_{\chi}^{\eta} - \langle \Log, \nu^{w} \rangle_{\chi}^{\eta}| \le C [N^{-\delta} + N^{-\varepsilon} \log N] \log N.
\end{equation}

If we also assume that we are on the event $\Omega_N(w)$, then we can also apply the estimates \eqref{bd:near-zero} and \eqref{eq:tail-nu} to bound
\begin{equation}
|\langle \Log, \nu^{w}_N \rangle_{0}^{\chi} - \langle \Log, \nu^{w} \rangle_{0}^{\chi}| \le |\langle \Log, \nu^{w}_N \rangle_{0}^{\chi}| + |\langle \Log, \nu^{w} \rangle_{0}^{\chi}| | \le C (N^{\kappa - \varepsilon} + N^{-\varepsilon + \tau}) + C N^{-\varepsilon + \tau} \log N.
\end{equation}

Now, since both $\nu^{w}_N(-y,y) = \nu^{w}(-y,y) =1$ for $|y| \ge d^2 + |w|$, we have that
$$
|\langle \Log, \nu^{w}_N \rangle_{\eta}^{\infty} - \langle \Log, \nu^{w} \rangle_{\eta}^{\infty}| = 0
$$

Combining all these estimates, we see that on the event $\Omega_N(w)$, we have that,
\begin{equation}
|\langle \Log, \nu^{w} \rangle - \langle \Log, \nu^{w}_N \rangle| \le C [N^{\kappa - \varepsilon} \log N +  N^{-\varepsilon +\tau} \log N + N^{-\delta } \log N + N^{-\varepsilon} (\log N)^2].
\end{equation}

In what proceeds, we choose $\kappa$ very small, $\tau$ very small, and $2a < \min(\varepsilon - \max[\kappa,\tau], \delta )$. Then, we have that on the event $\Omega_N(w)$ that,
\begin{equation}
|\langle \Log, \nu^{w} \rangle - \langle \Log, \nu^{w}_N \rangle| \le C N^{-2a} (\log N)^{-2}.
\end{equation}

Recall the event  $\Omega_N(w + U_N)$, which is the event that the smallest singular value of the matrix $A_N +  U_N$ is larger than $\delta_N = \exp[-N^{\kappa}]$, we notice that the above estimates can be written as follows.
\begin{align*}
 \big| \langle \Log, \nu_N^{w +  U_N} \rangle - \langle \Log, \nu^{w + U_N} \rangle \big| 
& \le 
 C (\log N)^{-2} N^{-2a}\mathbbm{1}[\Omega_N(w+  U_N)] \\& + \mathbbm{1}[\Omega_N(w+  U_N)^c]\left[ \frac{1}{N} \sum_{i=1}^N |\log| \lambda_i -w-  U_N||  +|\langle \Log, \nu^{w +  U_N} \rangle| \right] \,. 
\end{align*}

By putting this inequality back into Girko's formula and taking expectation, we notice that

\begin{align} \label{eq:someexpan}
&\mathbb{E}\left[\int_\bC \Delta \psi_{N^a} (w) \Gamma_N(w_o+w) d w\right]  \le  N^{2a} \int_{\mathbb{C}} |\triangle \psi(\xi)|  C( \log N)^{-2} N^{-2a} \mathbb{E}[\mathbbm{1}[\Omega_N(w_0 + N^{-a} \xi+ U_N)]] \text{d} \xi\nonumber \\
&+ N^{2a} \int_{\mathbb{C}}|\triangle \psi(\xi)| \mathbb{E}\left[\mathbbm{1}[\Omega_N(w_0 + N^{-a} \xi+  U_N)^c] \left[ \frac{1}{N} \sum_{i=1}^N |\log| \lambda_i -w_0 + N^{-a} \xi-  U_N|| \right] \right] \text{d}\xi \nonumber \\
 &+ N^{2a} \int_{\mathbb{C}}|\triangle \psi(\xi)| \mathbb{E}\left[\mathbbm{1}[\Omega_N(w_0 + N^{-a} \xi+  U_N)^c] \right]|\langle \Log, \nu^{w_0 + N^{-a}\xi +  U_N} \rangle|   \text{d}\xi.
\end{align}

The first line of \eqref{eq:someexpan} is clearly $o(1)$. Furthermore, since pointwise we know that

 $ \mathbb{E}\left[\mathbbm{1}[\Omega_N(w_0 + N^{-a} \xi+ U_N)^c] \right] \le N^{-\mathfrak{c}}$, the third line is also $o(1)$ provided $2a <\mathfrak{c}$, as otherwise all other quantities in the integral in the third line are constant in $N$.

To deal with the second line, we first notice that,
\begin{equation*}
\begin{aligned}
& \mathbb{E}\left[\mathbbm{1}[\Omega_N(w_0 + N^{-a} \xi+  U_N)^c] \left[ \frac{1}{N} \sum_{i=1}^N |\log| \lambda_i -w_0 + N^{-a} \xi- U_N|| \right] \right]\\& \le  \mathbb{E}^{1/2}[\mathbbm{1}[\Omega_N(w_0 + N^{-a} \xi + U_N)^c]] \frac{1}{N} \sum_{i=1}^N \mathbb{E}^{1/2} \left[ \log^2|\lambda_i -w_0 + N^{-a} \xi -  U_N|] \right]
\end{aligned}
\end{equation*}

Notice that, regardless of the value of $\lambda_i$, if we take the expectation over the random variable $U_N$, the value of $\mathbb{E}^{1/2} \left[ \log^2|\lambda_i -w_0 + N^{-a} \xi -  U_N|] \right]$ will be bounded by $ O(\log N)$. Now, $ \mathbb{E}^{1/2}[\mathbbm{1}[\Omega_N(w_0 + N^{-a} \xi + U_N)^c]] < N^{-\mathfrak{c}/2}$. Thus, as long as we set $a < \mathfrak{c}/2$, the integral on the second line of \eqref{eq:someexpan} is also $o(1)$. We find that $\mathbb{E}\left[\int_\bC \Delta \psi_{N^a} (w) \Gamma_N(w_o+w) d w\right] $ is $o(1)$ and must converge to $0$ in probability by Markov's inequality. 

\subsection{The model and Green's function}

We start with a more formal definition of our model.
\begin{defn}[Digraph]
A digraph $\GG$ has vertices $V$ and directed edges $E$, where each oriented 
edges $e=(v,w) \in E$ has an initial vertex $v$ and a terminal vertex $w \ne v$ (in particular, 
the $(v,w)$ and $(w,v)$ are distinct edges). In the adjacency matrix $\sfA_{\GG}$ 
of a digraph $\GG$, the entry $a_{i,j}$ indicates the number of edges that 
are oriented from the vertex $i$ to the vertex $j$
(and such adjacency matrix is, in general, not Hermitian), where a simple digraph has 
only $\{0,1\}$-valued entries.
\end{defn}

\begin{defn}[Uniformly random $d$-regular digraph]
A $d$-regular digraph is a graph such that for each vertex $v \in V$, exactly $d$ edges of the graph 
have $v$ as their initial vertex and exactly $d$ edges have $v$ as their terminal vertex. We let 
$\GG_{N,d}$ denote the (finite) set of all $d$-regular, simple, digraphs on $N$ vertices with 
$\mathbb{P}_{N,d}$ the corresponding uniform distribution on $\GG_{N,d}$ (namely, 
selecting each specific instance from $\GG_{N,d}$ with probability $\frac{1}{|\GG_{N,d}|}$).
\end{defn}

Hereafter we often use lighter notations, with the dimension $N$ and the identity matrix $\sfI_N$ 
implicit, as are sometimes also the arguments $w$ and $z$. Choosing $\GG$
according to $\bP_{N,d}$, \abbr{wlog} we also replace $\sfA_{\GG}$ hereafter by 
the normalized matrix 
$A:=\frac{1}{\sqrt{d-1}} \sfA_{\GG}$ and for $w \in \bC$ and $z \in \bC_+$, let
\begin{equation}\label{def:Hzw}
G(z,w):= H(z,w)^{-1}\,, \qquad 
H(z,w):=\begin{bmatrix}
-z & A -w\\
A^*- \overline{w}  & -z 
\end{bmatrix} \,.
\end{equation}
By definition, $G(z,w)=\sqrt{d-1} \sfG_N(\sqrt{d-1} z,\sqrt{d-1}w)$ and
towards Theorem \ref{thm:trace} our goal is to
control for large $N$, the normalized trace of $G(z,w)$, where of most interest is the 
case of $\Im(z) \downarrow 0$. 

To this end, recall the following expansion of 
the Green's function, at $|z|$ large enough,
\begin{equation}
G(z,w)= -z^{-1}(I - z^{-1}H(0,w))^{-1} = -z^{-1} \sum_{k=0}^{\infty} z^{-k} H(0,w)^k.
\end{equation}
Fixing $w$ and using $H:=H(0,w)$, we further have for any $a,b \le N$ and $k \ge 0$,
$$
[H^{2k}]_{a,b} = \sum_{v_1,v_2,\ldots, v_{2k-1}} \prod_{j=0}^{k-1} H_{v_{2j},v_{2j+1}+N} H_{v_{2j+1}+N,v_{2j+2}} ,
$$
where the sum is over $v_i \le N$, with $v_0=a$ and $v_{2k}=b$. Similarly, 
$$
[H^{2k}]_{a+N,b+N} = \sum_{v_1,v_2,\ldots,v_{2k-1}} \prod_{j=0}^{k-1}H_{v_{2j}+N, v_{2j+1}} H_{v_{2j+1},v_{2j+2}+N},
$$
again summing over $v_i \le N$, with $v_0=a$ and $v_{2k}=b$. Moreover, 
$$
[H^{2k+1}]_{a,b+N} = \sum_{v_1,\ldots,v_{2k}} \Big[\prod_{j=0}^{k-1} H_{v_{2j}, v_{2j+1} +N} H_{v_{2j+1}+N, v_{2j+2}}  \Big] H_{v_{2k},b+N}
$$
summing over $v_i \le N$, with $v_0=a$, whereas $[H^{2k+1}]_{a+N,b}$ is the complex conjugate of 
$[H^{2k+1}]_{a,b+N}$.

Now, in case 
$v_{2j} \ne v_{2j+1}$,
an entry $H_{v_{2j},v_{2j+1}+N}$ represents a forward edge (agreeing with the orientation) from vertex $v_{2j}$ to vertex $v_{2j+1}$ while $H_{v_{2j+1} +N,v_{2j+2} } = H_{v_{2j+2},v_{2j+1}+N}$ represents a traversal through a backward edge (opposing the orientation) from the vertex $v_{2j+1}$ to $v_{2j+2}$. For this reason, we include in a local neighborhood of each vertex $v \in \GG$ also all
those vertices that are reachable through edges traversed in the reverse direction. 
When we have in the expansion for $[H^{2k}]_{a,b}$ some $v_{2j}= v_{2j+1}$, it follows that
$v_{2j-1}$ travels to $v_{2j}$ via a backward edge, stays at $v_{2j}$ with weight factor $-w$ 
(as there are no self-edges in $\GG$), and proceeds from $v_{2j+1}$ to $v_{2j+2}$ 
via a backward edge (once again). Similarly, in case we have in such expansion 
some $v_{2j-1}=v_{2j}$, we enter $v_{2j-1}$ via a forward edge, stay with weight factor
$-\overline{w}$ and leave $v_{2j}$ to $v_{2j+1}$ again via a forward edge. Thus, the inclusion 
of such factors $-w$ or $-\overline{w}$ breaks the otherwise 
regular pattern of alternating between traveling on forward edges and on backward edges.
More generally, we may travel a few forward edges in succession or a few backward edges 
in succession (each consecutive pair of forward edges then contributes the factor
$-\overline{w}$ while each consecutive pair of backward edges contributes $-w$).
This 
solidifies the assertion that a local neighborhood 
should consider traversal via forward and backward edges \emph{in any type of succession},
measuring  the neighborhood's size by the distance in the unoriented graph.

\subsection{The  $d$-regular directed tree and its Green's function.}\label{sec:GreenFunc}

Our proof relies on relating entries $G_{ij}$ of the Green's function for a $d$-regular digraph,  
to the Green's function of the local neighborhood of the graph. Typically, such local neighborhood 
should look like some subset of the infinite directed tree which we describe next.

\begin{defn}[Construction of the $d$-regular $K$-neighborhoods in directed tree]\label{def:K-tree}

First, fix a vertex $\hat{r}$ to be the root. From this root $\hat{r}$, create a $2d$ regular tree of $K$ levels. The root will have $2d$ children while any other vertex that is not a root or a child at the $K$ level will have $2d-1$ children. At this point, we can start orienting the edges. For convenience, assume that we have drawn this tree on the plane so that the root is on the very top and any vertex at distance $\ell$ from the root will be below its parent at the $(\ell-1)$-th level. With this planar drawing, the notion of $j$ leftmost children of a vertex would be a meaningful notion. With these notions, we can now describe the orientation of our edges.

Our process proceeds level by level; first, we orient the edges that are adjacent to the root in the unoriented graph. Afterwards, we proceed to orient the edges that are adjacent to a vertex of distance $1$ from the root in the unoriented graph and so on. At the root $\hat{r}$, we will orient the leftmost $d$ edges away from the root and the rightmost $d$ edges are oriented toward the root. Now assume that we have oriented the edges that are adjacent to a vertex of distances $1$ to $\ell-1$ from the root in the unoriented graph. We now proceed to orient the edges that are adjacent to a vertex at the $\ell$-th level in the unoriented graph that have not been assigned an orientation as of yet. Observe for every vertex $v$ at the $\ell$-th level, there are still $2d-1$ children edges that have not been assigned an orientation and these edges connect the vertex $v$ to a vertex below it.

Depending on the orientation of the parent edge (i.e., the edge already assigned an orientation), there will be two ways to assign the orientation to the remaining children edges. If the parent is oriented away from the vertex in question, then orient the $d-1$ leftmost children edges of the vertex away from the vertex and the remaining $d$ edges towards the vertex. If the parent edge is oriented towards the vertex, then we orient the $d$ leftmost children away from the vertex and the remaining $d-1$ edges towards the vertex.
  
  \end{defn}

\begin{defn}[$d$-regular infinite directed trees]\label{def:digraphtree}
We have three natural notions of $d$-regular directed trees (while only one is truly $d$-regular, 
we abuse the notation to make our descriptions more concise):
\begin{enumerate}[(a)]
 \item The $d$-regular infinite directed tree of Definition \ref{def:K-tree} is denoted 
 by $\TT$. Ignoring orientation, it is a $2d$-regular infinite tree, and every 
 vertex has $d$ outedges and $d$ inedges.

\item The tree $\TT_1$ is such that its unoriented version is a $(2d-1)$-ary tree; namely, the root has $2d-1$ edges while all other vertices have $2d$ edges in the unoriented version. When oriented, 
$\TT_1$  has the property that the root has $d$ outedges and $d-1$ inedges.

\item The tree $\TT_2$ is also such that its unoriented version is a $(2d-1)$-ary tree. The difference
from $\TT_1$ is that under the orientation, here the root vertex has $d-1$ outedges and $d$ inedges.

\end{enumerate}
\end{defn}


\subsubsection{Self-consistent equation for the $d$-regular directed tree}
Our primary quantity of interest is the Green's function $G_{\mathcal T}$.
However, in order to utilize
the Schur complement formula to analyze this function, one must also analyze the Green's
functions of $\TT_1$ and $\TT_2$. Our basic operation in the application of the Schur complement formula is to remove the root vertex and induct on the remaining graph. Indeed, 
removing the root of $\TT$ yields $d$ disjoint copies of
$\TT_1$ and $d$ disjoint copies of $\TT_2$. Similarly, removing the root of $\TT_1$ yields $d$ copies of $\TT_1$ and $d-1$ copies of $\TT_2$, with the reverse holding for 
$\TT_2$. Utilizing these relations allows us to inductively compute entries of the Green's functions 
$G_{\TT}$, $G_{\TT_1}$ and $G_{\TT_2}$, 
but there is one complication we must consider. In our convention each vertex 
$v \in \TT$ corresponds to $2$ rows in the corresponding Hermitized 
matrix $H$, which we shall call $v$ and $v + \aleph$ (just as in an $N$ vertex digraph, such vertex $v$ corresponds to both rows $v$ and $v+N$). Thus, when applying the Schur complement formula
we must remove at the same time two vertices: $v$ and $v+\aleph$.
Specifically, our starting point in deriving Green's function identities is the block decomposition
of a matrix and its inverse,
\begin{equation}
\begin{bmatrix}
 A & B\\
C & D
\end{bmatrix}
\begin{bmatrix}
X & Y \\
Z & W
\end{bmatrix}=
\begin{bmatrix}
I & 0\\
0 & I
\end{bmatrix}.
\end{equation}
resulting with by elementary algebra with the Schur complement formulas
$$
X = (A - B D^{-1} C)^{-1}, \qquad \quad Z= - D^{-1}C X, \qquad \quad
 W = D^{-1} + D^{-1} C X B  D^{-1} \,.
$$
Suppressing the common arguments $z,w$ and setting $G^{(T)}:=(H|_{T^c \times T^c})^{-1}$ for 
the Green's function 
with the rows and columns of $T$ removed, these equations amount to,
\begin{align}\label{eq:firstGreen}
 G|_{T \times T} &= (H|_{T \times T} - H|_{T \times T^c} \, G^{(T)} \, H|_{ T^c \times T})^{-1},\\
\label{eq:secondGreen}
G|_{T^c \times T} &= - G^{(T)}  \, H|_{T^c \times T} \, G|_{T \times T},\\
\label{eq:thirdGreen}
G|_{T^c \times T^c} &= G^{(T)} + G^{(T)}  \, H|_{T^c \times T} \, G|_{T \times T} \, H|_{T \times T^c} G^{(T)} \,.
\end{align}


{Hereafter, similarly to \eqref{def:Hzw}, we first scale the adjacency matrix of $\TT$ by the factor 
$\frac{1}{\sqrt{d-1}}$, so our Green's function is given by 
$G_\TT(z,w):=\sqrt{d-1} \sfG_\TT(\sqrt{d-1} z,\sqrt{d-1}w)$ in terms of the original, non-scaled 
function $\sfG_\TT$.}
Our analysis of $G_{\TT}$ relies on the removal of both indices $\hat{r}$ and $\hat{r} + \aleph$ 
associated with the root vertex $\hat{r}$, to relate 
$\mathcal T$ to $d$ copies of $\TT_1$ and $d$ copies of $\TT_2$  
(we use the same notation
 $\hat{r}, \hat{r} + \aleph$
for the root indices of $\TT_1$ and $\TT_2$,
and it be clear from the context to which tree's root we refer each time). Likewise, 
removing the root of $G_{\TT_1}$ yields $d$ copies of $\TT_1$ and $d-1$ copies of $\TT_2$, where the index $\hat{r}$ of the original $\TT_1$ was connected to the $\hat{r}+\aleph$ index of its children (namely, the $d$ copies of the subtree $\TT_1$), while the $\hat{r}+\aleph$ index of the original $\TT_1$ was connected to the $\hat{r}$ index of its children (namely, the $d-1$ copies of the subtree $\TT_2$). Moreover, having no contribution from $ H|_{T \times T^c} \, G^{(T)}\, H|_{ T^c \times T}$ to the entries indexed by $\hat{r}, \hat{r}+\aleph$ or $\hat{r}+\aleph,\hat{r}$, it thus follows from \eqref{eq:firstGreen}
that
\begin{equation}\label{eq:zeroselfconc}
\begin{bmatrix}
(G_{\TT_1})_{\hat{r},\hat{r}} & (G_{\TT_1})_{\hat{r},\hat{r} + \aleph}\\
(G_{\TT_1})_{\hat{r}+\aleph,\hat{r}}  & (G_{\TT_1})_{\hat{r}+\aleph,\hat{r}+\aleph}
\end{bmatrix}^{-1} = -
\begin{bmatrix}
z  & w \\
\overline{w} & z
\end{bmatrix}
- \begin{bmatrix}
\frac{d}{d-1} (G_{\TT_1})_{\hat{r}+\aleph, \hat{r}+\aleph} & 0\\
0 & (G_{\TT_2})_{\hat{r},\hat{r}}
\end{bmatrix}.
\end{equation}
Similarly, removing the root of $G_{\TT_2}$ results with $d$ subtrees $\TT_2$ and $d-1$ subtrees $\TT_1$. The index $\hat{r}$ of the original $\TT_2$ was connected to the root $\hat{r}+ \aleph$ of the $d-1$ subtrees $\TT_1$, while $\hat{r}+\aleph$ was connected to the root $\hat{r}$ of 
the $d$ subtrees $\TT_2$, so by Schur complement formula,
\begin{equation} \label{eq:firstselfconc}
\begin{bmatrix}
(G_{\TT_2})_{\hat{r},\hat{r}} & (G_{\TT_2})_{\hat{r},\hat{r} + \aleph}\\
(G_{\TT_2})_{\hat{r}+\aleph,\hat{r}}  & (G_{\TT_2})_{\hat{r}+\aleph,\hat{r}+\aleph}
\end{bmatrix}^{-1} =
- \begin{bmatrix}
z  & w \\
\overline{w} & z
\end{bmatrix}
- \begin{bmatrix}
 (G_{\TT_1})_{\hat{r}+\aleph, \hat{r}+\aleph} & 0\\
0 &  \frac{d}{d-1}(G_{\TT_2})_{\hat{r},\hat{r}}
\end{bmatrix}.
\end{equation}
Since, as mentioned earlier, 
\begin{equation} \label{eq:secondselfconc}
\begin{bmatrix}
(G_{\TT})_{\hat{r},\hat{r}} & (G_{\TT})_{\hat{r},\hat{r} + \aleph}\\
(G_{\TT})_{\hat{r}+\aleph,\hat{r}}  & (G_{\TT})_{\hat{r}+\aleph,\hat{r}+\aleph}
\end{bmatrix}^{-1} =
- \begin{bmatrix}
z  & w \\
\overline{w} & z
\end{bmatrix}
- \begin{bmatrix}
 \frac{d}{d-1}(G_{\TT_1})_{\hat{r}+\aleph, \hat{r}+\aleph} & 0\\
0 &  \frac{d}{d-1}(G_{\TT_2})_{\hat{r},\hat{r}}
\end{bmatrix},
\end{equation}
we see that all Green's function entries are derived from 
$(G_{\TT_1})_{\hat{r}+\aleph,\hat{r}+\aleph}$ and $(G_{\TT_2})_{\hat{r},\hat{r}}$. 
Further, the investigation of the path expansion of the Green's function at large $|z|$ suggests
that $(G_{\TT_1})_{\hat{r}+\aleph,\hat{r}+\aleph} = (G_{\TT_2})_{\hat{r},\hat{r}}$. Proceeding under this 
ansatz and denoting the latter quantity by $m_{\infty}$, we deduce from \eqref{eq:zeroselfconc} by 
 the inversion formula for $2 \times 2$ matrices, that 
\begin{equation}\label{eq:self-conc}
m_{\infty} = \frac{z + \frac{d}{d-1} m_{\infty}}{|w|^2-\big( z+ \frac{d}{d-1} m_{\infty} \big)(z+m_{\infty})}.
\end{equation}
Fixing $w \in \mathbb{C}$, by the implicit function theorem in terms of $(m_\infty,z^{-1})$ at $(0,0)$, 
there exists a unique analytic solution $m_\infty(\cdot,w):\mathbb{C}_+ \mapsto \mathbb{C}_+$ of
the cubic equation \eqref{eq:self-conc}, 
such that $z \, m_\infty (z,w) \to -1$ when $|z| \to \infty$. Using hereafter 
this solution of \eqref{eq:self-conc}, and suppressing the arguments $z,w$, we get from 
the central parameter $m_\infty$ also the quantities
%
\begin{equation}\label{eq:m-infty-rel}
\begin{aligned}
& m_{\infty}^{sd} := \frac{z + m_{\infty}}{z + \frac{d}{d-1}m_{\infty}}m_{\infty},
& \qquad m_{\TT}^{d}:= \frac{z + \frac{d}{d-1}m_{\infty}}{|w|^2-\big(z + \frac{d}{d-1}m_{\infty} \big)^2}, \\
& m_{\infty}^{uod}:= \frac{-w}{z+ \frac{d}{d-1} m_{\infty}}m_{\infty},
& m_{\TT}^{uod}:= \frac{-w}{z+ \frac{d}{d-1} m_{\infty}}m_{\TT}^d,
\\
& m_{\infty}^{lod} := \frac{-\overline{w}}{z+ \frac{d}{d-1} m_{\infty}} m_{\infty},
& m_{\TT}^{lod}:= \frac{-\overline{w}}{z+ \frac{d}{d-1} m_{\infty}}m_{\TT}^d,
\end{aligned}
\end{equation}
where $m_\infty^{sd} := (G_{\TT_1})_{\hat{r},\hat{r}}$, 
$m_{\infty}^{uod} :=(G_{\TT_1})_{\hat{r},\hat{r}+\aleph}$ and 
$m_{\infty}^{lod} :=(G_{\TT_1})_{\hat{r}+\aleph,\hat{r}}$ are the remaining entries of $G_{{\mathcal T}_1}$
(see the \abbr{lhs} of \eqref{eq:zeroselfconc}). Upon exchanging $\hat{r}$ with $\hat{r}+\aleph$, these are also 
the entries of $G_{{\mathcal T}_2}$ (see \eqref{eq:firstselfconc}), whereas
$m_{\TT}^d := (G_{\TT})_{\hat{r},\hat{r}} =(G_{\TT})_{\hat{r}+ \aleph,\hat{r}+\aleph}$,
$m_{\TT}^{uod}:= (G_{\TT})_{\hat{r},\hat{r}+\aleph}$ and
$m_{\TT}^{lod}:=(G_{\TT})_{\hat{r}+\aleph,\hat{r}}$ (see \eqref{eq:secondselfconc}).

{
In particular, it is not hard to verify that 
\begin{equation}\label{eq:scaling}
m_\TT^d(z,w)=\sqrt{d-1} m_\star(\sqrt{d-1}z,\sqrt{d-1}w) \,.
\end{equation}} For details, one can refer to lemma \ref{lem:compareStieltjes}.

\subsection{Extension of graphs and main result}

A key ingredient in the asymptotic evaluation of the Green's function for our digraphs, is the
following concept of graph extension (or more precisely, the extension of the
associated weighted adjacency matrix).
\begin{defn}[Extensions of graphs] \label{def:GraphExt}
Suppose $\GG$ is a digraph of vertex set $V$, where each $v \in V$ has 
in-degree $d_{in} (v) \le d$ and out-degree $d_o(v) \le d$. Any given 
deficit functions $0 \le def_{in}(v) \le d- d_{in}(v)$, $0 \le def_{o}(v) \le d - d_o(v)$,
and two functions $\Delta_1$ and $\Delta_2$ of $z,w$, induce the following 
extension of the matrix $H$ which corresponds to $\GG$ via \eqref{def:Hzw} :
\begin{equation} \label{eq:ExtensionGraph}
\begin{aligned}
 \text{Ext}_{\text{def}} & (\GG,\Delta_1,\Delta_2) = \\
&  H 
- \sum_{v \in V} \Big[ \frac{d- d_{o}(v) -  \text{def}_{o}(v)}{d-1} \Delta_1 e_{v,v} + \frac{d-d_{in}(v) - def_{in}(v)}{d-1} \Delta_2 e_{v+\aleph,v+\aleph} \Big]  \,,
\end{aligned}
\end{equation}
where $e_{v,v}$ and $e_{v+\aleph,v+\aleph}$ denote, respectively, the $(v,v)$ and  
$(v+\aleph,v+\aleph)$ coordinates of the matrix
{(with isolated vertices in $V$ also considered in the sum).} The Green's function 
$G(\text{Ext}_{\text{def}}  (\GG,\Delta_1,\Delta_2))$ associated with such 
extension is merely the inverse of the \abbr{rhs} of \eqref{eq:ExtensionGraph}.
In particular, we use the notation $Ext$ for a graph extension where both 
$def_{in}$ and $def_{o}$ are identically zero. Similarly,  
$\text{Ext}_i$ denotes the case where a specific vertex $\hat{r}$ (the root) has $def_{in}(\hat{r}) =1$, while
all other values of $def_{o}$ and $def_{in}$ are zero, 
whereas in $\text{Ext}_o$ one has that $def_{o}(\hat{r}) =1$, while all other values 
of the deficit functions are zero.  That is,
\begin{equation} \label{eq:exti}
\begin{aligned}
\text{Ext}_i(\GG,\Delta_1,\Delta_2)= 
& H 
- \sum_{v \in V} \left[ \frac{d- d_{o}(v)}{d-1} \Delta_1 e_{v,v} + \frac{d-d_{in}(v) - \delta_{v,\hat{r}}}{d-1} \Delta_2 e_{v+\aleph,v+\aleph} \right] .
\end{aligned}
\end{equation}
\begin{equation} \label{eq:exto}
\begin{aligned}
&\text{Ext}_o(\GG,\Delta_1,\Delta_2) = 
H 
- \sum_{v \in V} \left[ \frac{d- d_{o}(v) -\delta_{v,\hat{r}}}{d-1} \Delta_1 e_{v,v} + \frac{d-d_{in}(v)}{d-1} \Delta_2 e_{v+\aleph,v+\aleph} \right]  .
\end{aligned}
\end{equation}
\end{defn}

We denote by $B_\ell(v)$ the unoriented ball in digraph $\GG$ of radius $\ell$ centered at vertex $v$,
with $\partial B_\ell(v)$ denoting the set of vertices of $\GG$ having an undirected 
distance $\ell$ from $v$. More generally, for a subset $U$ of vertices of $\GG$, 
we denote by $B_\ell(U)$ the union of such balls $B_\ell(v)$ for $v \in U$.
With these quantities at hand, our main result, stated next, is a refinement of Theorem 
\ref{thm:trace}, providing also a uniform over $(i,j)$, control on the entries $G_{i,j}$ of $G(z,w)$.
\begin{thm} \label{thm:mainthm}

Fix a constant  $0 < \mathfrak{c} < 1$ and $d \ge 3$. Given these terms, we can find parameters 
$\delta=\delta(\mathfrak{c},d)>0$,  
$\varepsilon =\varepsilon(\mathfrak{c},d)>0$ 
and  $\mathfrak{r}= \mathfrak{r}(\mathfrak{c},d)$  such that for 
$r = r_N=\mathfrak{r} \log_{d-1} N$,
some set of $d$-regular random digraphs $\GG$ 
occurring with probability $\mathbb{P}_{N,d}$ at least $1- O(N^{-\mathfrak{c}})$ 
and any compact $\mathbb{D} \subset \mathbb{C}$, 
\begin{align}
\sup_{w \in \mathbb{D}, \Im(z) \ge N^{-\varepsilon}}
\max_{ij} |G_{i,j} - G_{i,j}(Ext(B_{r}(\{i,j\}),m_{\infty}, m_{\infty}))| & \le N^{-\delta}
,\\
\sup_{w \in \mathbb{D}, \Im(z) \ge N^{-\varepsilon}}
\Big| \frac{1}{N} \sum_{i=1}^N G_{i,i} - m_{\TT}^d \Big|  & \le N^{-\delta}.
\label{eq:N-delta}
\end{align}
\end{thm}


Theorem \ref{thm:mainthm} will be shortly deduced from our next two results, the first of which 
is a perturbation theory on infinite $d$-regular digraphs.
\begin{thm} \label{thm:pert}
 Let $\GG$ be a graph with excess of at most $1$. Assume further that $\text{diam}(\GG)^3[|Q_I - m_{\infty}| + |Q_O - m_{\infty}|]\ll 1$. Then, we have,
\begin{equation} \label{eq:carefulpert}
\begin{aligned}
&|G_{ij}(Ext(\GG,Q_I,Q_O)) - G_{ij}(Ext(\GG,m_{\infty},m_{\infty}))| \\ &\lesssim \left(\frac{1}{d-1}\right)^{\text{dist}(i,j)}(1+ \text{dist}(i,j)) [|Q_I - m_{\infty}| + |Q_O - m_{\infty}|]
\end{aligned}
\end{equation}
\end{thm}

We will elaborate on the meanings of $Q_I$ and $Q_O$ in the next section. For now, just understand these terms as a tool that help mediate the derivation of an appropriate self-consistent equation. Our next result 
states the relationship between $G_{ij}$ and the $d$-regular infinite tree extension, except with parameters $Q_I$ and $Q_O$. The specific definitions of the error terms ($\epsilon(z,w)$ and $S_g^1$), can also be found in the next section.
\begin{thm} \label{lem:auxmain}
Fix $d \ge 3$, $0 < \mathfrak{c} < 1$ and a compact domain $\bD \subset \bC$. There exist
$\varepsilon=\varepsilon(\mathfrak{c},d) >0$ and $r=r_N=\mathfrak{r} {\log_{d-1} N}$,
such that on a set of $d$-regular, random digraphs 
occurring with probability $1- O(N^{-\mathfrak{c}})$, all $w \in \bD$ and $\Im(z) \ge N^{-\varepsilon}$ 
we have
\begin{equation} \label{eq:auxcollect}
\begin{aligned}
&|Q_{I/O} - m_{\infty}(z,w)| \le \frac{\epsilon(z,w)}{S_g^1},\\
&\max_{i,j} |G_{ij} - G_{ij}(Ext(B_{\ell}(\{i,j\},\GG),Q_I, Q_O))| \le \epsilon(z,w) .
\end{aligned}
\end{equation}
(See Part 2 of Section \ref{sec:sec2} for a formal definition of $\epsilon(z,w)$, which are independent of 
$\GG$. We remark also that $S^1_g$ of \eqref{eq:singularities}
is a deterministic constant depending only on $z$ and $w$.) 
\end{thm}

\begin{proof}[Proof of Theorem \ref{thm:mainthm} (from Theorems \ref{thm:pert} and \ref{lem:auxmain})]
With $|Q_{I/O} - m_{\infty}(z,w)| \le \frac{\epsilon(z,w)}{S_g^1}$, we can apply the perturbation estimate from Lemma \ref{lem:pert} to ensure that for $r=r_N=\mathfrak{r} \log N$,
$$
|G_{ij}(Ext(B_{r}(\{i,j\},\GG),m_{\infty}(z,w), m_{\infty}(z,w)))- G_{ij}(Ext(B_{r}(\{i,j\},\GG),Q_I, Q_O))| \lesssim (\log N)  \frac{\epsilon}{S_g^1}.
$$
This combined with the second estimate in \eqref{eq:auxcollect}, ensures that,
\begin{equation}
|G_{ij} - G_{ij}(Ext(B_{r}(\{i,j\},\GG),m_{\infty}(z,w), m_{\infty}(z,w)))| \lesssim \epsilon + (\log N) \frac{\epsilon}{S_g^1}.
\end{equation}

Now, if $i$ has a radius $\mathfrak{R}$ tree-like neighborhood, we would know that
$$G_{ii}(Ext(B_{r}(\{i,j\},\GG),m_{\infty}(z,w), m_{\infty}(z,w)))  = m^{d}_{\TT}.$$
There are also at most $O(N^{\mathfrak{c}})$ vertices without a radius $\mathfrak{R}$ tree-like neighborhood. For these vertices, we can instead use the bound that $|G_{ii}| = O(1)$. 

This again comes from Lemma \ref{lem:pert} to bound $G_{ij}(Ext(B_{r}(\{i,j\},\GG),Q_I, Q_O))$ by $O(1)$ and the triangle inequality.  This ensures that,
\begin{equation}
\Big| \frac{1}{N} \sum_{i=1}^N G_{ii} - m^d_{\TT}\Big| \lesssim N^{\mathfrak{c}-1} + \epsilon + (\log N) 
\frac{\epsilon}{S_g^1}.
\end{equation}
Finally,  recall \eqref{eq:Sg-ident} that here $S_g^1 \ge \frac{1}{2} N^{-\varepsilon}$ and 
\eqref{def:eps} that $\epsilon(z,w) \le N^{-\delta'}$ for any $\delta'< \min(\mathfrak{r},\frac{1-\varepsilon}{2})$.
We thus deduce the stated uniform bound of order $N^{-\delta}$,  upon choosing 
$\varepsilon(\mathfrak{c},\mathfrak{r})>0$ small enough.
\end{proof}

\subsection*{Acknowledgment}
This project was supported in part by NSF grants DMS-2102842 (A.A.) and DMS-2348142 (A.D).
We thank Anirban Basak, Charles Bordenave,  Nicholas Cook,  Konstantin Tikhomirov and 
Ofer Zeitouni for their most helpful comments on an earlier version of the paper.

\section{Challenges and outline for proving
Theorem \ref{lem:auxmain}} \label{sec:sec2}

{
Sharp estimates on the Green's function are key to many works in random matrix theory.  A standard technique to establish those is by first deriving a self-consistent equation for this Green's function,  which under suitable stability properties,  
allows for a multi-scale argument transferring the estimates on the Green's function 
from one value of $\Im[z]$ to a lower value of $\Im[z]$ (see \cite{EY17}). 

This strategy is used in \cite{HuangYau},  where 
the control parameter $Q$ of \cite[Eqn. (1.6)]{HuangYau},  which represents the 
averaged Green's function of the graph with one edge removed,  satisfies 
the self-consistent equation of  \cite[Eqn.  (1.7)]{HuangYau}.  It roughly says that one can 
determine the Green's function by treating a local neighborhood of large size $K$ around each vertex 
as part of an infinite tree.   We proceed to relate our work with \cite{HuangYau} describing also 
the new challenges that appear in the context of our analysis.  First,  note that to deal with digraphs,  we must 
take into account whether an in-edge or an out-edge has been removed.  To this end,  we associate two vertices 
$v$ and $v+N$ to every vertex $v$ in the original graph,  resulting with a pair of control parameters
 \begin{equation}\label{eq:QIO}
\begin{aligned}
&Q_I := \frac{1}{Nd} \sum_{(x\to y) \in E } G_{y+N,y+N}^{(x,x+N)}\,, \\
& Q_O := \frac{1}{Nd} \sum_{(y \to x) \in E} G_{y,y}^{(x,x+N)}\,.
\end{aligned}
\end{equation}
Further,  our basic removal operation via the Schur-complement formula,  is now the removal of a two-dimensional 
submatrix rather than a single entry,  resulting with a pair of self-consistent equations
\begin{equation}\label{eq:fixed-pt}
\begin{aligned}
& Q_I = Y_{i,K}(Q_I,Q_O),\\
& Q_O = Y_{o,K}(Q_I,Q_O),
\end{aligned}
\end{equation}
representing the treatment of a neighborhood of size $K$ of our digraph as part of the $d$-regular directed tree.

The general outline for deriving \eqref{eq:fixed-pt} is as in \cite{HuangYau}.  Namely,  apply the Schur complement formula to describe the effects of removing a large neighborhood of a vertex,  then perform a switching transformation and finally add back the removed neighborhood.   However,  the intermediate estimates required for implementing this strategy differ drastically between our digraph case and the non-oriented case.  In particular,  the task of proving many estimates in \cite{HuangYau} is much simplified by having the Green's function estimates depend only on the graph distance 
between vertices.  For example,  utilizing the identity \eqref{eq:SchurcompOutsideTree},  namely 
\begin{equation} 
\begin{aligned}
G^{(\mt)} - P^{(\mt)}& = G-P - (G-P) (G|_\mt)^{-1}G
- P[(G|_\mt)^{-1} - (P|_\mt)^{-1}] G - P (P|_\mt)^{-1} (G-P),
\end{aligned}
\end{equation}
one aims to show that the \abbr{lhs} is small by relying on a-priori control on the size of the terms $P_{i,j}$ on its \abbr{rhs}.
Here $P_{i,j}$ roughly corresponds to the Green's function of vertices $i$ and $j$ in the infinite tree,  which are
of distance $O(\log \log N)$.  Having removed a large neighborhood,  one is to sum over $O((\log N)^p)$ such terms
and it is crucial to have a contractive effect of the total contribution,  so the total error on the \abbr{rhs} will
not get exponentially worse as we remove larger and larger pieces (when $N$ grows).  

In the digraph case,  when bounding $P_{i,j}$ the contribution of each path between the two vertices
depends on the orientations of the various edges along that path.  Applying the worst-case value over
such orientations yield bounds which grow exponentially in the size of the neighborhood removed
and thereby to a complete failure of the general strategy.  To address this challenge and ensure that the 
perturbation estimates on the Green's function are contractive,  we leverage on having certain Green's function 
terms which are better than the worst case bound and carefully enumerate over all types of factors that may appear.  
Specifically,  see Lemma \ref{lem:minftybnd},  where the contractive factor is $X+Y<1$ but a less careful 
analysis of the perturbation estimates would lead to an expansive factor of the form $2X>1$.  
For this reason,  we also can not apply the intermediate lemmas of \cite{HuangYau} which rely on the 
simplification offered by the Green's function structure of the infinite unoriented regular tree (that we lack here).

Note also another major difficulty in our self-consistent equations preventing us from 
obtaining estimates on the Green's function reaching the optimal $\Im[z] \asymp N^{-1+\epsilon }$.
Indeed,  the linearization \eqref{eq:linearizedselfconsist} of our self-consistent equations is 
very singular as $\Im[z]$ gets close to this critical value.  This problem is circumvented in 
\cite{HuangYau} by a second order expansion of their equation for $Q$ and explicitly solving for the 
solutions of the resulting quadratic equation.  Applying such approach for our \eqref{eq:fixed-pt}
yields a pair of quadratic equations,  reduced to a single quartic equation,  
whose solutions are not analytically tractable.  As already shown,  our primary goal of 
estimating the Green's function well enough for deriving the global limit law of $d$-regular digraphs
does not require descending to the critical threshold of $\Im[z]$,  and we thus settle for the 
linearized self-consistent equations. 
}

\subsection{Part 1: Heuristics of the switching argument}
We proceed to provide the heuristics behind our derivation of the self-consistent equations 
for the Green's function of the random $d$-regular digraph.  To this end,  recall 
the important control quantities { $Q_I$ and $Q_O$
of \eqref{eq:QIO}}
(and as before, we suppress the arguments $z,w$ whenever their specific value is irrelevant).
Such quantities  
deal with the effect at of removing an edge on the Green's function of the digraph. Specifically, 
$G_{y+N,y+N}^{(x,x+N)}$ 
measures the effect at $y$ of
removing the incoming edge $x \to y$,  
and $G_{y,y}^{(x,x+N)}$ likewise measures the effect at $y$ of removing the outgoing edge
$y \to x$. Thus, our control quantities $Q_I$ and $Q_O$ measure the average effect on 
the Green's function due to the removal of a uniformly chosen incoming edge, 
or outgoing edge, respectively.
As outlined next, our proof strategy consists of the following two steps:
\begin{enumerate}[(a)]
\item Our first step relates the Green's function of $\GG$ around a vertex $v$ to the 
Green's function of the  $\ell$(= $\ell_N$) neighborhood $\mathcal{N}_\ell(v)$ of $v$ for $\ell= O(\log \log N)$.
\item Performing a switching of the edges connecting $\mathcal{N}_\ell(v)$ to $\mathcal{N}_\ell^c(v)$
averages out the effects of specific edges, leading to $Q_I$ and $Q_O$. We then show that
the latter functions concentrate for large $N$ around the non-random $m_\infty$.
\end{enumerate}

\subsubsection{Reduction to the local $K$ neighborhood}

Ignoring for the time being the specific value of $\ell$, consider the relation
via the Schur complement formula between entries of the Green's function of 
$\GG$ within $\mathcal{N}_K(v)$ and the matrix $H$ restricted to this local neighborhood.
Specifically, denote by $\tilde{\mathcal{N}}_K(v)$ the collection of indices in $\mathcal{N}_K(v)$ 
and $\mathcal{N}_K(v) +N$, with $O_{\mathcal{N}_K(v)}$ denoting the collection 
of directed edges that start in $\mathcal{N}_K(v)$ and terminate in $\mathcal{N}^c_K(v)$, while
$I_{\mathcal{N}_K(v)}$ denotes the directed edges that start in $\mathcal{N}^c_K(v)$ and 
terminate in $\mathcal{N}_K(v)$. Considering 
\eqref{eq:firstGreen} 
when $T=\tilde{\mathcal{N}}_K(v)$, yields that for any pair of indices $c,d \in \tilde{\mathcal{N}}_K(v)$
\begin{equation} \label{eq:Knbd}
G_{c,d}= \Bigg[\Big( H|_{\tilde{\mathcal{N}}_K(v)} - \frac{1}{d-1} \bigg[\sum_{(b \to a) \in I_{\mathcal{N}_K(v)}}  
\!\!\!\!\! G^{(\tilde{\mathcal{N}}_K(v))}_{b,b} e_{a+N,a+N} + \!\!\!\!\! \sum_{(a \to b) \in O_{\mathcal{N}_K(v)} }  
\!\!\!\!\!\! G^{(\tilde{\mathcal{N}}_K(v))}_{b+N,b+N} \, e_{aa} \bigg] \Big)^{-1}\Bigg]_{cd} \,.
\end{equation}
We now proceed with the following approximations of the two sums on the \abbr{rhs}. First one would 
expect that since our $d$-regular digraph is locally tree-like around $v$ (when $K$ does not 
grow fast with $N$), the value of $G^{(\tilde{\mathcal{N}}_K(v))}_{b,b}$ should only depend on 
the local behavior around $b$ of the restricted 
graph $\GG|{\mathcal{N}^c_K(v)}$. Moreover, by the same reasoning, 
locally around $b$, the digraph $\GG|{\mathcal{N}^c_K(v)}$ should 
be similar to the original graph $\GG$ with only the vertex $a$ removed. 
We thus expect that, at least for most edges involved in the first sum of \eqref{eq:Knbd},
$$
G^{(\tilde{\mathcal{N}}_K(v))}_{b,b} \approx G^{(a,a+N)}_{b,b} \,.
$$
After such replacement, assuming $K$ is large enough,
we may expect to be able to do even further averaging, thereby 
replacing each $G^{(a,a+N)}_{bb}$ by the total average of terms of this form along all $N d$ 
directed edges of our $d$-regular digraph. Namely, fixing $a \in \mathcal{N}_K(v)$,
$$
\sum_{\{b : (b \to a) \in I_{\mathcal{N}_K(v)}\}} G^{(a,a+N)}_{b,b} \to (d-d_{i}(a)) \frac{1}{Nd} 
\sum_{(y \to x) \in E} G^{(x,x+N)}_{y,y} = (d-d_i(a)) Q_O,
$$
where $d_{i}(a)$ denotes the number of in-edges of $a$ within ${\GG}|\mathcal{N}_K(v)$.
By the same reasoning, one can also expect to have per fixed $a \in \mathcal{N}_K(v)$,
$$
\sum_{\{b:(a \to b) \in O_{\mathcal{N}_K(v)}\}} G^{(a,a+N)}_{b+N,b+N} \to (d-d_o(a)) 
\frac{1}{Nd} \sum_{(x \to y) \in E} G^{(x,x+N)}_{y+N,y+N} = (d-d_o(a)) Q_I\,.
$$
Taken together, this would posit an almost equality 
$$
G_{c,d} \approx G(Ext(\GG|\mathcal{N}_K(v), Q_I,Q_O))_{c,d}
$$ 
for any pair of indices $c,d \in \tilde{\mathcal{N}}_K(v)$, thereby 
`reducing' the evaluation of the Green's function to the study of the two  
order parameters $Q_I$ and $Q_O$ of \eqref{eq:QIO}.

Turning to the latter task, we proceed to derive an equation for the terms $G^{(x,x+N)}_{y,y}$ by
applying a similar reduction to the local $K$ neighborhood of $y$, now for the graph $\GG$ 
after the removal of
$(y  \to x)$. Specifically, let $\mathcal{N}^{(y \to x)}_K(y)$ denote the $K$ neighborhood of the digraph, 
rooted at $y$, after removing the edge $(y \to x)$. Applying the same heuristics as done earlier, we arrive at
\begin{equation}\label{eq:G-Exto}
 G^{(x,x+N)}_{y,y} \approx G(Ext_{o}(\GG|\mathcal{N}^{(y \to x)}_K(y), Q_I,Q_O))_{y,y}
\end{equation}
(where the removal of $(y \to x)$ yields the deficiency $def_o(y)=1$ at the root). Similarly, one would 
propose that
\begin{equation}\label{eq:G-Exti}
G^{(x,x+N)}_{y+N,y+N} \approx G(Ext_{i}(\GG|\mathcal{N}^{(x \to y)}_K(y),Q_I,Q_O))_{y+N,y+N} \,.
\end{equation}

\subsubsection{Tree-like neighborhoods and self-consistency}

For any  $K \ge 1$, let $\TT^K_1$ denote the $(2d-1)$-ary directed tree of $K$ levels, 
whose root vertex has $d$ out-edges and $d-1$ in-edges. Similarly, let $\TT^K_2$ denote
the $(2d-1)$-ary directed tree of $K$ levels, whose root vertex has $d-1$ out-edges and $d$ in-edges.
Define for functions $\Delta_1(z,w)$, $\Delta_2(z,w)$ on $\mathbb{C}_+ \times \mathbb{C}$ the functions 
\begin{equation}\label{dfn:YioK}
\begin{aligned}
Y_{i,K}(\Delta_1,\Delta_2) &:= G(Ext_i(\TT^K_1),\Delta_1,\Delta_2)_{r+\aleph,r+\aleph},\\
Y_{o,K}(\Delta_1,\Delta_2)&:= G(Ext_o(\TT^K_2),\Delta_1,\Delta_2)_{r,r},
\end{aligned}
\end{equation}
where $r$ and $r+\aleph$ are the two indices corresponding to the root of the relevant tree.

To determine the asymptotic of $Q_I$ and $Q_O$ we average the approximations \eqref{eq:G-Exto} 
and \eqref{eq:G-Exti} over all directed edges of our digraph. While doing so, we note that 
most neighborhoods $\mathcal{N}^{(x \to y)}_K(y)$ should look like the $(2d-1)$-ary directed 
tree $\TT^K_1$ of $K$ levels, whose root vertex has $d$ out-edges and $d-1$ in-edges.
Similarly, most neighborhoods $\mathcal{N}^{(y \to x)}_K(y)$ should look like the $(2d-1)$-ary directed
tree $\TT^K_2$ of $K$ levels, whose root vertex has $d-1$ out-edges and $d$ in-edges. Thus, 
after averaging, we expect $Q_I$ and $Q_O$ to approximately satisfy the following self-consistency 
equations 
\begin{equation}
\begin{aligned}
& Q_I = Y_{i,K}(Q_I,Q_O),\\
& Q_O = Y_{o,K}(Q_I,Q_O).
\end{aligned}
\end{equation}
Further, by definition the pair $(Q_I,Q_O)$ should not depend on $K \ge 1$, and for any $w \in \mathbb{C}$,
either function should be analytic from $\mathbb{C}_+$ to $\mathbb{C}_+$ and such that 
$z Q_I(z,w) \to -1$ and $z Q_O(z,w) \to -1$ when $|z| \to \infty$. In view of our 
expectation that $Q_I \approx Q_O$, we verify that the unique solution $Q_I=Q_O$
of \eqref{eq:fixed-pt} with such analytic properties is $m_{\infty}(z,w)$ of \eqref{eq:self-conc}.

\subsection{Part 2: 
The structure of the proof of Theorem \ref{lem:auxmain}}

The proof of Theorem \ref{lem:auxmain} is based on careful and technical estimates of Green's functions after performing a switching. In order to even discuss our proof, it is first necessary to devote some time to defining appropriate notation and the correct order parameters. The first two sub-subsections here are devoted to introducing the appropriate notation. After defining all necessary notions, we then state our important results on Green's function estimates for graphs after switching. These are Theorems  \ref{thm:firstswitch} and \ref{thm:secondswitching}. After some computations, we will show that Theorem \ref{lem:auxmain} follows from these substeps. 


\subsubsection{Notation related to Graph Structure}

The goal of this section is to introduce the notation that we will use later as well as give a broad overview of the strategy of the proof.  The sections that will appear later will go over many of the detailed estimates.

We will start with definitions related to the graph.
\begin{defn}[Ball of Radius $R$]
Let $\GG=(V,E)$ be a digraph and let $S$ be some set of vertices in $V$; here, we will let $\GG^u$ denote the unoriented version of $\GG$.We let $B_R(S, \GG)$ denote the set of vertices that are of distance $R$ from the vertices $V$ in the unoriented graph $\GG^u$.  If $S$ consists of a single element $v$, then we will use $B_R(v, \GG)$ as a shorthand for $B_R(\{v\}, \GG)$; $B_{R}(S,\GG)$ will also be used to denote the subgraph induced by the vertices in the set; which version is used will be clear in context.  In what follows later in the paper, we will always use distance between vertices to refer to the distance between vertices in the unoriented graph $\GG^u$. As this is the only reason to refer to the unoriented version $\GG^u$; we will not use this notation later and let the reader understand that all distances refer to the unoriented distance.

We will say that the vertex $v$ has a radius $R$ tree-like neighborhood in $\GG$ is the unoriented structure of the subgraph $B_R(v,\GG)$ is a tree. If this neighborhood is not tree-like, then we can define the excess of the neighborhood as the number of edges that need to be removed from the subgraph $B_R(v,\GG)$ that need to be removed in order to make the neighborhood tree-like.


\end{defn}

Our method also requires the introduction of many scales of order parameters.

\begin{defn}[Order Parameters]\label{def:ordparam}
\begin{enumerate}
\item We define the value $\mathfrak{R}:= \frac{\mathfrak{c}}{4} \log_{d-1} N$. This value is chosen so that nearly all vertices will have a radius $\mathfrak{R}$ tree-like neighborhood. We will show that with high probability, almost all vertices have a radius $\mathfrak{R}$ tree-like neighborhood. 
\item We introduce the value $r$ which satisfies the inequalities $\frac{\mathfrak{R}}{16} \le r \le \frac{\mathfrak{R}}{8}$. Alternatively, $r:= \mathfrak{r} \log_{d-1} N$ with $\frac{\mathfrak{c}}{64} \le \mathfrak{r} \le \frac{\mathfrak{c}}{32}$. The purpose of this parameter $r$ is to ensure that the Green's functions of two graphs that agree on a neighborhood of size $r$ will be very close to one another. This is the statement of Lemma \ref{lem:nbdpert}. 

\item Finally, we have the value $\ell\in[ \mathfrak{a} \log_{d-1} \log N,2 \mathfrak{a} \log_{d-1} \log N]$, with $\mathfrak{a}\ge 12$.  The specific value of $\ell$ is chosen so that the stability estimates in the self-consistent equation from Lemma  \ref{thm:Self-conc} hold.  $\ell$ represents the size of the neighborhood on which we perform a switching argument. 

\end{enumerate}

We have the following order of scales on our constants:  $\mathfrak{R}> r \gg l$. 

\end{defn}

Indeed with our order parameters, we can make the following definition of good $d$-regular digraphs.
\begin{defn}[Good $d$-regular digraphs] \label{def:Goodregular}

We define the set $\Omega$ (the set of good $d$-regular digraphs) to be the set of $d$-regular digraphs on $N$ vertices that satisfy the following two properties.
\begin{enumerate}
\item All radius $\mathfrak{R}$ neighborhoods have excess at most $1$.
\item All except $N^{\mathfrak{c}}$ vertices have neighborhoods of size $\mathfrak{R}$ that are tree-like.
\end{enumerate}

Fix some parameter $\mathfrak{q}>0$. This constant $\mathfrak{q}$ dictates the probability that a `good' switching event occurs with probability $1- O(N^{-\mathfrak{q}})$. Depending on this parameter $\mathfrak{q}$, we can define the set $\Omega_{S,\mathfrak{q}}$ of the graphs that have sufficiently nice structural properties after switching. A graph belongs in $\Omega_{S,\mathfrak{q}}$ if it satisfies the following two properties,
\begin{enumerate}
\item All except $2 N^{\mathfrak{c}}$ vertices have neighborhoods of size $\mathfrak{R}$ that are tree-like.
\item The $\mathfrak{R}/8$ neighborhood of any vertex has excess at most 1 and the radius $\mathfrak{R}/2$ neighborhood of any vertex has excess at most $C_{\mathfrak{q}}$. $C_{\mathfrak{q}}$ is a constant depending on $\mathfrak{q}$ but not on $N$.
\end{enumerate}

\end{defn}

\subsubsection{Notation related to the Switching Procedure}

With our order parameters defined, we can now define our set of switching data.
In order to avoid an overly long definition, we split the discussion of the switching procedure into two parts. The first part will introduce useful notation, while the second part will actually describe the procedure used in the switching.

\begin{defn} [Switching Data for a Radius $\ell$ Neighborhood] \label{defn:Switchingpart1}
Consider a $d$-regular digraph $\GG$ with $N$ vertices.
Consider a vertex $o$ and its radius $\mathcal{\ell}$ neighborhood $\mt:=B_{\ell}(o,\GG)$.

We define the boundary of $\mt$ to be the vertices of $\mt$ that are connected to a vertex in $\GG$. The boundary edges of $\mt$ are the set of edges that connect a vertex in the boundary of $\mt$ to a vertex in $\mt^c$. We can index these boundary edges as follows:
$\{(l_{\alpha},a_{\alpha}
)\}$; $l_{\alpha}$ is chosen to be the vertex that lies in $\mt$, while $a_{\alpha}$ is chosen to be the vertex that is in $\mt^c$.  We will use 
$\mu$ to denote the number of such edges connecting $\mt$ to $\mt^c$. We remark that if $\mt$ were a purely treelike neighborhood, then $\mu = 2d(2d-1)^{\ell-1}$. Furthermore, in a treelike neighborhood then $\mu/2$ of these edges are oriented in the direction $l_{\alpha} \to a_{\alpha}$ and $\mu/2$ of these edges are oriented in the direction $a_{\alpha} \to l_{\alpha}$.


Since the edges can have different orientations,  we have to be careful regarding the choice of the switching when constructing the switching data.
Our switching data $S=(S_1,\ldots,S_{\mu})$ will consist of a choice of $|\mu|$ random edges that lie strictly inside $\mt^c$ in $\GG$; these will be indexed as $(b_{\alpha},c_{\alpha})$, but the specific value of $b_{\alpha}$ and $c_{\alpha}$ will depend on the orientation. If the edge were oriented in the direction $l_{\alpha}$ to $a_{\alpha}$, then the indices are chosen so that $(b_{\alpha},c_{\alpha})$ were oriented in the direction $b_{\alpha} $ to $c_{\alpha}$. If instead the vertices we had the orientation $a_{\alpha}$ to $l_{\alpha}$, then the edge $(b_{\alpha},c_{\alpha})$ will be oriented $c_{\alpha}$ to $b_{\alpha}$. The purpose of this is to ensure that $a_{\alpha}$ will naturally be switched with $c_{\alpha}$.

The collection $S$ is chosen so that for each $\alpha \le |\mu|$, each edge $b_{\alpha},c_{\alpha}$ is chosen independently and uniformly over all edges that lie in $\mt^c$. Note that this will allow for repetition, but such collisions are highly unlikely and will not make any difference in the ultimate analysis.  The set of all possible switching data corresponding to a graph $\GG$ around the vertex $o$ will be denoted as $\mathbb{S}_{\GG,o}$. The vertices $o$ are naturally identified in $\GG$ and $\tilde{\GG}$.


\end{defn}

We continue our discussion of the switching procedure by actually describing the graphs that are constructed by switching.

\begin{defn}[The Switching Construction]\label{defn:switchingtwo}
Recall the digraph $\GG$, the neighborhood $\mt$ around $o$, as well as the switching data $S$ from Definition \ref{defn:Switchingpart1}.

For each $\alpha$, we define the indicator function variable $\chi_{\alpha}$ which will take values $1$ or $0$ based on the following conditions. $\chi_{\alpha}=1$ if and only if the subgraph $B_{\mathfrak{R}/4}(\{a_{\alpha},b_{\alpha},c_{\alpha}\}, \GG \setminus \mt)$ is a tree even after adding the edge connecting $a_{\alpha}$ to $b_{\alpha}$ and , furthermore, the vertices $\{a_{\alpha},b_{\alpha},c_{\alpha}\}$ is of distance at least $\mathfrak{R}/4 $ from all other switching vertices $\{a_{\beta},b_{\beta},c_{\beta}\}$ for $\beta \ne \alpha$.

The switching map $T$ can be thought of as a bijective map from pairs of $d$-regular digraphs and switching data  $T: (\GG, S) \to (\tilde{\GG},T(S))$ with $S \in \mathbb{S}_{\GG,o}$ and $T(S)  \in \mathbb{S}_{\tilde{\GG},o}$. There is a natural way to associate the vertices $o$ in both $\GG$ and $\tilde{\GG}$.

To construct $\tilde{\GG}$, we perform the following procedure. If $\chi_{\alpha}=1$, then we remove the edge $(l_{\alpha},a_{\alpha})$ and connect $l_{\alpha}$ to $c_{\alpha}$ as well as $b_{\alpha}$ to $a_{\alpha}$. The orientation will be such that the orientation of $l_{\alpha}$ to $c_{\alpha}$ is the same as the orientation of $l_{\alpha}$ to $a_{\alpha}$; in addition, the orientation of $b_{\alpha}$ to $a_{\alpha}$ will match the orientation of $b_{\alpha}$ to $c_{\alpha}$. We also set $\tilde{a}_{\alpha}:= c_{\alpha}$. However, if $\chi_{\alpha}=0$, we make no change to the graph and set $\tilde{a}_{\alpha}:= a_{\alpha}$.

If $\chi_{\alpha}=1$, then we change the switching data $S_{\alpha}$ to the edge $(a_{\alpha},b_{\alpha})$ with the proper orientation. However, if $\chi_{\alpha}=0$, we make no change.  It is easy to check that $T$ is an involution and preserves probability.
\end{defn}

We now give the statement of the following proposition. This will show that all except an $O(1)$ indices $\alpha$ will satisfy $\chi_{\alpha}=1$.
\begin{Prop} \label{Prop:detswitch}
Fix some constant $\mathfrak{q} >0$ and some graph $\GG$ in either $\Omega$ or $\Omega_{S,\mathfrak{q}}$. Fix a vertex $o$ and its $\ell$ neighborhood $\mt$. Then with probability $1- O(N^{-\mathfrak{q}})$, one has that for all except $O_{\mathfrak{q}}(1)$ indices $\alpha$, that $\chi_{\alpha}=1$. Here, $O_{\mathfrak{q}}(1)$ means that this is bounded by a constant depending on the parameter $\mathfrak{q}$. Furthermore, for any vertex $x$ in $\mt^c$, the set of indices $\alpha$ such that $x$ is of distance less than $\mathfrak{R}/4$ from $\{a_{\alpha},b_{\alpha},c_{\alpha}\}$ is less than $O_\mathfrak{q}(1)$.
\end{Prop}

\begin{proof}
This is \cite[Prop. 5.10]{HuangYau}.
\end{proof}

\subsubsection{The inductive method}

At this point, we will now introduce our preliminary induction hypothesis.





\begin{defn}[Induction Hypothesis] \label{def:IndHyp}

Let us define $\epsilon= \epsilon(z,w)$ as our control parameter,
\begin{equation}\label{def:eps}
\epsilon= (\epsilon(z,w)):=  \frac{(\log N)^{K \mathfrak{a}}}{N^{\mathfrak{r}}}+\frac{(\log N)^{K \mathfrak{a}}\sqrt{\text{Im}[m_{\infty}]}}{\sqrt{N \eta}},
\end{equation}
where $K$ is a large positive constant $K \ge 100$.


We let $\Omega^0(z,w)$ denote the set of graphs inside $\Omega$ from Definition \ref{def:Goodregular} such that the following estimates hold at the point $(z,w)$.
\begin{equation} \label{eq:IndHyp}
\begin{aligned}
&|Q_I + Q_O - 2 m_{\infty}| \le \frac{\epsilon}{S_g^1},\\
& |Q_I - Q_O| \le \frac{\epsilon}{S_g^2},\\
&|G_{ij} - G_{ij}(Ext(B_r(\{i,j\},\GG),Q_I,Q_O)| \le \epsilon.
\end{aligned}
\end{equation}
where the latter estimate holds for all pairs $i$ and $j$.   The singularity parameters, $S_g^{1,2}$, are formally defined in Definition \ref{def:Singularities}. 




\end{defn}

Now, we make a definition designed to capture the properties of graphs produced after one switching on the neighborhood $\mt$ around a vertex $o$.

\begin{defn}

We define the new error parameter $\epsilon'$ $(=\epsilon'(z,w))$ as,
\begin{equation}
\epsilon':= (\log N)^{\mathfrak{j}(d) \mathfrak{a}} \epsilon,
\end{equation}
where $\mathfrak{j}(d) = \log_{d-1} \left[ \frac{1}{2} + \frac{1}{2} \sqrt{\frac{2d-1}{d-1}} \right]<1/2$.

Let $\GG$ be a graph inside $\Omega_{S,\mathfrak{q}}$ from Definition \ref{def:Goodregular}, and we fix a vertex $\hat{r}$ in $\GG$ with radius $\ell$ neighborhood, $\mt$. We will include $\GG$ inside the set $\Omega^1_{S,\mathfrak{q},\hat{r}}(z,w)$ if it satisfies the following estimates at the point $(z,w)$,
\begin{equation} \label{eq:firstswitch}
\begin{aligned}
&|Q_I +Q_O - 2 m_{\infty}| \lesssim   \frac{\epsilon}{S_g^1},\\
&|Q_I - Q_O| \lesssim \frac{\epsilon}{S_g^2},\\
&|G_{i+o_i,j +o_j} - G_{i+o_i,j+o_j}(Ext(B_r(\{i,j\},S(\GG)),Q_I,Q_O))| \lesssim \frac{1}{\log N}, \forall i,j \in \GG \\
&|G^{(\mt)}_{i+o_i,j+o_j}- G_{i+o_i,j+o_j}(Ext(B_r(\{i,j\},S(\GG)) \setminus \mt,Q_I,Q_O))| \lesssim \epsilon', \forall i,j \in \mt^c.
\end{aligned}
\end{equation}
Here, $o_i$ and $o_j$ are symbols that can take the value $0$ or $N$, so that $i + o_i$ can correspond to one of the two labels corresponding to the vertex $i$.
The constant implicit in $\lesssim$ is uniform and does not depend on $N$.

\end{defn}

We have the following  relating $\Omega^0(z,w)$ to $\Omega^1_{S,\mathfrak{q},o}(z,w)$.
\begin{thm} \label{thm:firstswitch}
Fix a graph $\GG$in $\Omega^0(z,w)$, as well as a vertex $o$ with radius $\ell$ neighborhood $\mt$. In addition, fix the parameter $\mathfrak{q}>0$.  There is an event $S_G(\GG)$ of switching events $S$ in $\mathbb{S}_{\GG,o}$ with probability $\mathbb{P}(S_G(\GG)) = 1 - O(N^{-\mathfrak{q}})$, such that for every switching event $S \in S_G(\GG)$, we have that $\tilde{\GG}$ is in $\Omega^1_{S,\mathfrak{q},o}(z,w)$.
\end{thm}

The issue with the above statement is that it does not take into account improved concentration estimates from the switching. Our issue is that we cannot ensure that after a single switching, the event $\Omega^0(z,w) \setminus T(\Omega^0(z,w))$ is of vanishingly small probability. In order to avoid this issue, we use the fact that $T$ is an involution and two evaluations of $T$ on the set $\Omega^0(z,w)$ will result in a graph that belongs to $\Omega^0(z,w)$.

To this end, we introduce a new definition that describes the graphs that we would obtain after two switching.
\begin{defn} \label{def:Omegsecswitch}
Let $\GG$ be $d$-regular digraph with $N$ vertices, and let $\hat{r}$ be a vertex in $\GG$ with radius $\ell$ neighborhood $\mt$.  For all vertices $i$ in $\mt$, we define $\ell_i$ to be the distance from $\hat{r}$ to $i$. We set that the graph $\GG$ belongs to the set $\Omega^{2}_{S,\mathfrak{q},\hat{r}}(z,w)$ if it satisfies the following conditions at the point $(z,w)$:
\begin{equation} \label{eq:secswitchest1}
\begin{aligned}
& |Q_I + Q_O- 2 m_{\infty}| \lesssim \frac{\epsilon}{S_g^1},\\
 &|Q_{I}-Q_O| \lesssim \frac{\epsilon}{S_g^2}.
\end{aligned}
\end{equation}
\begin{equation}\label{eq:secswitchest2}
\begin{aligned}
&|G_{\hat{r},i} - G_{\hat{r},i}(Ext(B_r(\{\hat{r},i\}, \GG), Q_I,Q_O))| \lesssim \left(\frac{\log N}{ (d-1)^{\ell/2}} + \frac{(\log N)^2}{(d-1)^{\ell -\ell_i/2}} \right) \epsilon'\\
& + \log N(|S^1_g| |Q_I +Q_O - 2 m_{\infty}| + |S^2_g||Q_I - Q_O| + |Q_I-m_{\infty}|^2 + |Q_O-m_{\infty}|^2), \forall i \in \mt\\
\end{aligned}
\end{equation}

\begin{equation}\label{eq:secswitchest3}
\begin{aligned}
&|G_{\hat{r},i} - G_{\hat{r},i}(Ext(B_r(\{\hat{r},i\}, \GG), Q_I,Q_O))| \lesssim \frac{(\log N)^2 \epsilon'}{(d-1)^{\ell/2}}\\
&+ \log N(|S^1_g| |Q_I +Q_O - 2 m_{\infty}| + |S^2_g||Q_I - Q_O| + |Q_I - m_{\infty}|^2 + |Q_O-m_{\infty}|^2), \forall i \in \mt^c
\end{aligned}
\end{equation}

Furthermore, if the vertex $\hat{r}$ has a radius $\mathfrak{R}$ tree-like neighborhood in $\GG$, then we also have,

\begin{equation} \label{eq:estYi}
\begin{aligned}
&\Big|\frac{1}{d} \sum_{(i \to \hat{r})}G^{(i)}_{\hat{r}+N, \hat{r}+N} - Y_{i, \ell}(Q_I,Q_O) \Big| \lesssim \frac{\log N \epsilon'}{(d-1)^{\ell/2}},\\
&\Big|\frac{1}{d} \sum_{(\hat{r} \to i)}G^{(i)}_{\hat{r}, \hat{r}} - Y_{o, \ell}(Q_I,Q_O) \Big| \lesssim \frac{\log N \epsilon'}{(d-1)^{\ell/2}}.
\end{aligned}
\end{equation}

As always, $\lesssim$ represents a bound by a universal constant that does not depend on $N$.

\end{defn}

The improvement of the last estimate is found in the decay factors $\frac{1}{(d-1)^{\ell/2}}$ found in the denominator. Our main relationship between $\Omega^1_{S,\mathfrak{q},o}(z,w)$ and $\Omega^2_{S, \mathfrak{q},o}(z,w)$ is the following theorem.

\begin{thm}\label{thm:secondswitching}
Fix a graph $\GG$ in $\Omega^1_{S,\mathfrak{q},o}(z,w)$. This graph comes with a specially denoted vertex $o$ and neighborhood $\mt $. There is a collection of switches $S_C(\GG) \in \mathbb{S}_{\GG,o}$ with probability $\mathbb{P}(S_{C}(\GG)) = 1 - O(N^{-\mathfrak{q}})$ such that the image of $(\GG,s)$ under $T$ for $s \in S_C(\GG)$ is in $\Omega^2_{S,\mathfrak{q},o}(z,w)$.

\end{thm}

\subsection{A Multiscale Induction Scheme}

Our central Theorem \ref{lem:auxmain} is proven via a multiscale induction scheme.Namely, we start by proving the estimates in equation \eqref{eq:auxcollect}. We start by establishing these estimates  at some $z_0$ where $|z_0|$ is large and one can derive the inequalities by basic computations with expansions. Define $z_k = z_0 - k N^{-4}.$ Now, if we establish the equations $\eqref{eq:auxcollect}$ at some scale $z_k$, the goal is to prove these inequalities at $z_{k+1}$. On a high level, the key input is Theorem \ref{thm:secondswitching}; this shows that, with very high probability, a switched graph would satisfy stronger Green's function estimates at the scale $z_k$. Afterward, one can use basic Lipschitz continuity estimates to take these improved Greens function estimates at $z_{k}$ to obtain the inductive hypothesis at $z_{k+1}$. We can more precisely quantify the improvement from switching int he following definition. 



\begin{defn}
We let $\Omega_g(z,w) \subset \Omega$ be the set of $d$-regular digraphs $\GG$ that satisfy the following estimates at the point $(z,w)$,
\begin{equation}
\begin{aligned}
&|Q_I + Q_O - m_{\infty}| \leq \frac{\epsilon}{2 S_g^1},\\
&|Q_I - Q_O| \leq \frac{\epsilon}{2 S_g^2},\\
&|G_{ij} -G_{ij}(Ext(B_r(\{i,j\}, \GG), Q_I,Q_O))| \le \frac{\epsilon}{2}, \qquad \forall i,j.
\end{aligned}
\end{equation}

\end{defn}

Our first goal is to show that for $|z|$ large enough, we know that any graph lies in $\Omega_{g}(z,w)$ deterministically.

\begin{Prop} \label{Prop:DeterBnd}
Let $w$ lie in some compact domain $\mathbb{D}$ of $\mathbb{C}$ and choose $z$ with $|z| \ge 2d^2 + \max_{a \in \mathbb{D}} d|a|$. Then, any $d$-regular digraph $\GG$ on $N$ vertices belongs to $\Omega_g(z,w)$.
\end{Prop}

\begin{proof}
We know that the largest eigenvalue of a Hermitized version $H$ of the adjacency matrix $A$ of a 
$d$-regular digraph can be bounded by $\sqrt{d} +|w|$; this can be obtained by taking the maximum sum of the elements in a row or in a column. This shows us that we may expand the Green's function of our Hermitized matrix as,
\begin{equation}
(H - z)^{-1} =  -z^{-1}  - \sum_{k=1}^{\infty} H^k z^{-(k+1) }.
\end{equation}

Now, one can see that the computation of $H^{k}_{a b}$ involves counting the contribution of all length $k$ paths between $a$ and $b$ in the graph $\GG$ with possible self-loops weighted by $w$ or $\overline{w}$.  From this logic, we can assert the following statement on graphs $\GG_1$ and $\GG_2$ that each have vertices $o$ that are identified to each other. If the radius $r$ neighborhood of $o$ in $\GG_1$ is isomorphic to the radius $r$ neighborhood of $o$ in $\GG_2$, then $\sum_{k=1}^r H^k z^{-(k+1)}$ agree with each other in both $\GG_1$ and $\GG_2$.

Thus, we would know that,
\begin{equation}
|G_{ij} -G_{ij}(Ext(B_r(\{i,j\}, \GG), m_{\infty},m_{\infty}))|  \le 2 \sum_{k= r+1}^{\infty}||H|| |z|^{-(k+1)} \le 4 \left(\frac{1}{d}\right)^{r} \lesssim \frac{\epsilon}{(\log N)^2}
\end{equation}

In particular, if $i$ has a radius $\mathfrak{R}$ tree-like neighborhood, then
\begin{equation}
|G_{ii} - m^{d}_{\TT}| \le 4  \left(\frac{1}{d}\right)^{r} .
\end{equation}
Here, we only used the fact that $G_{ii}(Ext(B_r(i, \GG),m_{\infty},m_{\infty}))$ will be equal to $m^{d}_{\TT}$.

 These perturbation estimates will also give us information on $G_{ii}^{(j)}$ via an application of the following identity,
\begin{equation}
G_{ii}^{(j)} = G_{ii} - G|_{i, \{j, j+N\}} (G|{\{j,j+N\}})^{-1} G|_{\{j,j+N\} ,i},
\end{equation}
as well as a similar one for the extension $G_{ii}(Ext(B_r(i, \GG^{(j)}),m_{\infty},m_{\infty}))$. This will show that,
\begin{equation}
|G_{ii}^{(j)} - m_{\infty}| \lesssim \frac{\epsilon}{(\log N)^2},
\end{equation}
for those $i$ with a tree-like neighborhood. Otherwise, we can assert that $G_{ii}^{(j)}$ is $O(1)$ for the $N^{\mathfrak{c}}$ many vertices that do not have a tree-like neighborhood.

This will show that,
\begin{equation}
\Big| \frac{1}{Nd} \sum_{(j \to i)} G_{ii}^{(j)} -  m_{\infty} \Big| \lesssim N^{\mathfrak{c}-1} + \frac{\epsilon}{(\log N)^2} \lesssim \frac{\epsilon}{(\log N)^2}.
\end{equation}
We have a similar statement for $Q_O$.

Finally, we can use Lemma \ref{lem:pert} to replace $$G_{ij}(Ext(B_r(\{i,j\}, \GG), m_{\infty},m_{\infty}))$$ with $$G_{ij}(Ext(B_r(\{i,j\}, \GG), Q_I, Q_O)).$$



\end{proof}

Our next Proposition discusses how the switching procedure will allow us to take an element from
$\Omega^0(z,w)$ from Definition \ref{def:IndHyp} and return an element of the improved set $\Omega_g(z,w)$ after two random switching.

\begin{Prop} \label{Prop:SwitchBnd}
Fix $w$ in a compact domain $\mathbb{D} \in \mathbb{C}$. There exists some $\varepsilon>0$  such that for any $z$ with $\text{Im}[z] \ge N^{-\varepsilon}$ and $d \ge 3$, we have that,
\begin{equation}
\mathbb{P}(\Omega^0(z,w) \setminus \Omega^0_g(z,w)) = O(N^{-\mathfrak{q} +1}).
\end{equation}
\end{Prop}
\begin{proof}
The proof will be divided into two steps:
\begin{enumerate}
\item We will show for each vertex $o$ that $\mathbb{P}(\Omega^0(z,w) \setminus \Omega^2_{S, \mathfrak{q},o}(z,w)) = O(N^{- \mathfrak{q}})$. $\Omega^2_{S,\mathfrak{q},o}$ is the set from  Definition \ref{def:Omegsecswitch}. A union bound will then show that $$\mathbb{P}(\Omega^0(z,w) \setminus \bigcap_{o \in \GG} \Omega^2_{S,\mathfrak{q},o}(z,w)) \lesssim N^{1- \mathfrak{q}}. $$

\item After this, we show that a graph that is in $\Omega^0(z,w)\bigcap_
{o \in \GG} \Omega^2_{S, \mathfrak{q},o}(z,w)$ is in $\Omega_g(z,w)$ via some self-consistent estimates.

\end{enumerate}

\textit{Proof of Step 1}

Just to distinguish notation, we will let $\mathbb{P}$ denote the probability distribution over graphs, while we let $\tilde{\mathbb{P}}$ denote the joint probability distribution over graphs and switching data. Furthermore, for a set of graphs $\mathbb{G}$, we will use $\tilde{\mathbb{G}}$ to denote the set of the form $(\GG,S)$ with $\GG \in \mathbb{G}$ and $S$ can be any switching data in $\mathbb{S}_{\GG,o}$. (We assume that each graph has a specially marked vertex $o$).

Theorem \ref{thm:firstswitch} shows that,
$$
\tilde{\mathbb{P}}(T(\tilde{\Omega}^0(z,w)) \setminus \tilde{\Omega}^1_{S,\mathfrak{q},o}(z,w)) \lesssim N^{-\mathfrak{q}}.
$$
Since $T$ both preserves measure and is an involution, this will also show that
\begin{equation}
\tilde{\mathbb{P}}(\tilde{\Omega}^0(z,w) \setminus T(\tilde{\Omega}^1_{S,\mathfrak{q},o}(z,w)) \lesssim N^{-\mathfrak{q}}
\end{equation}

Theorem \ref{thm:secondswitching} shows that,
\begin{equation}
\tilde{\mathbb{P}}(T(\tilde{\Omega}^1_{S,\mathfrak{q},o}) \setminus \tilde{\Omega}^2_{S,\mathfrak{q},o}(z,w)) \lesssim N^{-\mathfrak{q}}.
\end{equation}

Combining these two estimates will give us,
\begin{equation}
\tilde{\mathbb{P}}(\tilde{\Omega}^0(z,w) \setminus \tilde{\Omega}^2_{S,\mathfrak{q},o}(z,w)) \lesssim N^{-\mathfrak{q}}.
\end{equation}

A simple projection shows that this will imply that $\mathbb{P}(\Omega^0(z,w) \setminus \Omega^2_{S,\mathfrak{q},o}(z,w)) \lesssim N^{-\mathfrak{q}}$.

\textit{Proof of Step 2}

We first form a self-consistent equation to understand the values of $Q_I$ and $Q_O$.
If we sum up the estimates in equation \eqref{eq:estYi}, we see that we have that,
\begin{equation}
\begin{aligned}
 &\Big|\frac{1}{Nd} \sum_{(i\to j) \in \GG} G_{j+N,j+N}^{(i)} - Y_{i,\ell}(Q_I,Q_O)\Big| \lesssim \frac{\log N \epsilon'}{(d-1)^{\ell/2}},\\
&\Big|\frac{1}{Nd} \sum_{(j\to i) \in \GG} G_{j,j}^{(i)} - Y_{i,\ell}(Q_I,Q_O)\Big| \lesssim \frac{\log N \epsilon'}{(d-1)^{\ell/2}}.
\end{aligned}
\end{equation}

This is the self-consistent equation for $Q_I$ and $Q_O$ that has been analyzed in Theorem \ref{thm:Self-conc}. 

Now, the estimates on $G_{oi}$ from \eqref{eq:secswitchest2} and \eqref{eq:secswitchest3} are good enough for the Green's functions estimates in  Theorem \ref{thm:Self-conc}. The main benefit is the presence of the $(d-1)^{\ell/2} = (\log N)^{\mathfrak{a}/2}$ factor in the denominator, which is sufficient to cancel  the $(\log N)$ factor in the numerator as well as the factor $(\log N)^{\mathfrak{j}(d) \mathfrak{a}}$ factor in the definition of $\epsilon'$ when $d \ge 3$.

This shows that,
\begin{equation}
\begin{aligned}
&|Q_{I}+Q_O- 2m_{\infty}| \lesssim \frac{\epsilon'}{S^1_g(\log N)^k},\\
&|Q_I - Q_O| \lesssim \frac{\epsilon'}{S^2_g (\log N)^k}.
\end{aligned}
\end{equation}
where $k$ can be arbitrarily large if $\mathfrak{a}$ is chosen large enough.

Now, this bound is good enough to use in the other Green's function estimates found in equation \eqref{eq:secswitchest2} and \eqref{eq:secswitchest3}. As we have mentioned before, the term $\frac{\log N \epsilon'}{(d-1)^{\ell/2}}\epsilon'$ is $ \lesssim \frac{\epsilon}{\log N}$ due to the presence of the factor $(d-1)^{\ell/2}$ in the denominator.  Now, our improved bounds on $|(Q_{I} - m_{\infty}) \pm (Q_O -m_{\infty})|$ can be substituted into the right hand sides of these expressions in order to deal with the remaining terms on the right hand sides of  equations \eqref{eq:secswitchest2} and \eqref{eq:secswitchest3}. This completes the proof of the proposition.
\end{proof}

At this point, one is left with preforming a standard multi-scale analysis with continuity estimates.

\begin{proof}[Proof of Theorem \ref{lem:auxmain}]

Consider the domain
$$
\{(w,z): w \in \mathbb{D}, \text{Im}[z] \ge N^{-\epsilon}, |z| \le d^2 + d \max_{y \in \mathbb{D}}|y| +1 \}.
$$ Now consider a dense discrete grid $\mathbb{L}$ inside this domain where the spacing between elements is $\frac{1}{N^5}$. By basic derivative estimates of the trace of a Green's function, we know that if $|z- z'| \le \frac{1}{N^5}, |w - w'| \le \frac{1}{N^5}$. Then, $\left|\frac{1}{N} \text{Tr}[G(z,w)] - \frac{1}{N} \text{Tr}[G(z',w')] \right| \lesssim \frac{1}{N}$, deterministically. Thus, it suffices to prove the local law with high probability on the elements of the discrete grid under question.

For fixed value of $\text{Re}[z]$ and $w$, we can order the values of $z$ in the grid with these values of $\text{Re}[z]$ and this $w$ in order of decreasing imaginary part as $\text{Im}[z_1] \ge \text{Im}[z_2] \ge \text{Im}[z_3] \ge \ldots \ge \text{Im}[z_M]$ with $M= O(N^5)$.

Now, we know that $\mathbb{P}(\Omega_g(z_1,w)) =1$ due to Proposition \ref{Prop:DeterBnd}. Out continuity bound shows that $\Omega_g(z_k,w) \subset \Omega^0(z_{k+1},w)$. Now, Proposition  \ref{Prop:SwitchBnd} shows that $\mathbb{P}(\Omega^0(z_{k+1},w) \setminus \Omega_g(z_{k+1},w)) \lesssim N^{-\mathfrak{q}}$. Thus, we know that $\mathbb{P}(\Omega_g(z_{k+1},w) \ge \mathbb{P}(\Omega_g(z_k,w)) - O(N^{1-\mathfrak{q}}).$ This is iterated up to at most $N^5$ steps so $\mathbb{P}(\Omega_g(z_M, w)) \ge 1 - O(N^{16 - \mathfrak{q}})$. We can now take a union bound over our dense grid. Choosing $\mathfrak{q}$ sufficiently large will show that the intersection $\mathbb{P}(\bigcap_{(z,w) \in \mathbb{L}} \Omega_g(z,w)) = 1- o(1)$ as desired.
\end{proof}

\subsection{The strategy of proving Theorems \ref{thm:firstswitch} and \ref{thm:secondswitching}}

 Now that we have introduced the notation necessary to discuss the switching arguments and the proofs of the Green's function estimates, we will begin to give an overview of the series of estimates that are required in order to derive the theorems that were previously mentioned in this section. However, before we detail the strategy, we will discuss on a high level the differences that arise when considering the Green's functions of the adjacency matrices of directed graphs as compared to undirected graphs. First of all, since we now have to deal with edges with orientations, this means that when performing a switching, we have to preserve the orientation; furthermore, since we now have two distinct classes of edges, we get two different order parameters $Q_I$ and $Q_O$ corresponding to the different orientations of the edges. Furthermore, it will be observed in the proof that many estimates of Green's function involve understanding the values of Green's functions for the infinite $d$-regular tree and obtaining appropriate $l^2$ and $l^1$ controls for the sum of Green's function for the infinite tree. In the undirected case, these Green's function values were a simple function of the distance. However, when considering the directed graph, the specific path and the orientation of edges along the path will drastically change the value of the Green's function between two vertices; keeping track of these differences occurring from the path structure is the main technical challenge that occurs in the directed case.

The proofs of these theorems will come a series of estimates. To simplify the organization, the proofs are divided into multiple parts. We will briefly describe the structure here. In the statements that follow, if we let $i$ be a vertex, then it could represent either $i$ or $i+\aleph$ when it appears as a sub-index of a Green's function value.

\textit{Steps to Prove Theorem \ref{thm:firstswitch}}

\begin{enumerate}

\item Our first step is to relate estimates of $G$ to $G^{(\mt)}$. This involves applying the Schur complement formula; a major ingredient in intermediate estimates are the Green's function estimates on tree-like neighborhoods from Section \ref{sec:Tree-like}.  This is the goal of Section \ref{sec:step1}.

The main result is the following:
\begin{Prop} \label{prop:removeTest}
Consider a graph $\mathcal{G}$ in $\Omega^0(z,w)$. Fix a vertex $o$ and let $\mathbb{T}$ be the radius $\ell$ neighborhood around $\mathbb{T}$. Assume, in addition, that $o$ has a radius $\mathfrak{R}$ neighborhood with excess at most 1.  Then, we have the following estimates on the Green's function of  the graph $\mathcal{G}$ with the neighborhood $\mathbb{T}$ removed. Let $i$ and $j$ be vertices outside of $\mathbb{T}$ in $\mathcal{G}$ and let $P$ denote $G(Ext(B_r(\mathbb{T}\cup\{i,j\},\mathcal{G}),Q_I,Q_O))$.  Then, we have that, 
\begin{equation} \label{eq:removeTest}
|G_{ij}^{(\mt)} - P_{ij}^{(\mt)}| \le  \epsilon'.
\end{equation}

\end{Prop}

\item Our second step is to then remove the vertices $W_S$ that take part in the switching. Again, this involves applying the Schur complement formula. However, this time we need to introduce a notion called Green's function distance. This uses the fact that the average case bound on the Green's functions $|G_{ij}|$ is usually better than the worst case bound; Green's function distance quantifies information on when we can apply the average case bound rather than the worst case bound. This is the goal of Section \ref{sec:step2}. Though the specific value of $\phi$ and other assorted definitions will be defined in the section, we state the Proposition we will prove here for convenience.

\begin{Prop}\label{prop:removeWS}
Consider a graph $\mathcal{G}$ that belongs to the setting of Proposition \ref{prop:removeTest}. Thus, we consider graphs that satisfy the important conclusion \eqref{eq:removeTest}. Given a switching event $S$ involving exchanging the edges on the boundary of $\mathbb{T}$ with arbitrary edges given by $\{b_1,c_1\}\ldots\{b_{\mu},c_{\mu}\}$ as in Definition \ref{defn:switchingtwo}, then $W_S$ is the set of those $b_{\alpha}$ with $\chi_{\alpha}=1$. Then, there is an event of switching, $S_{G}(\mathcal{G})$ that holds with probability $1- O(N^{-q})$ such that the following estimates hold uniformly.

Let $i$ and $j$ be vertices in of $\mathcal{G}\setminus(\mathbb{T} \cup W_S)$ and  let $P$ be a shorthand for $G(Ext(B_r(\{i,j\}) \cup W_S, \mathcal{G} \setminus \mathbb{T},Q_I,Q_O))$.

\begin{equation} \label{eq:Greendiscon}
|\tilde{G}_{ij}^{(\mathbb{T} \cup W_S)} - P_{ij}| \lesssim \phi.
\end{equation}

Otherwise, we have the worst case bound,
\begin{equation}\label{eq:Greencon}
 |\tilde{G}_{ij}^{(\mathbb{T} \cup W_S)} - P_{ij}| \lesssim \epsilon'. 
\end{equation}

\end{Prop}

\item Our third step is to reintroduce the vertices in $W_S$, but with adjacency matrix given by the  switching procedure. This section will give us most of the estimates on $\tilde{G}^{(\mt)}$  for the switched graph; namely, we derive the third inequality of equation \eqref{eq:firstswitch}.  This is the goal of Section \ref{sec:Step3}. Our formal statement is as follows:
\begin{Prop}\label{Prop:SwitchGWS}
Consider the setting of the previous  Proposition \ref{prop:removeWS}. Thus, we are considering a graph $\mathcal{G}$ along with a switching event $S \in S_G(\mathcal{G})$ such that we satisfy both of the estimates \eqref{eq:Greendiscon} and \eqref{eq:Greencon}. Construct the switched graph $\tilde{\mathcal{G}}$ as in Definition \ref{defn:switchingtwo}. Let $\tilde{G}$ denote the Green's function of the switched graph $\tilde{G}$. Let $i$ and $j$ be arbitrary vertices in $\tilde{\mathcal{G}}\setminus{\mathbb{T}}$ and let $P$ be a shorthand for $G(Ext(B_r(\{i,j\} \cup (\mathbb{T} \cup W_S),\tilde{\mathcal{G}} \setminus T),Q_I,Q_O)$.

Then, we have the following estimates:
 If $i$ and $j$  are not Green's function connected and either both $i$ and $j$ belong to $W_S$ or both are outside $W_S$, then we have the bound,
\begin{equation} \label{eq:Gtilddiscon}
|\tilde{G}^{(\mathbb{T})}_{ij} - P_{ij}| \lesssim \phi.
\end{equation}

Recall our notation for $\{b_{\alpha},c_{\alpha}\}$ for the set of vertices that participate in the switching.
In the specific case where we let $i$ be of the form $b_{\beta}$ and $j$ be of the form $c_{\alpha}$ with $\alpha \ne \beta$, then we also have the estimate,
\begin{equation} \label{eq:Greendisconbal}
|\tilde{G}_{b_{\beta}c_{\alpha}}^{(\mathbb{T} \cup W_S)} - P_{b_{\beta}c_{\alpha}}| \lesssim \phi + (\epsilon')^2.
\end{equation}

Otherwise, we have the worst case bound,
\begin{equation} \label{eq:Gtildcon}
|\tilde{G}^{(\mathbb{T})}_{ij} - P_{ij}| \lesssim \epsilon'.
\end{equation}

\end{Prop}

\item Our fourth step is to finally re-introduce the vertices of $\mt$ to the switched graph. This is the goal of Section \ref{sec:step4}. Formally, our main result is the following,
\begin{Prop} \label{prop:introT}
Consider the setting of the previous Proposition \ref{Prop:SwitchGWS}; especially the important estimates \eqref{eq:Gtilddiscon}, \eqref{eq:Greendisconbal}, and \eqref{eq:Gtildcon}. When we add back the neighborhood $\mathbb{T}$ to the switched graph $\tilde{G}$, we will get the following Green's function estimates. Let $i$ and $j$ be arbitrary vertices in the graph $\tilde{G}$ and let $P$ be a shorthand for $G(\text{Ext}(B_{r}(\{i,j\} \cup \mt, \tilde{\GG}), \tilde{Q}_I,\tilde{Q}_O ))$. Here, $\tilde{Q}_I$ and $\tilde{Q}_O$ are the corresponding values of $Q_I$ and $Q_O$ corresponding to the switched graph $\tilde{G}$. 

If $i$ and $j$ both belong to the neighborhood $\mathbb{T}$, then we have,
\begin{equation} \label{eq:StabEst1}
|\tilde{G}_{ij} - P_{ij}| \lesssim \epsilon' (\log N)^2 \left(\frac{1}{\sqrt{d-1}} \right)^{2\ell -d_i -d_j} + (1 +|d_j  +d _i|)  \left( \frac{1}{d-1}\right)^{\frac{|d_j -d_i|}{2}} \epsilon'.
\end{equation}

If instead one of $i$ belongs to $\mathbb{T}$ and the other belongs to $\mathbb{T}^c$, then we have the estimates,
\begin{equation}\label{eq:step4io1}
|\tilde{G}_{ij}-P_{ij}| \lesssim (log N)^2\epsilon' \left(\frac{1}{\sqrt{d-1}} \right)^{\ell -d_i}.
\end{equation}

In addition, we also have,
\begin{equation} \label{eq:Qdiff1}
\begin{aligned}
&|Q_{I}+Q_O - 2 m_{\infty}| \lesssim \frac{\epsilon}{S^1_g},
&|Q_I - Q_O| \lesssim \frac{\epsilon}{S^2_g}.
\end{aligned}
\end{equation}

\end{Prop}

\end{enumerate}

Given this final statement, Theorem \ref{thm:firstswitch} is a collection of estimates proved already in the previous propositions. 
\begin{proof}[Proof of Theorem \ref{thm:firstswitch}]

Consider a graph $\mathcal{G}$ in $\Omega^0(z,w)$. To this graph, we can apply the conclusion of Proposition \ref{prop:introT}; from this we have a  set of good switching events $S_{G}(\mathcal{G})$.
On these good switching events, we can derive equations \eqref{eq:StabEst1} and \eqref{eq:step4io1}. This corresponds to the middle equation of \eqref{eq:firstswitch}. Additionally, \eqref{eq:Qdiff1} correspond to the first equation of \eqref{eq:firstswitch}. Finally, equation \eqref{eq:removeTest} is the final equation of \eqref{eq:firstswitch}.
\end{proof}

\textit{Steps to Prove Theorem \ref{thm:secondswitching}}

Most of the same estimates used to prove Theorem \ref{thm:firstswitch} can be used to prove many of the estimates necessary to prove Theorem \ref{thm:secondswitching}. The only new ingredient are new concentration estimates, that involve the fact that the switch used random edges. This step will be discussed in Section \ref{sec:concentration}.

\section{Perturbation theory for almost tree-like neighborhoods:\\
 The proof of Theorem \ref{thm:pert}}\label{sec:Tree-like}

The basic underpinning of our strategy is the ability to deal with the Green's function of graphs with almost tree-like neighborhoods via perturbations of the one for the infinite tree (and when the neighborhood is not tree-like, we use the relation between the Green's function of a graph and the Green's function of the covering graph). In this section we detail the relevant estimates one needs for this purpose.

Hereafter, for digraphs of a generic size we assign one standard index per vertex, 
denoted by lowercase letters, such as $i$. Pairing each such index $i$ with 
a second index, denoted $i+\aleph$, any forward edge $(i \to j)$ of the digraph 
yields two non-zero entries $H_{i,j+\aleph}$ and $H_{j +\aleph,i}$ of the matrix
$H(z,w)$. As we did in case of trees, the entry $H_{i,j+\aleph}$ is interpreted 
as traversing a forward edge, and $H_{j+\aleph,i}$ as traversing 
a backward edge. 

\subsection{General perturbation bounds}

The goal of this section if to prove general perturbations bounds for Green's functions of almost tree-like neighborhoods. Theorem \ref{thm:pert} is an example of a result of this type and is used directly in the proof of our main result. However, there are times where we need to use more sophisticated estimates and, thus, this subsection will also contain more delicate estimates for these cases.  Before we actually go on and prove Theorem \ref{thm:pert}, we will collect some necessary preliminary results; we remark here that many of these useful combinatorial results also appear in the Appendix. 
We start with some useful definitions and lemmas; our first definition will discuss the nature of singularities that appear in the self-consistent equations.

\begin{defn}[Singularities of the self-consistency equation] \label{def:Singularities}
Singularities that may occur in computations with
our (distinguished) solution $m_\infty$ of \eqref{eq:self-conc}, are parametrized by
\begin{equation} \label{eq:singularities}
S_g^1:= 1-X-Y  \le 1 +X - Y := S_g^2, 
\end{equation}
for $X,Y$ of \eqref{dfn:XY}.
\end{defn}
\begin{lemma} \label{lem:stability}
Suppose a tree-like digraph $\GG$ with root $\hat{r}$ is such that 
\[
 (1 + \text{diam}(\GG))^3 [|Q_+|+|Q_{-}|] \ll 1 \,,
\]
where $Q_+:= \frac{Q_I + Q_O-2 m_\infty}{\sqrt{2}}$ and $Q_- := \frac{Q_I - Q_O}{\sqrt{2}}$.
then we have in terms of $S_g^1$, $S_g^2$ 
and $\text{Ext}_i(\GG, Q_I, Q_O)$ of Definition \ref{def:GraphExt}, that
\begin{equation} \label{eq:stabilityestimate}
\begin{aligned}
&|G_{\hat{r}+\aleph, \hat{r}+\aleph}(\text{Ext}_i(\GG,Q_I,Q_O)) - Q_I| \lesssim \\ &
\hspace{1 cm} ( 1+ \text{diam}(\GG))^2 \left(S_g^1 |Q_+ | + S_g^2 |Q_-|   \right) 
 + ( 1+ \text{diam} (\GG))^3[|Q_+ |^2 + |Q_-|^2].
\end{aligned}
\end{equation}
One can obtain a similar expression comparing $G_{\hat{r},\hat{r}}(\text{Ext}_o(\GG,Q_I,Q_O))$ and $Q_O$.

\end{lemma}

\begin{rmk} We include the factors $1 \pm X- Y$ in order to match the singularity factors 
we expect in the stability equation for $\tilde{Q}_1 - \sqrt{2} m_{\infty}$ and $\tilde{Q}_2.$

\end{rmk}

In order to prove this lemma, we need to introduce the following definition that aids in counting.
\begin{defn}[Basal Tree]
Consider the $d$-regular infinite digraph tree $\mathcal{T}$ with root $\hat{r}$. For $k$ a positive integer, let $v_1,\ldots,v_k$ be vertices on this infinite tree. If we let $P$ be the Green's function corresponding to this infinite tree, consider the product $|P_{\hat{r}v_1}||P_{v_1v_2}||P_{v_2 v_3}|\ldots|P_{v_n \hat{r}}|$. This product can be associated with what we call a basal tree, $T_k$, that reflects the general structure of this product and is convenient for summation. 

If $k=1$, the basal tree, $T_1$, consists of a single edge between two vertices. The interpretation of this is as follows. The two vertices of the edge should correspond to the vertices $r$ and $v_1$, while the path between $r$ and $v_1$ corresponds to the edge between $r$ and $v_1$.  In general, our construction will correspond vertices of the basal tree to some vertices in $\mathcal{T}$ and edges of the basal tree to paths in $\mathcal{T}$.

 For $k >1$, we construct the basal tree inductively. Let $T_{k-1}$ denote the basal tree corresponding to the product $|P_{\hat{r}v_1}||P_{v_1v_2}||P_{v_2 v_3}|\ldots|P_{v_{n-1} \hat{r}}|$. Consider the subgraph in $\mathcal{T}$ consisting of the union of paths between $v_i$ to $v_{i+1}$ for $i=1,\ldots, n-2$ as well as the paths between $\hat{r}$ and $v_1$ and $v_{n-1}$; call this subgraph $S$. Now consider the path from $r$ to $v_n$. This path has some final ancestral vertex $A_n$, which is the last vertex in common between the path between $r$ and $v_n$ and $S$. 

 In the correspondence between the edges  of $T_{k-1}$ and paths in $\mathcal{T}$, the vertex $A_n$ must like on one of these paths; call this path $p$ and the edge corresponding to it in $T_{k-1}$ $e$. To construct $T_k$ from $T_{k-1}$, add a vertex in the middle of the edge $e$; this vertex in $T_n$ corresponds to $A_n$. Finally, add a new leaf vertex at this new vertex. The other endpoint of this leaf corresponds to the vertex $v_k$ while the leaf edge corresponds to the path between $A_n$ and $v_n$. This completes the inductive construction of the tree. 
\end{defn}

\begin{proof}
As in the proof of the stability equation, we consider the difference between 

\noindent $ G: =G_{\hat{r}+\aleph, \hat{r}+\aleph}(\text{Ext}_i(\GG,Q_I,Q_O)) $ and $P:= G_{\hat{r}+\aleph, \hat{r}+\aleph}(\text{Ext}_i(\GG,m_{\infty},m_{\infty}))$. The function $P$ gives the Green's function entries of the infinite tree $\TT_1$.

By applying the resolvent identity, we have,
\begin{equation}
G- P =  \sum_{k=1} ^{\infty}P[( P^{-1} - G^{-1}) P]^k.
\end{equation}
Notice that the terms of $P^{-1} - G^{-1}$ are of the form $ Q_I - m_{\infty}$ or $Q_O - m_{\infty}$. Thus, we see that higher order terms appearing in $P[(P^{-1} - G^{-1}) P]^k$ will give us errors of order $O( (Q_\cdot - m_{\infty})^k)$. To make this rigorous, we need to determine the coefficient.

\textit{ Part 1: Controlling the contribution of the higher order parameters}

Observe that the coefficients of $P[(P^{-1} - G^{-1}) P]^k$ can be bounded in absolute value by,
\begin{equation} \label{eq:kthordersum}
\sum_{v_1,\ldots, v_k} |P_{\hat{r} v_1}| |P_{v_1 v_2}||P_{v_2 v_3}|\ldots |P_{v_k \hat{r}}|,
\end{equation}
where the points $v_i$ are on the boundary of the graph $G$. This is too difficult to sum in each variable $v_k$ independently, so we introduce a new decomposition.

\textbf{Step 1: An alternative decomposition}

If we are considering the product $P_{\hat{r}v_1} P_{v_1 v_2}\ldots P_{v_k r}$, we can naturally associate this product to the subgraph of the tree consisting of all the unions of paths between $v_i$ to $v_{i+1}$. Since each term $P_{v_{i}v_{i+1}}$ is a function of the edges of the path between them, knowing the form of this subgraph is sufficient to bound the value of the desired product. We remark that this is the main point of Lemma \ref{lem:Factorization}. For example, we see that we can bound,
$$
|P_{\hat{r}v_1}| \le \prod_{i=1}^{k_1-1} K(\{e^1_i,e^1_{i+1}\}).
$$ 


Notice that, depending on the specific orientation of the edges, this path will have varying factors of $X$ and $Y$. Naively applying a worst case bound by using $X,Y \le \max(X,Y)$ will lead to a bound of $\sum_{v_1,\ldots, v_k} |P_{\hat{r} v_1}| |P_{v_1 v_2}||P_{v_2 v_3}|\ldots |P_{v_k \hat{r}}|$ that is useless. Instead, one must use the condition $X+Y \le 1$ by grouping together terms that have similar structures, but roughly exchange the roles of $X$ and $Y$. This grouping is formally encoded by finding all terms that map to the same basal tree structure.


To make this procedure clearer,  we will first consider the case where we only have to deal with two vertices $v_1$ and $v_2$. The construction used here is the linchpin of our method. Trees with more edges are constructed inductively. 


 \textbf{Two Vertex Case:} We let $A_1$ be the last common ancestor of $v_1$ and $v_2$. Pictorially,  to arrive at $v_2$ from $v_1$, we first proceed along the path $p^1$ to $A_1$. From the point $A_1$ we start branching out to the  point $v_2$ on the boundary; we let $p^2 = e_1^2,e_2^2,\ldots, e_{k_2}^2$ be the path from $A_1$ to $v_2$. If the path from $A_1$ to $v_1$ is $e^1_{b_1}, e^1_{b_1+1},\ldots, e^1_{k_1}$, then we see that the value of $|P_{v_1 v_2}|$ can be bounded as follows:

 $$ |P_{v_1 v_2}| \le C \prod_{i=b_1}^{k_1-1}K_{e^1_{i},e^1_{i+1}} \prod_{j=1}^{k_2-1} K_{e^2_j, e^2_{j+1}}$$. 

The path back from $v_2$ to $\hat{r}$ will first traverse $v_2$ to $A_1$ and then $A_1$ to $\hat{r}$. We can visualize this as a star with three edges.  This star is the tree associated with our subgraph. $A_1$ is the center of this star and it is connected to $v_1$,$v_2$ and $\hat{r}$. If we let $K_{A_1,\hat{r}/v_i}$ be a shorthand for the product of $K$ factors between $A_1$ and the respective vertex in question, then we have that,
\begin{equation}
|P_{\hat{r}v_1}| \le C K_{A_1 \hat{r}} K_{A_1 v_1}, \quad
|P_{v_1 v_2}| \le C K_{A_1 v_1} K_{A_1 v_2}, \quad  |P_{r v_2}| \le C K_{A_1 v_1} K_{A_1 v_2}.
\end{equation}

Thus, we would have that,
\begin{equation}
\sum_{v_1,v_2} |P_{\hat{r}v_1}| |P_{v_1 v_2}| |P_{v_2 \hat{r}}| \le \sum_{A_1,v_1,v_2} C^3 K_{A_1 \hat{r}}^2 K_{A_1 v_1}^2 K_{A_1 v_2}^2.
\end{equation}

\textbf{Step 2: The summation along the tree}

The purpose of the tree structure is to introduce the new summation variables $A_1\ldots A_{k-1}$. This also introduces a new way to sum along the edges. As we mentioned before, we use Lemma \ref{lem:Factorization} the fact that the products $|P_{v_i v_{i+1}}|$ decomposes as a product of $K$ factors along the edges of the path connecting $v_i$ and $v_{i+1}$. 

Assume that the path from $v_i$ to $v_{i+1}$ in the tree $T$ constructed earlier involves edges $\mathfrak{e}_1,\ldots, \mathfrak{e}_k$ of the tree $T$. Recall that any edge $\mathfrak{e}$ of $T$ corresponds to a path $p = (e_1,e_2,e_3,\ldots,e_m)$ of edges on our original graph. Given an edge $\mathfrak{e}$ of $T$ that corresponds to a path $p=(e_1,e_2,e_3,\ldots, e_m)$ in our original graph,  we let $K_{\mathfrak{e}}= \prod_{i=1}^{m-1} K_{e_i,e_{i+1}}$. We thus have that,
\begin{align*}
 |P_{v_i,v_{i+1}}| &\le C^k \prod_{j=1}^k P_{\mathfrak{e}_j},\\
 \prod_{i=0}^k |P_{v_i v_{i+1}}| &\le  C^{4k} \prod_{\mathfrak{e} \in T} K^2_{\mathfrak{e}},\\
 \sum_{v_1,v_2,\ldots,v_k} \prod_{i=0}^k |P_{v_i v_{i+1}}| &\le  \sum_{\mathcal{T}} \sum_{T \in \mathcal{T}}  C^{4k} \prod_{\mathfrak{e} \in T} K^2_{\mathfrak{e}}
\end{align*}
where $v_0$ and $v_{k+1}$ are both the root $\hat{r}$.
The constant $C^{2k}$ comes since we get a factor of at most $C$ each time we introduce a factor of the form $K_{\mathfrak{e}}$ for each edge $\mathfrak{e} \in \mathbb{R}$. Each edge $\mathfrak{e}$ will appear at most two times.  In the last line $\mathcal{T}$ represents all possible tree structures that could occur, while $T \in \mathcal{T}$ is a sum over all actual trees that have said tree structure.

Fix a tree structure $\mathcal{T}$; this will have a leaf edge $l$, and we can represent $\mathcal{T}= \mathcal{T}' \cup l$ where $\mathcal{T}'$ is the tree structure with the leaf edge removed. Then, each tree $T$ in $\mathcal{T}$ can be decomposed as the union of a tree $T'$ of the form $\mathcal{T}'$ and some path $p$ that will project down to the edge $l$.
\begin{equation}
\begin{aligned}
\sum_{T \in \mathcal{T}}  \prod_{\mathfrak{e} \in T} K^2_{\mathfrak{e}} &= \sum_{T' \in \mathcal{T}'}  \prod_{\mathfrak{e} \in T'} K^2_{\mathfrak{e}} \sum_{p \to l}K^2_{l}  
 \le (1+\text{diam}(\GG)) \sum_{T' \in \mathcal{T}'}  \prod_{\mathfrak{e} \in T'} K^2_{\mathfrak{e}} 
\end{aligned}
\end{equation}

The last line sums up $K^2_l$ over all paths $p$ that could possibly project down to $l$. For this quantity, the bound $(1+\text{diam}(\GG))$ is an application of  Lemma \ref{lem:factsum}. By iteration,  since there are $2k-1$ leaves, one can see that upon fixing the tree, one can derive a bound for the tree sum that is $\lesssim (\text{diam}(\GG)+1)^{2k-1}$. All that is left is to compute the sum over all possible tree structures $\mathcal{T}$. 

Since, we know that the distance between any two vertices $v_i$ and $v_{i+1}$ is less than $\text{diam}(\GG)$; we use the very generous bound that there are no more than $1+ \text{diam}(\GG)$ ways to chose the position of $A_i$ once the structure of the vertices from $\{v_0,\ldots,v_i\} \cup \{A_1,\ldots,A_{i-1}\}$ has been decided. This shows that there are no more than $(1 + \text{diam}(\GG))^{k-1}$ possible tree structures to consider. In total, this bounds the sum appearing in \eqref{eq:kthordersum} by $\lesssim (1 + \text{diam}(\GG))^{3k}$.

\textit{Part 2: Deriving the correct coefficient structure of the first order  term}

We have more exact expressions for the coefficient of the first order terms.
\begin{equation} \label{eq:firstorder}
\sum_{v \in \delta(G)} \frac{d- d_{\text{out}}(v)}{d-1} P_{\hat{r},v}  P_{v,\hat{ r}} ( Q_I - m_{\infty}) + \sum_{v \in \delta(G)} \frac{d - d_{\text{in}}(v)}{d-1} P_{\hat{r}, v+\aleph} P_{v+\aleph, \hat{r}} (Q_O - m_{\infty}).
\end{equation}
Here, $\delta(G)$ consists of the vertices that lie on the boundary of our graph $G$.

We can imbed our original graph $G$ into a tree $T_{\text{diam}}$ of depth $\text{diam}(G)$.
For each vertex $v$, we let $S^{out}_v$ denote the  vertices of $T_{\text{diam}}$ that are children of $v$ from an out-edge from $v$, while $S^{in}_v$ denotes the vertices of $T_{\text{diam}}$ that are children of $v$ from an in-edge of $v$.

An alternative way to write the expression in \eqref{eq:firstorder} is as,
\begin{equation} \label{eq:firstorderdecomp}
\begin{aligned}
&\sum_{v \in \delta(G)}  \Big[\frac{d- d_{\text{out}}(v)}{d-1} P_{\hat{r},v}  P_{v, \hat{r}}[Q_I - m_{\infty}] - \sum_{w \in S^{out}_v} P_{\hat{r},w} P_{w,\hat{r}}[Q_I - m_{\infty}] - \sum_{w \in S^{out}_v} P_{\hat{r},w+\aleph} P_{w+\aleph,\hat{r}} [Q_O - m_{\infty}] \Big]\\
&+ \sum_{v \in \delta(G)}   \Big[\frac{d- d_{\text{in}}(v)}{d-1} P_{\hat{r},v}  P_{v, \hat{r}}[Q_I - m_{\infty}] - \sum_{w \in S^{in}_v} P_{\hat{r},w} P_{w,\hat{r}}[Q_I - m_{\infty}] - \sum_{w \in S^{in}_v} P_{\hat{r},w+\aleph} P_{w+\aleph,\hat{r}} [Q_O - m_{\infty}] \Big]\\
& + \sum_{ v \in \delta(T_{\text{diam}})} \frac{d - d_{\text{out}}(v)}{d-1} P_{\hat{r},v} P_{v,\hat{r}} (Q_I - m_{\infty}) +  \frac{d - d_{\text{in}}(v)}{d-1} P_{\hat{r},v+\aleph} P_{v+\aleph,\hat{r}} (Q_O - m_{\infty}).
\end{aligned}
\end{equation}

The main point is that the last line can be computed exactly via a standard recursion, while the terms in the first two lines should correspond to a smaller order error term with the correct singularity we have to consider. We now treat bounding the last two lines. Indeed, for $w \in S^{out}_v$, we can first try to write $P_{\hat{r} w}P_{w \hat{r}}$ as $P_{\hat{r} v} P_{v \hat{r}} (G_{\TT_1})_{ vw} (G_{\TT_1})_{wv}$. Here, we treat $v$ as the root of a new tree of type $\TT_1$ and used the fact that the product $P_{\hat{r}w}$ splits as the product $P_{\hat{r}v}$ times this $(G_{\TT_1})_{vw}$.  The main benefit now is the fact that we have explicit ways to compute $\sum_{w \in S^{out}_v} (G_{\TT_1})_{ vw} (G_{\TT_1})_{wv}$.

Namely, the matrix $\begin{bmatrix} Y & X \\ X & Y \end{bmatrix}$ satisfies the property that we can write
 $$ \Big[\frac{d- d_{\text{out}}(v)}{d-1} P_{\hat{r},v}  P_{v, \hat{r}}[Q_I - m_{\infty}] - \sum_{w \in S^{out}_v} P_{\hat{r},w} P_{w,\hat{r}}[Q_I - m_{\infty}] - \sum_{w \in S^{out}_v} P_{\hat{r},w+\aleph} P_{w+\aleph,\hat{r}} [Q_O - m_{\infty}] \Big]$$
as,
\begin{equation} \label{eq:matrixcalc}
\frac{d- d_{\text{out}}(v)}{d-1} P_{\hat{r},v} P_{v,\hat{r}} \begin{bmatrix}
1 & 0
\end{bmatrix}
\Big( I  - \begin{bmatrix}
Y & X \\
X & Y
\end{bmatrix}^{\text{diam}(\GG) - \text{dist}(v,\hat{r})} \Big) \begin{bmatrix}
Q_I - m_{\infty}\\
Q_O - m_{\infty}
\end{bmatrix}.
\end{equation}

Now, we can diagonalize the matrix $ \begin{bmatrix}
Y & X \\
X & Y
\end{bmatrix}$
and show that we can rewrite \eqref{eq:matrixcalc} as,
\begin{equation}
\frac{d- d_{\text{out}}(v)}{d-1} P_{\hat{r},v}  P_{v, \hat{r}}\left[ \frac{1}{\sqrt{2}}[ 1- (X+Y)^k] Q_{+} + \frac{1}{\sqrt{2}} [ 1- (Y-X)^k] Q_{-} \right].
\end{equation}
In absolute value, this can be bounded by,
\begin{equation}
\frac{d - d_{\text{out}(v)}}{d-1}|P_{\hat{r},v}||P_{v,\hat{r}}| (1+\text{diam}(\GG)) \left[ |1- (X+Y)| |Q_{+}| + |1- (Y-X)| |Q_{-}|  \right]
\end{equation}
Finally, we see that,
\begin{equation}
\begin{aligned}
&\sum_{v} \frac{d - d_{\text{out}(v)}}{d-1}|P_{\hat{r},v}||P_{v,\hat{r}}| (1+\text{diam}(\GG)) \left[ |1- (X+Y)| |Q_{+}| + |1- (Y-X)| |Q_{-}|  \right]\\
&\le O((1+ \text{diam}(\GG)^2)\left[|1-(X+Y)| |Q_{+}| + |1-(Y-X)||Q_{-}| \right]
\end{aligned}
\end{equation}
We can treat the second line of \eqref{eq:firstorderdecomp} in a similar way.

Combining these error estimates gives us the desired result \eqref{eq:stabilityestimate}.
\end{proof}

Our next lemma gives us a perturbation analysis of $G_{ij}(Ext(\GG,Q_I,Q_O)) $ with respect to 

\noindent$G_{ij}(Ext(\GG,m_{\infty},m_{\infty}))$.
\begin{lemma} \label{lem:pert}
 Let $\GG$ be a graph with excess of at most $1$. Assume further that $\text{diam}(\GG)^3[|Q_I - m_{\infty}| + |Q_O - m_{\infty}|]\ll 1$. Then, we have, 
\begin{equation} 
\begin{aligned}
&|G_{ij}(Ext(\GG,Q_I,Q_O)) - G_{ij}(Ext(\GG,m_{\infty},m_{\infty}))| \\ &\lesssim \left(\frac{1}{d-1}\right)^{\text{dist}(i,j)}(1+ \text{dist}(i,j)) [|Q_I - m_{\infty}| + |Q_O - m_{\infty}|]
\end{aligned}
\end{equation}
\end{lemma}

\begin{proof}
As a shorthand, we let $P_{ij}$ denote the Green's function $G_{ij}(Ext(\GG,m_{\infty},m_{\infty})$, while $G_{ij}$ will denote the Green's function of $G_{ij}(Ext(\GG,Q_I,Q_O))$. We can determine the value of $|G_{ij} -P_{ij}|$ via the resolvent formula. Let $l_{\alpha}$ denote the indices corresponding to vertices on the boundary of the graph $\GG$. We have the following,
\begin{equation}
|G_{ij} -P_{ij}| = \sum_k \sum_{l^1,\ldots,l^k} P_{i,l^1} \prod_{m=1}^k [Q_{j(l^m)} -m_{\infty}] P_{{l^m},l^{m+1}},
\end{equation}
where $j(l) = I$ or $O$
if $l^m$ is of the form $v+ \aleph$ or $v$, respectively. We use the convention that, upon fixing $k$,  the final $l$ index $l^{k+1}$ is the same as $\hat{r}$. 

As in the proof of Lemma \ref{lem:stability}, our biggest issue is to determine the value of the $k$ product,
\begin{equation}\label{eq:Exkdef}
Ex_k(i,j):=\sum_{l^1,\ldots,l^k} |P_{i,l^1}|\prod_{m=1}^k|P_{l^m.l^{m+1}}|.
\end{equation} Now, we cannot directly apply the analysis detailed in the aforementioned lemma, as the proof there assumed a purely tree-like structure.

What we have do instead is appeal to the fact that our graph has a tree-like covering and relate the values of $P_{l^i,l^j}$ to the Green's functions of the covering. Now, the results of Lemma \ref{lem:excess1} show that, up to a constant factor, $|P_{l^i ,l^{i+1}}|$ can be bounded by $2 \max(|P_{\tilde{l}^i,\tilde{l}^{i+1}_1}| , |P_{\tilde{l}^i, \tilde{l}^{i+1}_2}|)$. Here, $\tilde{l}^i$ is some chosen covering of $l^i$ and $\tilde{l}^{i+1}_1$ and $\tilde{l}^{i+1}_2$ are the two closest chosen coverings of $l^{i+1}$ to $\tilde{l}^i$.
Thus, when we consider the lift to the covering tree,
$$
 \sum_{l^1,\ldots,l^k} |P_{i,l^1}| \prod_{m=1}^k |P_{l^m,l^{m+1}}|,
$$
can be upper bounded by the sum
$$
2^{k+1}\sum_{\substack{ \text{dist}(\tilde{l}^n,\tilde{l}^{n+1}) \le 3  \text{ diam}(\GG)\\  \tilde{l}^{k+1} \text{ covers } j}} |P_{\tilde{i},\tilde{l}^1}| \prod_{m=1}^k |P_{\tilde{l}^m, \tilde{l}^{m+1}}|.
$$
Here, $\tilde{i}$ is a fixed lift of $i$.
The restriction that $\text{dist}(\tilde{l}^n, \tilde{l}^{n+1}) \le 3   \text{ diam}(\GG)$ comes from the fact that the closest and the second closest cover of any point differ by distance at most $|C| \le 2 \text{ diam}(\GG)$. To perform the above summation, we first fix the final cover of the point $j$, and then compute the above sum in the same manner as in the proof of Lemma \ref{lem:stability}.

Namely, we decompose the product $|P_{l^m,l^{m+1}}|$ into its $K$ factors, introduce the least common ancestor variables, and construct the associated tree. The only difference with the constructed tree is that the $K$ factors associated with the vertices appearing on the path $p_{ij}$ between the lift of $i$ and $j$ may not appear twice.  To deal with these vertices, we remove the edges in the tree that correspond to vertices that are connected via this path. The remaining forest graph can be bounded via the same methods of the proof of Lemma \ref{lem:stability} and will give the same bound.

We only need to reintroduce the $K$ factors that we did not consider on the path between $i$ and $j$. Observe that at least $dist(\tilde{i},\tilde{j}) - 3k$ of the vertices have $K$ factors associated with them, since each factorization from \eqref{eq:fact} can lose no more than $3$ natural $K$ factors.
 Once can see that the $K$ factors corresponding to this path between $i$ and $j$ will give a factor bounded by $\left(\frac{1}{\sqrt{d-1}}\right)^{\text{dist}(\tilde{i},\tilde{j})} (\sqrt{d-1})^{3k}$. We see that the last $(\sqrt{d-1})^{3k}$ can be absorbed by $|Q_I - m_{\infty}|^k$. This latter factor can allow us again to sum over the possible lifts of $\tilde{j}$, since the successive lifts of $\tilde{j}$ increase by distance $|C|$ from the lift $\tilde{i}$. This allows the term  $\left(\frac{1}{\sqrt{d-1}}\right)^{\text{dist}(\tilde{i},\tilde{j})}$ to be summable over successive lifts.

Now, if we want to be careful, we can improve the first order coefficient from $(1+ \text{diam}(\GG))^3$ to $1 + \text{dist}(i,j)$. The method of proof is similar, one only needs to more careful in the discussion of branching. First, consider the path $p_{ij}$ between the lifts of $i$ and $j$. To choose $\tilde{l}_1$, one first picks a point $Q_i$ on the path $p_{ij}$ and then starts branching from there. There are $1 + \text{dist}(\tilde{i},\tilde{j})$ choices for the point $Q_i$. Now, if one takes into account that $\tilde{l}_1$ is the lift of a boundary point, one can use the superior analysis of the first order coefficient from Lemma \ref{lem:stability} to get rid of the factor of $(1+ \text{diam}(\GG))^2$ and replace it with an order $1$ coefficient. This gives us our final answer.
\end{proof}

Because of the particular dependence of pointwise estimates on Green's functions from $i$ to $j$, the worst case bound on pointwise Green's function are generally worse than the averaged case bound. By a similar analysis, one can derive average case bounds.

\begin{cor} \label{cor:AverageCase}
Consider the setting of the previous Lemma \ref{lem:pert}. Assume also that we have the improved bound $(\log N)^9[|Q_I - m_{\infty}|+|Q_O - m_{\infty}|] \ll 1$. Let $\hat{r}$ be a distinguished vertex of $\GG$, which we will call the `root'. We have that,
\begin{equation}
\sum_{x \in \delta(\GG)} \sum_{o_x}|G_{\hat{r}+o_r,x+o_x}(Ext(\GG,Q_I,Q_O))|^2 | \lesssim 1.
\end{equation}
Here, $o_r$ and $o_x$ are allowed to be either $0$ or $\aleph$.
\end{cor}
\begin{proof}
Without loss of generality we may assume that $o_r=0$.

Let $P_{\hat{r}x}:= G_{\hat{r}x}(Ext(\GG,m_{\infty},m_{\infty}))$ and use $G_{rx}$ as a shorthand for the Green's function

\noindent $G_{\hat{r}x}(Ext(\GG,Q_I,Q_O))$. We have that,
\begin{equation}
\sum_{x \in \delta(\GG)} \sum_{o_x} |G_{\hat{r},x+o_x}|^2 \le 2 \sum_{x \in \delta(\GG)} \sum_{o_x}[|P_{\hat{r},x+o_x}|^ 2+|G_{\hat{r},x+o_x} -P_{\hat{r},x+o_x}|^2].
\end{equation}

Now, we have the correct bound for
$\sum_{x \in \delta(\GG)} \sum_{o_x}|P_{\hat{r},x+o_x}|^ 2$ by Corollary \ref{cor:firstcor}.
We perform the perturbation expansion for $G_{\hat{r},x+o_x} -P_{\hat{r},x+o_x}$.

Recall $Ex_k$ from equation \eqref{eq:Exkdef}.
We have,
\begin{equation}
\begin{aligned}
|G_{\hat{r},x+o_x} -P_{\hat{r},x+o_x}|^2& \le \Big(\sum_{k=1}^{\infty} Ex_k(\hat{r},x+o_x) [|Q_I - m_{\infty}| + |Q_O - m_{\infty}]^k\Big)^ 2\\
& \le \sum_{k=1}^{\infty} Ex_k(\hat{r},x+o_x)^2 (2[|Q_I - m_{\infty}| + |Q_O - m_{\infty}|])^{2k} \sum_{k=1} 2^{-2k}.
\end{aligned}
\end{equation}
Now, we see our mail goal is to understand the asymptotics of the averaged sum,
\begin{equation} \label{eq:avgsum}
\sum_{x} \sum_{o_x} Ex_k(\hat{r},x+o_x)^2 \le (1+ \text{diam}(\GG))^{k} \sum_{l^1,\ldots,l^k,l^{k+1}} |P_{i l^1}|^2 \prod_{m=1}^k |P_{l^m, l^{m+1}}|^2
\end{equation}

Now, we can write,

\begin{equation}
\sum_x \sum_{o_x} Ex_k(\hat{r},x+o_x)^2 = \sum_{l^1,\ldots, l^k, x,o_x, \hat{l}^1,\ldots,\hat{l}^k} |P_{i,l^1}||P_{i.\hat{l}^1}| \prod_{m=1}^{k-1}|P_{l^m,l^{m+1}}||P_{\hat{l}^m,\hat{l}^{m+1}}| |P_{l^k, j}| |P_{\hat{l}^k,\hat{j}},
\end{equation}
However, one can relabel $x+o_x$ as a new summation variable $l^{k+1}$ and relabel $\hat{l}^j$ to $l^{k+1 + (k-j)+1}$ to understand this as $Ex_{2k+1}(\hat{r},\hat{r})$, which we can bound by $\lesssim(1 + \text{diam}(\GG)^{6k+3}$. Under our assumption that $(1 + \text{diam}(\GG))^9[|Q_I - m_{\infty}| +|Q_O - m_{\infty}|] \ll 1$, this shows that the sum in \eqref{eq:avgsum} is $\ll 1$. This gives us our desired result.
\end{proof}

The stability result detailed above shows, as one would expect, that the approximation of $\text{Ext}_i(\GG,Q_I,Q_O)$ to the infinite tree would get better as $\GG$ becomes closer to the tree. This is an essential input to our next lemma, which details how the Green's functions change as the graph neighborhood around two points change.

\begin{lemma} \label{lem:nbdpert}
Let $\GG$ be a graph containing two vertices $i$ and $j$. Let $\mathcal{H}$ be a subgraph of $\GG$ that contains the radius $r$ neighborhood of $i$ and $j$.
 Then, we have the following estimate,
\begin{equation}
\begin{aligned}
&|G_{ij}(Ext(\GG, Q_I,Q_O)) - G_{ij}(Ext(\mathcal{H},Q_I,Q_O))| \le  \Big(\frac{1}{d-1} \Big)^r\\
&\hspace{1cm}  +(1+ \text{diam}(\GG)^2 [S_g^1 |Q_I + Q_O - 2 m_{\infty}(z,w)| + S_g^2 |Q_I - Q_O|]\\& \hspace{1 cm} + (\log N)^2 [|Q_I - m_{\infty}|^2 + |Q_O - m_{\infty}|^2],
\end{aligned}
\end{equation}
provided we have that $(\log N)^9 |Q_{I/O} - m_{\infty}| \ll 1$.
\end{lemma}
\begin{proof}
We can write the adjacency matrices of $Ext(\GG,Q_I,Q_O)$ as,
\begin{equation}
\begin{bmatrix}
A & B\\
B^*& C
\end{bmatrix},
\end{equation}
Here, $A$ is the adjacency matrix for the subgraph $\mathcal{H}$. $B$ contains information on how $\mathcal{H}$ connects to the complement and $C$ is the adjacency information of the complement.

By the Schur complement formula, we have,
\begin{equation}
G_{ij}(Ext(\GG,Q_I,Q_O)) = [( A - B  C^{-1} B^*)^{-1}]_{ij}.
\end{equation}

At this point, we can apply the resolvent estimates,
\begin{equation}
\begin{aligned}
&|G_{ij}(Ext(\GG,Q_I,Q_O)) - G_{ij}(Ext(\mathcal{H},Q_I,Q_O))| \\
&\hspace{1 cm}= \sum_{a,b,c,d} G_{ia}(Ext(\GG,Q_I,Q_O)) B_{ab} (C^{-1}_{bc} - \delta_{b,c} Q_{k}) B^*_{cd} G_{dj}(Ext(\mathcal{H},Q_I,Q_O))
\end{aligned}
\end{equation}
Here, $Q_k= Q_I$ if $b$ is of the form $v + \aleph$ for a vertex $v$ in the complement of $\mathcal{H}$ while $Q_k= Q_O$ if $b$ is of the form $v$ for a vertex $v$ in the complement of $\mathcal{H}$. Let us first consider the off-diagonal terms with $b \ne c$.

In order to get non-zero terms,
we see from the above that $a$ and $d$ must be vertices on the boundary of $\mathcal{H}$, and are thus distance greater than $r$ from $a$ and $b$. Furthermore, $b$ and $c$ must be vertices on the boundary of $\GG \setminus \mathcal{H}$ that are adjacent to vertices in $\mathcal{H}$. Our combinatorial assumption on the graph $\GG$ allows us to make the following assertion:  there are only finitely many pairs of vertices $(b,c)$ with $b \ne c$ such that $C^{-1}_{bc} \ne 0$.

 Just to give a summary for the reasoning behind this point, we remark that $C^{-1}$ is a diagonal matrix on the connected components of $\GG\setminus \mathcal{H}$. Now, the fact that $\GG$ has at most one cycle will show that all but one connected neighborhood of $\GG \setminus \mathcal{H}$ is purely treelike. Furthermore, the only connected neighborhood that is not tree-like can contain no more than $O(1)$ many vertices that like on the boundary of $\GG\setminus \mathcal{H}$ lest there be too many cycles. For this $O(1)$ many off-diagonal entries, we can apply Lemma \ref{lem:pert} to show that the diagonal entry $|C^{-1}_{bc}|$ itself is $O(1)$.

We also have the Green's function bound $|B_{ab}| \le \left(\frac{1}{\sqrt{d-1}}\right)^{r}$ from Lemma \ref{lem:pert}. Thus, our bound on the off-diagonal terms is $O\left( \frac{1}{\sqrt{d-1}}\right)^{2r}$.

It is now left to understand the contribution of the terms that are on the diagonal.
Due to Lemma \ref{lem:stability}, we can argue that the size of the differences $|C^{-1}_{bc} - Q_k| \lesssim (\log N)^2 [|S_g^1| |Q_I + Q_O - 2 m_{\infty}(z,w)| + |S_g^2| |Q_I - Q_O|] + (\log N)^3 [|Q_I - m_{\infty}|^2 + |Q_O - m_{\infty}|^2]$. Now, given $b$, the only entries $B_{ab}$ and $B_{bd}$ that are non-zero correspond to vertices adjacent to $b$.

Furthermore, Corollary \ref{cor:AverageCase} can show that the contribution $\sum_{a}|G_{ia}||G_{aj}| \lesssim 1$. One can only gain a constant factor when considering $\sum_{a,b,d}|G_{ia}||B_{ab}||B_{bd}^*||G_{dj}|$. This shows that the contribution of the off-diagonal entries is
$$
\lesssim (\log N)^2 [|S_g^1| |Q_I + Q_O - 2 m_{\infty}(z,w)| + |S_g^2| |Q_I - Q_O|] + (\log N)^3 [|Q_I - m_{\infty}|^2 + |Q_O - m_{\infty}|^2]
$$
\end{proof}

\subsection{Stability Bounds}

\begin{thm} \label{thm:Self-conc}

Consider the following pairs of equations,
\begin{equation}
\begin{aligned}
& Q_I= Y_{i,K}(Q_I,Q_O) + E_I\\
& Q_O = Y_{i,K}(Q_I, Q_O) + E_O,
\end{aligned}
\end{equation}

Assume further that it is known that  $|Q_{I,O}- m_{\infty}|^2 \lesssim \frac{ S^1_g |E_O|}{(\log N)^4}$, and $N^{\delta}|Q_{I,O} - m_{\infty}| \ll 1$ for some $\delta>0$.
Recall the singularity functions from Definition \ref{def:Singularities}.

Fix some parameter $\mathfrak{a}>0$.
There exists $K$ in the domain $[\mathfrak{a} \log_{d-1} \log N, 2 \mathfrak{a} \log_{d-1} \log N]$ such that, we have the stability estimates,
\begin{equation}
\begin{aligned}
&|Q_I - Q_O| \le \frac{|E_I| + |E_O| }{|S_g^2|} + O \left(\frac{|E_I| +|E_O|}{\log N} \right)\\
& |Q_I + Q_O - 2 m_{\infty}| \le  \frac{|E_I| + |E_O|}{|S_g^1|} + O \left(\frac{(|E_I|+ |E_O|)}{\log N} \right)
\end{aligned}
\end{equation}

\end{thm}



\begin{proof}

This is mainly a consequence of perturbation theory using the resolvent identity.
$$
(A-z)^{-1} = (B-z)^{-1} + \sum_{p=1}^{\infty} (B-z)^{-1} [(B-A)(B-z)^{-1}]^p.
$$

We apply this identity to the compute the difference of the resolvents,
$$
Y_{i,K}(Q_I,Q_O) - Y_{i,K}(m_{\infty},m_{\infty})= G(\TT^1_K,Q_I,Q_O)_{\hat{r}+\aleph,\hat{r}+\aleph} - G(\TT^1_K,m_{\infty}, m_{\infty})_{\hat{r}+\aleph,\hat{r}+\aleph}.
$$

The main benefit is that the second quantity in the line above can be explicitly computed; these are the Green's function entries for the infinite tree.

Using that $ N^{\delta} |Q_I - m_{\infty}|, N^{\delta}|Q_O - m_{\infty}| \ll 1$ for some $\delta>0$ and furthermore that we have the a-priori bound that $|Q_I - m_{\infty}|^2 \ll \frac{S^1_g}{(\log N)^4}[|E_I|+|E_O|]$, then it would suffice to consider only the first order term in the resolvent expansion.



By applying the resolvent identity and the more detailed computations of Lemma \ref{lem:stability}, we have that,
\begin{equation}
\begin{aligned}
&G(\TT^1_K,Q_I,Q_O)_{r+\aleph,r+\aleph} - G(\TT^1_K,m_{\infty},m_{\infty})_{\hat{r}+\aleph,\hat{r}+\aleph} \\
&=  \frac{d-1}{d-1} (Q_I - m_{\infty}) \sum_{v \in V^K_{O}}  G_{\TT_1}(z,w)_{\hat{r}+\aleph,v} \overline{G_{\TT_1}(\overline{z},w)_{\hat{r}+\aleph,v}}  \\
&+ \frac{d}{d-1} (Q_O - m_{\infty}) \sum_{v \in V^K_O} G_{\TT_1}(z,w)_{\hat{r}+\aleph,v+\aleph} \overline{G_{\TT_1}(\overline{z},w)_{\hat{r}+\aleph,v+\aleph}} \\
&+ \frac{d}{d-1} (Q_I -m_{\infty}) \sum_{v \in V^K_{I} }  G_{\TT_1}(z,w)_{\hat{r}+\aleph,v} \overline{G_{\TT_1}(\overline{z},w)_{\hat{r}+\aleph,v}}\\
&+ \frac{d-1}{d-1}(Q_O - m_{\infty}) \sum_{v \in V^K_I} G_{\TT_1}(z,w)_{\hat{r}+\aleph,v+\aleph} \overline{G_{\TT_1}(\overline{z},w)_{\hat{r}+\aleph,v+\aleph}}\\
&+ O((\log N)^3(|Q_I - m_{\infty}|^2 + |Q_O - m_{\infty}|^2))
\end{aligned}
\end{equation}
Here, $V^K_O$ are the leaves of the tree $\TT^1_K$ whose only edge is an outgoing edge (so it has $d-1$ out-edges missing) while $V^K_I$ are those leaves whose only edge is an ingoing edge (so it has $d$ out-edges missing). When going to the second line, we used the fact that $G(\TT^1_K,m_{\infty},m_{\infty})$ are equal to the Green's functions of the infinite tree $\TT_1$. We also used the fact that $G_{ab}(z,w)= \overline{G_{ba}(\overline{z},w)}$.

Provided we know that the coefficient matrix corresponding to the first order terms are sufficiently well controlled, we can derive stability estimates on $Q_i - m_{\infty}$ by merely inverting the coefficients.  We can derive the first order coefficients by a direct recursion as follows:consider the root $\hat{r}$ and the unique path $\hat{r} \to v_1 \to v_2 \to \ldots \to v_n$ to some vertex $v_n$ (the path from $\hat{r}$ to $v_n$ does not care about the orientation of the edges.)




If $\hat{r} \to v_1$ is a forward edge, then $G_{\TT_1}(z,w)_{\hat{r} +\aleph, v_n + o_n} = G_{\TT_1}(z,w)_{\hat{r}+\aleph,\hat{r}} \frac{1}{\sqrt{d-1}} G_{\TT_1}(z,w)_{v_1+\aleph,v_n+o_n}$, where $o_n$ can either be $0$ or $\aleph$. If instead $\hat{r} \to v_1$ is a backward edge, then $G_{\TT_1}(z,w)_{\hat{r}+\aleph,v_n+o_n} = G_{\TT_1}(z,w)_{\hat{r}
+\aleph,\hat{r}+\aleph} \frac{1}{\sqrt{d-1}} G_{\TT_2}(z,w)_{v_1,v_n+o_n}$.

Combining these estimates can give us a recursion for the value of the coefficients of the linear term in the matrix.
Indeed, we have the following recursion to determine the coefficients.
\begin{equation}
\begin{aligned}
& \sum_{v \in V^K_O} G_{\TT_1}(z,w)_{\hat{r} +\aleph,v+\aleph} \overline{G_{\TT_1}(\overline{z},w)_{\hat{r} +\aleph,v+\aleph}}\\& = \frac{d}{d-1} G_{\TT_1}(z,w)_{\hat{r}+\aleph, \hat{r}} \overline{G_{\TT_1}(\overline{z},w)_{r+\aleph,r}}\sum_{v \in V^{K-1}_{O}} G_{\TT_1}(z,w)_{\hat{r}+\aleph,v+\aleph} \overline{G_{\TT_1}(\overline{z},w)_{\hat{r}+\aleph,v+\aleph}}\\& + \frac{d-1}{d-1} G_{\TT_1}(z,w)_{\hat{r}+\aleph,\hat{r}+\aleph}(z,w) \overline{G_{\TT_1}(z,w)_{\hat{r}+\aleph,\hat{r}+\aleph}(\overline{z},w)} \sum_{v \in V^{K-1}_{O}} G_{\TT_2}(z,w)_{\hat{r},v+\aleph} \overline{G_{\TT_2}(\overline{z},w)_{r,v+\aleph}} .
\end{aligned}
\end{equation}

We will more generally define the quantity,
\begin{equation}
A_{j}^{K,l, o}:= \sum_{v \in V^{K}_l} G_{T_j}(z,w)_{\hat{r}+o(j), v+o} \overline{G_{T_j}(\overline{z},w)_{\hat{r}+o(j),v+o}}.
\end{equation}
$l$ can be either $I$ or $O$ for in or for out vertices. $K$ is the depth of the nodes we are considering on the tree. $j$ can take the values either $1$ or $2$ (representing either the tree $\TT_1$ or $\TT_2$). $o$ can be either $\aleph$ or $0$ (representing the index of the vertex we are considering).  $o(j)$ is either $\aleph$ or $0$, when $j=1$, then $o(j)=\aleph$, otherwise $o(j)=0$ for $j=2$. $(3-j)$ will denote the opposite $o$ assignment. (Namely, it is $0$ when $j=1$ and $\aleph$ when $j=2$.)

We see that one way to write the recursive equation considered above is as follows:
\begin{equation}
\begin{aligned}
A_{j}^{K,l,o}&=\frac{d}{d-1} G_{T_j}(z,w)_{\hat{r}+o(j), \hat{r} +o(3-j))} \overline{G_{T_j}(\overline{z},w)_{\hat{r}+o(j),\hat{r}+o(3-j)}} A_{j}^{K-1,l,o} \\
&+ \frac{d-1}{d-1} G_{T_{j}}(z,w)_{\hat{r}+o(j),\hat{r}+o(j)}\overline{G_{T_{j}}(\overline{z},w)_{\hat{r}+o(j),\hat{r}+o(j)}} A_{3-j}^{K-1,l,o}
\end{aligned}
\end{equation}

From looking at the infinite tree, we would know that $$\frac{d}{d-1} G_{T_j}(z,w)_{\hat{r}+o(j), \hat{r} +o(3-j))} \overline{G_{T_j}(\overline{z},w)_{\hat{r}+o(j),\hat{r}+o(3-j)}} $$ is a constant in $j$, precisely matching 
$Y$ of \eqref{dfn:XY}. Similarly, $$ \frac{d-1}{d-1} G_{T_{j}}(z,w)_{\hat{r}+o(j),\hat{r}+o(j)}\overline{G_{T_{j}}(z,w)_{\hat{r}+o(j),\hat{r}+o(j)}} $$ is a constant in $j$, precisely matching $X$ of \eqref{dfn:XY}.

We see we have the general recursions,
\begin{equation}
\begin{bmatrix}
A_{1}^{K,l,o}\\
A_{2}^{K,l,o}
\end{bmatrix}=
\begin{bmatrix}
Y & X\\
X & Y
\end{bmatrix}
\begin{bmatrix}
A_1^{K-1,l,o}\\
A_2^{K-1,l,o}.
\end{bmatrix}
\end{equation}
The matrix in the middle has eigenvalues $X+Y$ with eigenvector $\begin{bmatrix} 1\\1 \end{bmatrix}$ and $Y-X$ with eigenvector $\begin{bmatrix} 1\\ -1 \end{bmatrix}$.

In fact, we observe that we have the same recursion matrix for the actual coefficients $\frac{d-1}{d-1} A_1^{K,O, 0} + \frac{d}{d-1} A_1^{K,I,0}$ and $\frac{d}{d-1}A_1^{K,O,\aleph} + \frac{d-1}{d-1} A_1^{K,I, \aleph}$ that we actually need.

Indeed, we see we have the coefficient matrix,
\begin{equation*}\begin{aligned}
&\begin{bmatrix}
& G(\TT^1_K,Q_I,Q_O)_{\hat{r}+\aleph,\hat{r}+\aleph} - G(\TT^1_K,m_{\infty},m_{\infty})_{\hat{r}+\aleph,\hat{r}+\aleph}\\
& G(\TT^2_K,Q_I,Q_O)_{\hat{r},\hat{r}} - G(\TT^2_K, m_{\infty},m_{\infty})_{\hat{r},\hat{r}}
\end{bmatrix}
\\
= & \begin{bmatrix}
\frac{d-1}{d-1} A_1^{K,O,0} + \frac{d}{d-1} A_1^{K,I,0} & \frac{d}{d-1}A_1^{K,O,\aleph} + \frac{d-1}{d-1} A_1^{K,I,\aleph}\\
\frac{d-1}{d-1} A_2^{K,O,0} + \frac{d}{d-1} A_2^{K,I,0} & \frac{d}{d-1}A_2^{K,O,\aleph} + \frac{d-1}{d-1} A_2^{K,I,\aleph}
\end{bmatrix}
\begin{bmatrix}
Q_I - m_{\infty}\\
Q_O - m_{\infty}
\end{bmatrix} + O(|Q_I - m_{\infty}|^2 + |Q_O - m_{\infty}|^2)
\\
=& \begin{bmatrix}
Y & X \\
X & Y
\end{bmatrix}^K
 \begin{bmatrix}
\frac{d-1}{d-1} A_1^{0,O,0} + \frac{d}{d-1} A_1^{0,I,0} & \frac{d}{d-1}A_1^{0,O,\aleph} + \frac{d-1}{d-1} A_1^{0,I,\aleph}\\
\frac{d-1}{d-1} A_2^{0,O,0} + \frac{d}{d-1} A_2^{0,I,0} & \frac{d}{d-1}A_2^{0,O,\aleph} + \frac{d-1}{d-1} A_2^{0,I,\aleph}
\end{bmatrix}
\begin{bmatrix}
Q_I - m_{\infty}\\
Q_O -m_{\infty}
\end{bmatrix}
+ O(|Q-m_{\infty}|^2)
\end{aligned}
\end{equation*}

Now we can apply our initial data conditions. This is the part where our coefficient dependence on the terms $I,O,o$ are seen.  For example, we have $A_1^{0,O,o}=0$ and $A_2^{0,I,o}=0$. In addition, $A_1^{0,I,\aleph} = X$, $A_1^{0,I,0} = \frac{d-1}{d}Y$, $A_2^{0,O,\aleph}= \frac{d-1}{d} Y$, and $A_2^{0, O,0}=X$.  Substituting these values inside the initial data matrix will give us the following equation.

\begin{equation}\label{eq:linearizedselfconsist}
\left(\begin{bmatrix}
1 & 0 \\
0 & 1
\end{bmatrix}-
\begin{bmatrix}
Y & X\\
X & Y
\end{bmatrix}^{K+1}\right)
\begin{bmatrix}
Q_I - m_{\infty}\\
Q_O - m_{\infty}
\end{bmatrix}=
\begin{bmatrix}
E_I + O(|Q- m_{\infty}|^2)\\
E_O + O(|Q- m_{\infty}|^2)
\end{bmatrix}
\end{equation}

If we let $C$ denote the coefficient matrix appearing in front of the vector of $Q_i - m_{\infty}$ in the above equation, then we see that $C$ has eigenvalues $1 - (X+Y)^{K+1}$ and $1- (Y- X)^{K+1}$.

By an appropriate choice of $K$, one can try to ensure that the singularity of these eigenvalues is no worse than the singularity of $1 - (X+Y)$ or $1 - (Y-X)$. 
\end{proof}

\section{Estimates on $G^{(\mt)}$: Proof of  Proposition \ref{prop:removeTest}} \label{sec:step1}

Recall we use $\mt$ to represent the neighborhood of size $\ell$ around a point. 
Our main strategy to prove Proposition \ref{prop:removeTest} would be to apply the Schur complement formula to relate the values of $G$ to $G^{(\mt)}$. To better visualize this, let us first write the adjacency matrix for the graph $\GG$ as follows.
\begin{equation}
G^{-1}=\begin{bmatrix}
A & B\\
B^* & D
\end{bmatrix}
\end{equation}
$A$ is a $2|\mt|$ by $2|\mt|$ matrix, $B$ is a $2|\mt|$ by $2N - 2|\mt|$ matrix, and $D$ is a $2N-2|\mt|$ by $2N-2|\mt|$ matrix.

 If we let  $\{v_1,\ldots,v_{|\mt|}\}$ be the indices of the vertices in $\mt$. Then, we find $A$ by considering the principal submatrix given by the indices $\{v_1,v_1+N,\ldots, v_{|\mt|}, v_{|\mt|}+N\}$.  $B$ is the adjacency graph of the edges connecting $\mt$ to $\mt^c$. $D$ is the adjacency graph of $\mt^c$.

\begin{rmk} [Notational convention for adjacency matrix decompositions]
We remark here for the reader that though we will use the same notation $A$, $B$, and $D$ to denote the upper-left diagonal block, the upper right off-diagonal block, and the lower-right diagonal blocks respectively for some adjacency matrix, the specific adjacency matrices will change based on the section. The exact matrix referred to by the notation $A$, $B$, or $D$ will be clear in context.
\end{rmk}

\begin{rmk}[Notational Convention for Perturbation Theory]

In all of the following few sections, our goal is to understand the Green's function of the Hermitized adjacency matrix of the graph $\GG$. These matrices will be denoted as $G$ or some modifications of $G$ such as $G^{(\mt)}$, or $\tilde{G}$ based on the context; for example, $\tilde{G}$ will correspond to the Green's function of a  switched graph.

These Green's functions will always be compared to a Green's function that will be constructed by understanding the local neighborhood around the vertices. For example, we will compare the Green's function $G_{xy}$ with $G(Ext(B_r(\mt \cup \{x,y\}, \GG),Q_I,Q_O))$. The specific comparison will highly depend on the context, but we will always denote this secondary matrix by $P$ for simplicity of presentation. We warn the reader here that the specific matrix $P$ found in the estimates will not be the same in different sections; however, their role as a characterizing the local neighborhood will always be the same.

\end{rmk}
In this section we let $P$ be a shorthand for $G(Ext(B_{r}(\mt\cup \{i,j\}, \GG) , Q_I,Q_O))$. Since our induction hypothesis is that $|G_{ij} - G(Ext(B_{r}(\{i,j\}, \GG),Q_I,Q_O)) | \lesssim \epsilon$, we can use Lemma \ref{lem:nbdpert} to assert that we would still have $|G_{ij} - P_{ij}| \lesssim \epsilon$. We only used the fact here that we contain the radius $r$ neighborhood around $i$ and $j$. The benefit of using $P_{ij}$ instead of $G(Ext(B_{r}( \{i,j\}, \GG) , Q_I,Q_O)$ is that we have a more convenient adjacency matrix to apply perturbation theory with in $P_{ij}$. In later sections, we will freely adjust the neighborhood to which we consider the extension to be more useful in perturbation theory without any explicitly mention; the justification of this change will always be Lemma \ref{lem:nbdpert}.

As we have mentioned before, the benefit of the matrix $P$ is that we have the adjacency matrix decomposition,
\begin{equation}
P^{-1} = \begin{bmatrix}
A & B\\
B^* & D_0
\end{bmatrix}.
\end{equation}

$A$ is the same adjacency matrix when of the graph restricted to the neighborhood $\mt$. We technically abuse notation when we use the matrix $B$, but the upper-left and lower-right blocks of $P^{-1}$ have the same non-zero elements as those of $G^{-1}$ (these are the elements that connect the vertices of $\mt$ to the complement $\mt^c$.)

Applying the Schur complement formula, we see that,
\begin{equation}
\begin{aligned}
&G^{(\mt)} = G - G (G|_\mt)^{-1}G ,\\
&P^{(\mt)} = P - P(P|_\mt)^{-1}P.
\end{aligned}
\end{equation}

If we take the difference of these expressions, we see that,
\begin{equation} \label{eq:SchurcompOutsideTree}
\begin{aligned}
G^{(\mt)} - P^{(\mt)}& = G-P - (G-P) (G|_\mt)^{-1}G\\
&- P[(G|_\mt)^{-1} - (P|_\mt)^{-1}] G - P (P|_\mt)^{-1} (G-P).
\end{aligned}
\end{equation}

This equation is the basis of the proof of Proposition \ref{prop:removeTest}, which we reproduce here. 
\begin{Prop} \label{Prop:GTminusPT}
Consider a graph $\mathcal{G}$ in $\Omega^0(z,w)$. Fix a vertex $o$ and let $\mathbb{T}$ be the radius $\ell$ neighborhood around $\mathbb{T}$. Assume, in addition, that $o$ has a radius $\mathfrak{R}$ neighborhood with excess at most 1.  Then, we have the following estimates on the Green's function of  the graph $\mathcal{G}$ with the neighborhood $\mathbb{T}$ removed. Let $i$ and $j$ be vertices outside of $\mathbb{T}$ in $\mathcal{G}$ and let $P$ denote $G(Ext(B_r(\mathbb{T}\cup\{i,j\},\mathcal{G}),Q_I,Q_O))$.  Then, we have that, 
\begin{equation}
|G_{ij}^{(\mt)} - P_{ij}^{(\mt)}| \le  \epsilon'(z,w).
\end{equation}
\end{Prop}

Before we can make assertions on the difference between $G$ and $P$, we must first make assertions on the values of $(G|_\mt)^{-1}$. This will be compared to the quantity $(P|_{\mt})^{-1}$; we start with a combinatorial lemma whose goal is to understand $(P|_{\mt})^{-1}$ using the fact that we are performing a computation on a treelike neighborhood.

\begin{lemma}\label{lem:compPT}
We have the following estimates of $(P|_{\mt})^{-1}$. For any index $x$ corresponding to a vertex in $\mt$, we have,
\begin{equation}
	\sum_{y} |(P|_{\mt})^{-1}_{xy}| \lesssim 1.
\end{equation}

\end{lemma}

\begin{proof}

As is standard, we will apply the Schur complement formula. We have,
\begin{equation} \label{eq:firstschur}
(P|_{\mt})^{-1}_{xy} = A_{xy} -  \sum_{a_{\alpha},a_{\beta}}(B^*)_{x a_{\alpha}} [(D_0)^{-1}]_{a_{\alpha},a_{\beta}} B_{a_{\beta}, y}.
\end{equation}
Here, $a_{\alpha}$ and $a_{\beta}$ vary over the indices of vertices in $\mt^c$ that are adjacent to vertices $l_{\alpha}$ on the boundary of $\mt$. We see that the matrix product is trivial unless $x= l_{\alpha}$ and $y= l_{\beta}$ are indices corresponding to some vertices on the boundary of $\mt$.

If we fix $x$ not of the form $l_{\alpha}$ and we sum up over $y$, we do not have to worry about the last term on the \abbr{rhs} of equation \eqref{eq:firstschur}. Furthermore, $A_{xy}$ can only be nonzero for $O(1)$ many terms ( the indices $y$ must either correspond to vertices adjacent to the one corresponding to $x$, or they would correspond to the same index as $y$ itself).

 Let us now consider the case that $x$ corresponds to some vertex $l_{\alpha}$ on the boundary. Since $\mt$ is centered around a neighborhood of excess at most 1, this means that $[(D_0)^{-1}]_{a_{\alpha},a_{\beta}}=0$ when $a_{\alpha} \ne a_{\beta}$, except for at most $O(1)$ many pairs. Recall $D_0$ is a block diagonal matrix with the blocks corresponding to neighborhoods of $\mt^c$.  This further gives only $O(1)$ many options for $y$ to get a nonzero quantity.
\end{proof}

For the sake of completeness, we also present the corresponding combinatorial bound for $|P|_{iy}$.
\begin{lemma} \label{lem:sumintree}
If we let $i$ be a vertex in $\mt^c$, we have the following estimates of $P|_{\mt}$.

\begin{equation}
\sum_{y \in \mt} |P_{iy}| \lesssim (\log N)^{\mathfrak{j}(d) \mathfrak{a}},
\end{equation}
where $\mathfrak{j}(d)= \log_{d-1} \left[ \frac{1}{2} + \frac{1}{2} \sqrt{\frac{2d-1}{d-1}} \right]<1/2$ is a constant that decays to $0$ as $d$ increases to $\infty$.
\end{lemma}
\begin{rmk}
The main difference between the case with $|w| =0$ and $|w| \ne 0$ is that our combinatorial estimates on the tree-like graph now are bounded by $(\log N)$ to a power depending on $\mathfrak{a}$, but the coefficient behind this $\mathfrak{a}$ can be made arbitrarily small by choosing $d$ large, but still finite.
\end{rmk}

\begin{proof}

For this lemma, we have to apply the estimate $|P_{iy}| \lesssim \left(\frac{1}{\sqrt{d-1}}\right)^{\text{dist}(i,y)}$ coming from perturbation theory from the infinite directed $d$-regular graph and as in Lemma \ref{lem:excess1}.

 Let $p$ denote the  shortest path between $i$ and the center $o$ of the neighborhood $\mt$. We let $a_k$ denote the vertex on the path $p$ that is of distance $k$ from $i$. We also let $\mathcal{C}_j$ denote the vertices on $\mt$ that are reachable from $a_k$ without using an edge in $p$. Note that it is possible that some sets $\mathcal{C}_{j_1} =\mathcal{C}_{j_2}$ for $j_1 \ne j_2$, but this is not possible in $\mt$ is purely tree-like.

Now, let us consider the computation of 
$\sum_{y \in \mt} |P_{yj}|$ is tree-like. If $\mt$ were not tree-like, then Lemma \ref{lem:excess1} would show that we would get the same bound up to a constant factor. Since we are only concerned about the correct size up to a constant factor, this is not a major loss for us.

By the factorization lemma \ref{lem:Factorization}, we see that $ \sum_{y \in \mathcal{C}_j} |P_{i y}| \lesssim \max \left \{\sqrt{X}, \sqrt{\frac{d-1}{d}} \sqrt{ Y} \right \}^{ j -1} \sum_{y \in  \mathcal{C}_j} |P_{a_j} y|$.

Now, the computation of $\sum_{y \in \mathcal{C}_j} |P_{a_j y}|$ is like a computation performed if $a_j$ were the root of a $j$ level tree. If we consider the quantities,
\begin{equation}
\begin{aligned}
L_{1,k}:= \sum_{\text{dist}(y,r) = k} |(G_{\mathcal{T}_1})_{r + \aleph, y+ o_y}|,
L_{2,k}:= \sum_{\text{dist}(y,r)=k} |(G_{\mathcal{T}_2})_{r, y + o_y}|,
\end{aligned}
\end{equation}
and $r$ is the root of $\mathcal{T}_1$ and $\mathcal{T}_2$, then we have the recursion,
\begin{equation}
\begin{bmatrix}
L_{1,k}\\
L_{2,k}
\end{bmatrix}=
\begin{bmatrix}
\sqrt{\frac{d}{d-1}}\sqrt{ Y } & \sqrt{X}\\
\sqrt{X} & \sqrt{\frac{d}{d-1}}\sqrt{Y}
\end{bmatrix}
\begin{bmatrix}
L_{1,k-1}\\
L_{2,k-1}
\end{bmatrix}.
\end{equation}

Thus, we see that $L_{1,k}, L_{2,k} \lesssim  \left[\sqrt{\frac{d}{d-1}} \sqrt{Y} + \sqrt{X} \right]^k$.

We thus see that 
$$
 \sum_{y \in \mathcal{C}_j} |P_{i y}| \lesssim  \ell \max \left\{\sqrt{X}, \sqrt{\frac{d-1}{d}} \sqrt{ Y}\right\}^{ j -1} \left[\sqrt{\frac{d}{d-1}} \sqrt{Y} + \sqrt{X} \right]^{j-1}.
$$

Thus,
\begin{equation}
\sum_{y \in \mt} |P_{iy}| \lesssim (\ell)^2 \left\{\sqrt{X}, \sqrt{\frac{d-1}{d}} \sqrt{ Y}\right\}^{ \ell -1} \left[\sqrt{\frac{d}{d-1}} \sqrt{Y} + \sqrt{X} \right]^{\ell-1}.
\end{equation}

We now let $a^2:=X$ and $b^2:=Y$.

We solve the maximization problem
\begin{equation}
\begin{aligned}
\max \Big\{ \sqrt{\frac{d}{d-1}} ab + a^2 :  a^2 + b^2 =1 \Big\}.
\end{aligned}
\end{equation}
We remark here that the other relevant maximization problem would be,
\begin{equation}
\max  \Big\{ ab + \sqrt{\frac{d-1}{d}} b^2 :  a^2 + b^2 =1 \Big\},
\end{equation}
which is manifestly smaller.

By Lagrange muiltipliers, this amounts to maximizing,
\begin{equation}
\sqrt{\frac{d}{d-1}} ab + a^2 - \lambda( a^2 + b^2).
\end{equation}
Taking derivatives in $b$ gives us,
\begin{equation}
a= 2   \lambda \sqrt{\frac{d-1}{d}} b. 
\end{equation}
Taking derivatives in $a$ give us,
\begin{equation}
2a - 2\lambda a + \sqrt{\frac{d}{d-1}}b =0.
\end{equation}
Substituting in our previous relation between $a$ and $b$ gives us the following relationship on $\lambda$.
\begin{equation}
 4 \lambda^2 - 4 \lambda  + 1 - \frac{d}{d-1} -1 = 0.
\end{equation}
Thus, $\lambda$ is given by,
\begin{equation}
\lambda = \frac{1}{2} + \frac{1}{2} \sqrt{\frac{2d-1}{d-1}}.
\end{equation}

We can now substitute this into the relation $a^2 + b^2 =1$ to derive,
\begin{equation}
b = \frac{1}{\sqrt{1 + 4 \frac{d-1}{d} \lambda^2}}, a = \frac{2 \lambda \sqrt{\frac{d-1}{d}}}{\sqrt{1 + 4 \frac{d-1}{d} \lambda^2}}
\end{equation}

The value of  $ \sqrt{\frac{d}{d-1}} ab + a^2$ is now given by,
\begin{equation}
\frac{2 \lambda \left[1 + 2 \lambda \frac{d-1}{d} \right]}{1 + 4 \frac{d-1}{d}\lambda^2}.
\end{equation}

Now the equation for $\lambda$ gives,
\begin{equation}
4 \frac{d-1}{d} \lambda^2 +1 = 4 \frac{d-1}{d}\lambda + 2.
\end{equation}
Thus, the value of  $ \sqrt{\frac{d}{d-1}} ab + a^2$ is now given by $\lambda$. 
This $\lambda$ is our upper bound on the sum $\sum_{y \in \delta(\mt)}|P_{iy}|$.




\end{proof}

Now, with combinatorial estimates for $(P|_{\mt})^{-1}$ in hand, we will be able to return to compute estimates on the difference $(G|_{\mt})^{-1} - (P|_{\mt})^{-1}$. This is encompassed in the following lemma,

\begin{lemma}

Recall the setting of the previous Lemma.


Then
\begin{equation}
|(G|_{\mt})^{-1} - (P|_{\mt})^{-1}|_{xy} \lesssim \epsilon.
\end{equation}
\end{lemma}

\begin{proof}
Let $E$ be the matrix of differences,
\begin{equation}
\begin{aligned}
&(G|_\mt)^{-1} = (P|_{\mt})^{-1} + E\\
& G|_\mt = P|_{\mt} +W.
\end{aligned}
\end{equation}
By our inductive assumption, the entries of $W$ are bounded by $W_{ab}\le \epsilon(z,w)$ for all $a,b$ found in $\mt$.

Multiplying these equations together give us,
\begin{equation}
I = I + E P|_{\mt} + (P|_{\mt})^{-1}W + EW.
\end{equation}
Manipulating this expression gives us,
\begin{equation}
E = -(P|_{\mt})^{-1}W (P|_{\mt})^{-1} - EW(P|_{\mt})^{-1}.
\end{equation}
Considering the appearance of $E$ on the right hand side, our hope is to use a maximum principle argument to show that the entries of $E$ are bounded.

Define $\Lambda:= max_{c,d \in \mt} |E_{cd}|$. Then, we see that,
\begin{equation}
|[E W (P|_{\mt})^{-1}]_{cd}| \le \sum_{ef} |E_{ce}| |W_{ef}| |[(P|_{\mt})^{-1}]_{fd}| \le 2 (2d-1)^{l} \Lambda \epsilon \sum_{f} |[(P|_{\mt})^{-1}]_{fd}|.
\end{equation}

Here, we used the fact that the number of possible indices $e$ in $\mt$ is $2(2d-1)^l$, which satisfies $(2d-1)^{l} \le (\log N)^2$. Furthermore, one can use Lemma \ref{lem:compPT} to assert
\begin{equation}
\sum_{f}  |[(P|_{\mt})^{-1}]_{fd}| \lesssim 1.
\end{equation}


Thus, we can bound the second term as,
\begin{equation}
|[E W (P|_{\mt})^{-1}]_{cd}|  \le o(1) \Lambda.
\end{equation}

Now, we can try to bound the first term,
we see that,
\begin{equation}
\begin{aligned}
|[(P|_{\mt})^{-1}W (P|_{\mt})^{-1}]_{cd}| &\le \sum_{ef}|[(P|_{\mt})^{-1}]_{ce}| |W_{ef}| |[(P|_{\mt})^{-1}]_{fd}| \\
&\le \epsilon (\sum_{e} |[(P|_{\mt})^{-1}]_{ce}|) (\sum_{f}|[(P|_{\mt})^{-1}]_{fd}|)
\\& \lesssim  \epsilon.
\end{aligned}
\end{equation}
The last line is again based on Lemma \ref{lem:compPT}.

Combining these estimates together, we see that,
\begin{equation}
\Lambda \le o(1) \Lambda + C \epsilon.
\end{equation}

This gives us our desired statement.
\end{proof}

We now return to the proof of Proposition \ref{prop:removeTest}

\begin{proof}[Proof of Proposition \ref{prop:removeTest}]


To compare $|G_{ij}^{(\mt)} - P_{ij}^{(\mt)}|$, we use equation \eqref{eq:SchurcompOutsideTree} and deal with each term separately. We can bound $|G_{ij} -P_{ij}|$ by $\epsilon$ by our induction hypothesis.
We need to perform a little bit more manipulation on our first term $(G-P) (G|_\mt)^{-1}G$.

We see that we have,
\begin{equation}
(G-P)(G|_\mt)^{-1} G = (G-P) (P|_\mt)^{-1} G + (G-P) [(G|_\mt)^{-1} - (P|_\mt)^{-1}] G.
\end{equation}

We now bound the first term appearing above,
\begin{equation} \label{eq:appear1}
|[(G-P) (P|_\mt)^{-1} G ]_{ij}| \le \sum_{x,y \in \mt}|[G-P]_{ix}| |[(P|_{\mt})^{-1}]_{xy} ||G_{yj}|.
 \end{equation}
We can bound $[[G-P]_{ix}|$ by $\epsilon$.
If we fix $y$ and sum over $x$, we observe that,
\begin{equation}
\sum_{x\in \mt} |[(P|_{\mt})^{-1}]_{xy}| \lesssim1,
\end{equation}
uniformly over $x$.

Thus, the preceding equation \eqref{eq:appear1} is bounded by,
\begin{equation}
\epsilon \sum_{y \in \mt}|G_{yj}| \le \epsilon \sum_{y \in \mt} [\epsilon+ |P_{yj}|]
\end{equation}
We have, from Lemma \ref{lem:sumintree}, that
$$
\sum_{y \in \mt} |P_{yj}| \le (\log N)^{\mathfrak{j}(d) \mathfrak{a}} .
$$


The term  $(G-P) [(G|\mt)^{-1} - (P|\mt)^{-1}] G $ will give an error of the form $(\log N)^{8\mathfrak{a}} \epsilon^2$. This is of even smaller order than our proposed main error term.

Similar logic can be used to estimate the other terms remaining in \eqref{eq:SchurcompOutsideTree}. Ultimately, we can get our desired bound,
\begin{equation}
|G_{ij}-P_{ij}| \lesssim \epsilon'.
\end{equation}
\end{proof}

\section{Removal of the switched vertices $W_S$: Proof of Proposition \ref{prop:removeWS} } \label{sec:step2}

As one can see from the statement of Proposition \ref{prop:removeWS}, one needs many auxiliary notions for optimal estimates. As we will describe later, it is not good enough to simply use the worst case bound at all time. In general, most elements will be much smaller than the worst case bound $\epsilon'$. It is this property which we exploit in what follows. 

\subsection{Introduction of Green's function distance metric}

In the past section, we used $\epsilon$ as a control parameter to determine the maximum value of the distance of the Green's functions $|G_{ij} - P_{ij}|\le \epsilon.$  However, we find that there is a small difference between this worst case estimate and the average case estimate.

By the Ward Identity, we have that,
\begin{equation}
\sum_{j }|G^{(\mt)}_{ij}|^2 = \frac{\text{Im}[G^{(\mt)}_{ii}]}{\eta}.
\end{equation}
Thus, for any function $f(N)$, we see that no more than $\frac{N}{f(N)}$ indices $j$ (for a fixed $i$) can satisfy the property that,
\begin{equation}
|G^{(\mt)}_{ij}| \ge \frac{\sqrt{f(N) \text{Im}[G^{(\mt)}_{ii}]}}{\sqrt{N\eta}}.
\end{equation}
By replacing $\text{Im}[G_{ii}]$ with the deterministic value $m_{\infty}$ with small error, we see that the pigeonhole principle will assert the concrete estimate that,
\begin{equation}
|G^{(\mt)}_{ij}| \ge \frac{\sqrt{f(N) \text{Im}[m_{\infty}]}}{\sqrt{N\eta}},
\end{equation}
for no more than $2\frac{N}{f(N)}$ indices $j$ ( This merely used the fact that $\text{Im}[G^{(\mt)}_{ii}] \le 2 \text{Im}[m_{\infty}]$ from our inductive local law estimate on $\text{Im}[G^{(\mt)}_{ii}]$.)


It might look like this gives very little gain at first; however, we will soon see that improperly applied stability estimates would ruin the stability estimates by vastly increasing the prefactor of $(\log N)$ in front of the epsilon. Namely, instead of $(\log N)^{j(d) \mathfrak{a}}$, we may get $(\log N)^{C \mathfrak{a}}$ for very large $C$ after our stability iteration. Unfortunately, the concentration estimate we apply later is not good enough to cancel out the increase of this multiplicative prefactor.

 \begin{rmk} the concentration estimate which we will prove in Section \ref{sec:concentration} can kill a factor of say $(\log N)^{ \mathfrak{a}}$. Thus, it unfortunately cannot kill a large multiplicative factor of the form $(\log N)^{K \mathfrak{a}}$. However, choosing $\mathfrak{a}$ large enough can kill any power of the form $(\log N)^p$ for some constant power $p$. The notions in this section allow us to reduce the decay in the stability estimate from $(\log N)^{K\mathfrak{a}}$ to $(\log N)^p$.
\end{rmk}


\begin{defn} \label{defn:phi}
Fix the parameter
\begin{equation}
\phi(z,w) := \frac{(\log N)^{\frac{K}{2} \mathfrak{a}}}{N^{1-\mathfrak{c}}} + \frac{(\log N)^{\frac{K}{2} \mathfrak{a}} \sqrt{\text{Im}[m_{\infty}]}}{\sqrt{N \eta}}.
\end{equation}
The important part of $\phi$ is that $\phi \ll  (\log N)^{-\frac{K}{2} \mathfrak{a}}\epsilon(z,w)$. We also use $(\epsilon')^2 \ll \phi$ many times for convenience of the final expression.

We say two indices $i$ and $j$  corresponding to vertices in $(\mt)^c$ are Green's function connected if there is any neighbor $i'$ of $i$ and a neighbor $j'$ of $j$ such that $|G^{(\mt)}_{i'j'}|\ge \phi$. We will use the notation $i \sim j$ if this is the case. Note that this is not an equivalence relation due to lack of transitivity.

\end{defn}

We will use the fact that most Green's function would have the slightly better error estimate $|G^{(\mt)}_{ij} - P^{(\mt)}_{ij}| \le \phi$ so that we do not need to worry as much about the decay of the multiplicative factor. In order to see this, we prove the following lemma, which shows how most vertices in $\mt^c$ cannot be Green's function connected to too many ( more than $O(\log N)$) of the vertices that we want to switch.


\begin{lemma} \label{lem:Grnfunccon}
Fix our random $d$regular digraph $\GG$ as well as a vertex $v$ whose local $\mathfrak{R}$-like neighborhood is treelike. Consider the local $l$ neighborhood $\mt$ around the vertex $v$. Let $e_1=(v_1,w_1),\ldots, e_\mu=(v_\mu,w_\mu)$ be randomly selected edges that lie outside $\mt$ in $\GG$ and $\mu$ is the number of edges connecting $\mt$ to $\mt^c$. Then, with probability $1 - O(N^{-\mathfrak{q}})$, we have the following two conditions:
\begin{enumerate}
\item Any index $v$ in $\mt^c$ is Green's function connected to no more than $O(\log N)$ of the indices corresponding to vertices $\{b_1,c_1,\ldots, b_\mu, c_\mu\}$.
\item The set $\alpha$ such that $\{b_{\alpha},c_{\alpha}\}$ is Green's function connected to some $\{b_{\beta},c_{\beta}\}$ for $\beta \ne \alpha$ is $O(\log N)$.
\end{enumerate}

 Any vertex $v$ in $\mt^c$ is Green's function connected to no more than $O(\log N)$ vertices in $\cup_{i=1}^m \{v_i,w_i\}$.

\end{lemma}
\begin{proof}Let $X$ be any vertex in $\mt^c$. For the vertex $X$, the set $X_{GF}$ of the vertices that are Green's function connected to $X$ is of size less than $\frac{N}{(\log N)^{K \mathfrak{a}}}$. For each $i$ in $\{1,\ldots,m\}$, the chance that $X$ is Green's function connected to one of $v_i,w_i$is an independent event. At this point, one can obtain the desired result by straightforward counting arguments. One can see \cite[Prop. 5.14]{HuangYau} for details.
\end{proof}

At this point, we can finally define the event $S_G(\GG)$ of good switching events.

\begin{defn}
Consider a graph $\GG$ in $\Omega_g(z,w)$ as well as a distinguished vertex $o$ along with its radius $\ell$ neighborhood $\mt$. Recall $\mathbb{S}_{\GG,o}$, the switching events that correspond to the graph $\GG$ around vertex $o$, as well as the characteristic variables $\chi_{\alpha}$. We define the event $S_G(\GG) \subset \mathbb{S}_{\GG}$ as follows: a switch $S$ will belong to $S_G(\GG)$ if it satisfies the following conditions,
\begin{enumerate}
\item All except for $O_\mathfrak{q}(1)$ of the vertices $\{c_1,\ldots, c_{\mu}\}$ have radius $\mathfrak{R}$ tree neighborhoods in $\GG^{(\mt)}$.
\item $\chi_{\alpha}=1$ for all except $O_{\mathfrak{q}}(1)$ many vertices.
\item The Green's function connectivity conditions of Lemma \ref{lem:Grnfunccon} hold.

\end{enumerate}

Due to Lemmas \ref{lem:Grnfunccon} and Proposition \ref{Prop:detswitch}, the event $S_G(\GG)$ holds with probability $\mathbb{P}(S_G(\GG))= 1- O(N^{-\mathfrak{q}})$.

\end{defn}

In what follows in Sections \ref{sec:step2} ,\ref {sec:Step3}, and \ref{sec:step4}, we will always implicitly assume we are considering a switching event $S$ found in $S_G(\GG)$.

\subsection{Switching estimates involving the Green's function distance}

We can apply a very similar type of analysis for the removal of the vertices in $W_S$. Given a switching event $S$, where we elect to switch the vertices on the boundary of $\mathbb{T}$ with arbitrary edges, the vertices $W_S$ consist of those  vertices $\{b_{1},\ldots, b_{\mu}\}$. 
Again, we introduce, as a shorthand, the matrix $P:= G(Ext(B_r(\{i,j\} \cup W_S, \GG\setminus \mt),Q_I,Q_O))$.
As before, the point of this matrix is that we can write $G^{(\mt)}$ as,
\begin{equation}
(G^{(\mt)})^{-1} = \begin{bmatrix} A & B \\
B^* &D			\end{bmatrix},
\end{equation}

and
\begin{equation}
P^{-1} = \begin{bmatrix} A&  B \\
B^* & D'
\end{bmatrix}.
\end{equation}
Here, $A$ is the adjacency matrix of the graph restricted to $W_S$ and $B$ will be the adjacency matrix of the edges connecting $W_S$ to $W_S^c$. At this point, we remark that we can restrict to the event that all the edges that are chosen for switching are far away from each other; more specifically, two vertices in $b_{\alpha}, b_{\beta}$ in $W_S$ are of distance at least $2r$ from each other.

Due to this restriction, the matrix $P|_{W_S}$ is much easier to understand. It is a block diagonal matrix (with $2 \times 2$ blocks) where the blocks consist of the indices $b_{\alpha}$ and $b_{\alpha}+N$ of the vertices that border the edge used in the switching. For this reason, it is also easy to describe the matrix $(P|_{W_S})^{-1}$.

Another remark that we will have to make is that in addition to proving the stability of the worst case bound $|G^{(T)}_{ij} - P^{(T)}_{ij}|\le \epsilon$, we would further need to prove the stability of the average case bound for pairs of vertices that are not Green's function connected. With this remark in mind, we now return to the proof.

By applying the Schur complement formula, we have that,
\begin{equation}
\begin{aligned}
G^{(\mt \cup W_S)} = G^{(\mt)} - G^{(\mt)} (G^{(\mt)}|_{W_S})^{-1} G^{(\mt)},
P^{( W_s)} = P - P (P|_{W_S})^{-1} P.
\end{aligned}
\end{equation}

We can explicitly write out our difference as,
\begin{equation} \label{eq:differenceWS}
\begin{aligned}
G^{(\mt \cup W_S)} - P^{(W_S)} &= [G^{(\mt)} - P] - [ G^{(\mt)} -P] (G^{(\mt)}|_{W_S})^{-1} G^{(\mt)} \\&- P[(G^{(\mt)}|_{W_S})^{-1} - (P|_{W_S})^{-1}] G^{(\mt)} - P(P|_{W_S})^{-1}[G^{(\mt)} - P].
\end{aligned}
\end{equation}


As before, one main issue would be to try to estimate the value of $(G^{(\mt)}|_{W_S})^{-1}$

\begin{lemma} \label{lem:switchinv}
We have the following estimates on
$$
\left|(P|_{W_S})_{ij}^{-1} - (G^{(\mt)}|_{W_S})_{ij}^{-1}\right| \lesssim  \epsilon'.
$$

Furthermore, if we know that $i$ and $j$ are not Green's function connected, then we are able to get the superior estimate,
\begin{equation}
|(P|_{W_S})_{ij}^{-1} - (G^{(\mt)}|_{W_S})_{ij}^{-1}| \lesssim \phi,
\end{equation}
for some constant $C$.
We remark here that if $i$ and $j$ are not Green's function connected, then they would belong to different blocks of $P|_{W_S}$. In this case, $[P|_{W_S}]_{ij}$ and $(P|_{W_S})^{-1}_{ij}$ are both zero.
\end{lemma}
\begin{proof}
Again, we write,
\begin{equation}
\begin{aligned}
& (G^{(\mt)}|_{W_S})^{-1} = (P|_{W_S})^{-1} + E\\
& G^{(\mt)}|_{W_S}= P|_{W_S} + J.
\end{aligned}
\end{equation}

Multiplying these equations, we again can derive that,
\begin{equation} \label{eq:multWS}
 E = -(P|_{W_S})^{-1} J (P|_{W_S})^{-1} - EJ (P|_{W_S})^{-1}.
\end{equation}

Let us estimate the first term. We have,
\begin{equation} \label{eq:lattersum}
|[(P|_{W_S})^{-1} J (P|_{W_S})^{-1}]_{ij}| \le \sum_{a,b \in W_S} | (P|_{W_S})^{-1}_{ia}| |J_{ab}|| (P|_{W_S})^{-1}_{bj}|.
\end{equation}

We can bound $|J_{ab}| \le \epsilon'$. Furthermore, observe that there are at most $2$ values $a$ such that $(P|_{W_S})^{-1}_{ia}$ is non-zero and $2$ values $b$ such that $(P|_{W_S})^{-1}_{bj}$; this is due to our assumption that the switched edges in $W_S$ are far apart from one another.  Furthermore, these values are $O(1)$ quantities; this can be shown be a direct application of the Schur complement formula on a pure tree-like neighborhood combined with  some perturbation computations.

Let us define $$P_{\infty}:= G(Ext(B_r(\{i,j\} \cup W_S, \GG \setminus \mathbb{T}),m_{\infty},m_{\infty})).$$ By Lemma \ref{lem:pert}, one can show that $$||P|_{W_S} - P_{\infty}|_{W_S}||_{\infty} \ll 1.$$ By $\infty$ norm here, we refer to the supremum over all entries. Furthermore, since the vertices in $W_S$ are chosen so that they have tree-like neighborhoods, we additionally have that $(P_{\infty}|_{W_S})^{-1}$ is explicitly given by
$$
\begin{bmatrix}
-z - \frac{d}{d-1} m_{\infty} & w\\
\overline{w} & - z -\frac{d}{d-1}m_{\infty}
\end{bmatrix},
$$
on each of its diagonal blocks. Clearly, all of these terms are $O(1)$.  By the same analysis we have done earlier, we can argue that $ ||(P|_{W_S})^{-1} - (P_{\infty}|_{W_S})^{-1}||_{\infty} \ll 1$; this is especially simple since we only have to consider only $2 \times 2$ diagonal blocks.  This shows that the entries of $(P_{\infty}|_{W_S})^{-1}$ are of $O(1)$, as desired.

Thus, the sum from equation \eqref{eq:lattersum} can be bounded by $C |J_{ab}|$ for the unique possible values of $a$ and $b$ (if they exist), that return a non-zero quantity in the computation  $ | (P|_{W_S})^{-1}_{ia}| |J_{ab}|| (P|_{W_S})^{-1}_{bj}|$. Since we have earlier shown that $|J_{ab}|\le  \epsilon'$ for any pair $a$ and  $b$, we know that $\eqref{eq:lattersum} \le C\epsilon.$

Now, we estimate the second term from \eqref{eq:multWS},
\begin{equation}
|\sum_{a,b \in W_S} E_{ia} J_{ab}  (P|_{W_S})^{-1}_{bj}| \le 2 (\sup_{a,b \in W_S} E_{ab})|W_S| \epsilon'.
\end{equation}

We applied the triangle inequality and used the fact that $(P|_{W_S})^{-1}_{bj}$ is zero for all values of $b$ except for $2$ of them. We can bound $E_{ia}$ by the supremum over all values and $J_{ab}$ by $\epsilon'$. Now, $ |W_S| \epsilon'(z,w) \ll 1$.  Combining this inequality and the previous one, we can derive the inequality,
\begin{equation}
 \sup_{ab} E_{ab} \le (\sup_{ab} E_{ab}) |W_S| \epsilon' +  \epsilon'.
\end{equation}

This would imply that $\sup_{a,b \in W_S} E_{ab} \le 2 \epsilon'$ as desired.

With the estimate on the supremum in hand, we can improve our estimate for the difference on terms that are not Green's function connected.
Observe, that we can bound
\begin{equation}
|\sum_{a,b} E_{ia} J_{ab}  (P|_{W_S})^{-1}_{bj}| \le (\sup_{ab} E_{ab}) |W_S| \epsilon' \le  (\log N)^{4 \mathfrak{a}} (\epsilon')^2 \ll \phi.
\end{equation}

Furthermore, we have the following improvement when considering the first term from equation \eqref{eq:multWS}.
\begin{equation}
\sum_{a,b \in W_S} | (P|_{W_S})^{-1}_{ia}| |J_{ab}|| (P|_{W_S})^{-1}_{bj}|,
\end{equation}
Let $i$ correspond to one of $b_{\alpha},b_{\alpha}+N$ and and $j$ correspond to one of $b_{\beta},b_{\beta}+N$, where $b_{\alpha}$ and $b_{\beta}$ are vertices in $W_S$. The only non-zero terms $[(P|_{W_S})^{-1}]_{ia}$ has $a= b_{\alpha}$ or $b_{\alpha}+N$ . Similarly, the only nonzero terms $(P|_{W_S})^{-1}_{bj}$ has $b=b_{\beta}$ or $b_{\beta}+N$. For these indices, we can apply the bound $|J_{ij}|\le \phi$ due to our assumption that $i$ and $j$ are not Green's function connected.
\end{proof}

With this inequality in hand, we can start estimating the terms that appear in equation \eqref{eq:differenceWS}. This will complete the proof of Proposition \ref{prop:removeWS}.



\begin{proof}[Proof of Proposition \ref{prop:removeWS}]
The simplest term to estimate is the last one in equation \eqref{eq:differenceWS}.

Consider the following expression,

\begin{equation}
[P (P|_{W_S})^{-1} [G^{(\mt)} - P]]_{ij} = \sum_{a,b \in W_S} |P_{ia}| |(P|_{W_S})^{-1}_{ a b} ||[G^{(\mt)} -P]_{bj}|.
\end{equation}

Let $i$ correspond to one of $b_{\alpha},b_{\alpha}+N$ and $j$ correspond to one of $b_{\beta},b_{\beta}+N$ for $b_{\alpha},b_{\beta} \in W_S$.

As we have argued before, due to the block diagonal structure of $P$, $P_{ia}$ is nonzero for at most $2$ values of $a$. These values can be $b_{\alpha}$ or $b_{\alpha}+N$. Given this value of $a$, $(P|_{W_S})_{ab}^{-1}$ is nonzero for at most $2$ values of $b$. This $b$ can still be only one of $b_{\alpha}$ or $b_{\alpha}+N$. Finally, we can apply the bound that $|[G^{(\mt)} -P]_{bj}|\le \epsilon'$ in general, or $\le \phi$ if $i$ and $j$ are not Green's function connected.

Now, let us consider bounding the second term in equation \eqref{eq:differenceWS}.
This is,
\begin{equation}
\sum_{a,b\in W_S} |P_{ia}||[(P|_{W_S})^{-1} - (G^{(\mt)}|_{W_S})^{-1}]_{ab}| |G^{(\mt)}|_{bj}.
\end{equation}

We first remark that since we know that $|G^{(\mt)}_{bj} - P_{bj}| \le  \epsilon'$, this shows that $|G^{(\mt)}|$ will  have the worst case bound $O(1)$ for at most $2$ values of the index $b$ (those that are in the same block as $j$ in $P$). There are a further $O(\log N)$ indices of $b$ where we would need to apply the worst case bound $|G^{(\mt)}|_{bj} \le \epsilon'$ and $P_{bj}=0$. Finally, we can bound the Green's function for all remaining indices that are possibilities for $b$ as $|G^{(\mt)}|_{bj} \le \phi$, due to our assumption that the indices are not Green's function connected.

Further, we also know that for each fixed index $a$, there are at most $O(\log N)$ indices where we would need to apply the worst case bound $|[(P|_{W_S})^{-1} - (G^{(\mt)}|_{W_S})^{-1}]_{ab}| \le \epsilon'$. These are the indices that $a$ is Green's function connected to $b$. Otherwise, we could apply the bound $\phi$. As before, $|P_{ia}|$ is $0$ for all except for $2$ possible indices $a$.

Now, consider the case that $i$ is not Green's function connected to $j$. Then, if $b$ is one of the two indices connected to $j$ in $W_S$, then we would have the improved bound $|[(P|_{W_S})^{-1} - (G^{(\mt)}|_{W_S})^{-1}]_{ab}| \le \phi$ for these indices $b$ and all indices $a$ such that $|P_{ia}| \ne 0$.

Thus, we would derive the bound,
$$
\begin{aligned}
& \sum_{a,b\in W_S} |P_{ia}||[(P|_{W_S})^{-1} - (G^{(\mt)}|_{W_S})^{-1}]_{ab}| |G^{(\mt)}|_{bj} \\ &
\lesssim \phi +  \log N (\epsilon')^2 +   \log N |W_S| \epsilon' \phi + |W_S|^2 \phi.
\end{aligned}
$$

Similar techniques allow us to treat the first term in equation \eqref{eq:differenceWS}.
\end{proof}

\section{Relating $G^{(\mt \cup W_S)}$ to $\tilde{G}^{(\mt)}$ of the switched graph: Proof of Proposition \ref{Prop:SwitchGWS}} \label{sec:Step3}

As in the previous sections, we will prove Proposition \ref{prop:removeWS} by applying the Schur complement formula and other Green's function manipulations. Since we are adding vertices back into the graph, we will have slightly different formulas for $\tilde{G}_{ij}^{(\mt)}$ from $G_{ij}^{(\mt \cup W_S)}$ depending on whether $i$ or $j$ is in $W_S $ or not. The full conclusion of Proposition \ref{prop:removeWS} can be derived from the conclusions of Lemmas \ref{lem:ijws}, \ref{lem:switchbound}, and \ref{lem:ijwsc}. 

We consider $P$ given by $G(Ext(B_r(\{i,j\} \cup W_S, \tilde{\GG} \setminus \mt),Q_I,Q_O))$ where $\tilde{\GG}$ is the graph obtained by switching the edges associated to the switching set $W_S$. We observed that $P^{(W_S)} = G(Ext(B_r(\{i,j\} \cup (\mt \cup W_S), \GG \setminus \mt),Q_I,Q_O))$. Namely, $P^{(W_S)}$ is exactly the Green's function we were relating $G$ to in the previous section.

In this section we will apply the Schur complement formula in the form,
\begin{equation}
\begin{aligned}
& \tilde{G}^{(\mt)}|_{W_S} = (A - \tilde{B}G^{(\mt \cup W_S)} \tilde{B}^*)^{-1},\\
& P|_{W_S} = (A - \tilde{B} P^{(W_S)} \tilde{B}^*)^{-1},
\end{aligned}
\end{equation}
where $\tilde{B}$ represents the adjacency matrix between the vertices of $W_S$ to those outside of $W_S$.

The Schur complement formula give us
\begin{equation}
\begin{aligned}
& \tilde{G}^{(\mt)}|_{W_S \times W_S^c} = - \tilde{G}^{(\mt)}|_{W_S} \tilde{B} G^{(\mt \cup W_S)}\\
& P|_{W_S \times W_S^c} = - P|_{W_S} \tilde{B}  P^{(W_S)},
\end{aligned}
\end{equation}
and an application of the resolvent identity would give us,
\begin{equation}
\begin{aligned}
& \tilde{G}^{(\mt)}|_{W_S^c \times W_S^c} = G^{(\mt \cup W_S)} + G^{(\mt \cup W_S)} \tilde{B}^* \tilde{G}^{(\mt)}|_{W_S} \tilde{B} G^{(\mt \cup W_S)}\\
& P= P^{(W_S)} + P^{(W_S)} \tilde{B}^* P|_{W_S} \tilde{B} P^{(W_S)}.
\end{aligned}
\end{equation}

We see first that it is crucial to understand the difference of $\tilde{G}^{(\mt)}$ and $P$ on the set $W_S$ at the first step. This can be done by the resolvent identity,
\begin{equation} \label{eq:resolWS}
\tilde{G}^{(\mt)}|_{W_S} - P|_{W_S} = \tilde{G}^{(\mt)}|_{W_S} \tilde{B}(P^{(W_S)} - G^{(\mt \cup W_S)}) \tilde{B}^* P|_{W_S}.
\end{equation}

We can estimate the maximum difference of the right hand side.
\begin{lemma} \label{lem:ijws}
For any two indices $i$ and $j$ corresponding to vertices in $W_S$, we have,
\begin{equation}
|[\tilde{G}^{(\mt)}|_{W_S} - P|_{W_S}]_{ij}| \lesssim \epsilon'.
\end{equation}

Furthermore, if $i$ and $j$ correspond to vertices that are not Green's function connected, then we have the improved bound,
\begin{equation}
    |[\tilde{G}^{(\mt)}|_{W_S} - P|_{W_S}]_{ij}| \lesssim \phi.
\end{equation}

\end{lemma}

\begin{proof}

Let $\Gamma := \max_{i,j \in W_S} |[\tilde{G}^{(\mt)}|_{W_S} - P|_{W_S}]_{ij}|.$
Using this definition inside \eqref{eq:resolWS}, we would derive the relation,

\begin{equation} \label{eq:7771}
\begin{aligned}
 |[\tilde{G}^{(\mt)}|_{W_S} - P|_{W_S}]_{ij}| &=( [\tilde{G}^{(\mt)}|_{W_S} - P|_{W_S}] \tilde{B}(P^{(W_S)} - G^{(\mt \cup W_S)}) \tilde{B}^* P|_{W_S})_{ij} \\
&+ (P|_{W_S}\tilde{B}(P^{(W_S)} - G^{(\mt \cup W_S)}) \tilde{B}^* P|_{W_S})_{ij}.
\end{aligned}
\end{equation}

Let $i$ correspond to one of the indices $b_{\alpha}$ or $b_{\alpha}+N$ and let $j$ correspond to one of the indices $b_{\beta}$ or $b_{\beta}+N$, where $b_{\alpha}$ and $b_{\beta}$ are vertices that belong to $W_S$

The first quantity on the right hand size can be bounded by $(\log N)^{8 \frak a} \epsilon' \Gamma \ll \Gamma$. We write the second term on the right hand side as,
\begin{equation}
\begin{aligned}
&(P|_{W_S}\tilde{B}(P^{(W_S)} - G^{(\mt \cup W_S)}) \tilde{B}^* P|_{W_S})_{ij} \\
&= \sum_{x,y,z,w}(P|_{W_S})_{ix} \tilde{B}_{xy}(P^{(W_S)} - G^{(\mt \cup W_S)})_{yz}(\tilde{B}^*)_{zw} (P|_{W_S})_{wj}.
\end{aligned}
\end{equation}

 For the second quantity on the right hand size, we remark that $[P|_{W_S}]_{ix} =1$ only if $x$ corresponds to one of the indices $b_{\alpha}$ or $b_{\alpha}+N$. This is due to the fact that $b_{\alpha}$ and $b_{\beta}$ for $\alpha \ne \beta$ are still far apart even after $b_{\alpha}$ is connected to $l_{\alpha}$ for all $\alpha$ in $W_S$. Thus, $P|_{W_S}$ is still a block-diagonal matrix. Similarly, $w$ must be one of $b_{\beta}$ or $b_{\beta}+N$ if $(P|_{W_S})$ were to be non-zero. Since $B$ and $B'$ are the adjacency matrices for 
 $d$-regular graphs, this would further imply that there are only finitely many values that $y$ could take and finitely many values that $z$ could take. By our choice of the construction, we would know that these $y$ vertices and these $z$ vertices would all have a local treelike neighborhood.

Since we have the bound $|P^{(W_S)}_{cd} - G^{(\mt \cup W_S)}_{cd}| \le 2 \epsilon'(z,w)$ for any generic pair $c$,$d$ that has a tree-like neighborhood, we would apply this bound everywhere and observe that,
\begin{equation} \label{eq:71bound}
\sum_{x,y,z,w}(P|_{W_S})_{ix} \tilde{B}_{xy}(P^{(W_S)} - G^{(\mt \cup W_S)})_{yz}(\tilde{B}^*)_{zw} (P|_{W_S})_{wj} \lesssim C\epsilon'.
\end{equation}
From this point, a standard maximum principle argument completes the the proof of the bound on $P- G$ for general $i$ and $j$.

If $i$ and $j$ are not Green's function connected, the only observation we need is that we can improve the bound on $|P^{(W_S)} - G^{(\mt \cup W_S)|}$ in equation \eqref{eq:71bound} to $\phi$.  The first term in equation \eqref{eq:7771} can be bounded by $\epsilon'^2 \ll \phi$.
\end{proof}

We can apply  these estimates in order to understand the differences of $\tilde{G}^{(\mt)}$ and $P$ on the rectangular sub-matrix $W_S \times W_S^c$. The following lemma encompasses this information. However, we mention that some of these estimates can improve if we instead consider some Green's function connected information on the specific vertices $\{c_{\alpha},b_{\alpha}\}$ that take part in the switching. We introduce a definition to encompass this information.

\begin{defn} \label{def:U}
We define a new set $U$ of  switched vertices $b_{\alpha},c_{\alpha}$ that satisfy some specific properties. The vertices $b_{\alpha},c_{\alpha}$ are included in the set $U$ if the following properties are satisfied:

\begin{enumerate}
\item  For $\beta \ne \alpha$, the set of vertices $\{b_{\beta},c_{\beta}\}$ is not Green's function connected to $\{b_{\alpha},c_{\alpha}\}$.
\item The $\mathfrak{R}$ neighborhood around $c_{\alpha}$ in $\GG \setminus \mt$ is tree-like.
\end{enumerate}

\end{defn}

\begin{lemma} \label{lem:switchbound}
Let $i$ and $j$ be two vertices: $i$ belonging to $W_S$ and $j$ belonging to $W_S^c$. Then, we have the following comparison bound,
\begin{equation}
|\tilde{G}^{(\mt)}_{ij} -P_{ij}| \lesssim \epsilon'.
\end{equation}

Recall our notation $\{a_{\alpha},b_{\alpha},c_{\alpha}\}$ for the types of vertices that participate in the switching. Assume that $j$ is of the form $c_{\alpha}$ for $c_{\alpha}$ a vertex in $U$ and that $i= b_{\beta}$ for $\beta \ne \alpha$.
\begin{equation}
|\tilde{G}^{(\mt)}_{b_{\beta} c_{\alpha}} -P_{b_{\beta} c_{\alpha}}| \lesssim \phi + (\epsilon')^2.
\end{equation}

\end{lemma}

\begin{proof}

\textit{General Case:}

We start with the identity,
\begin{equation} \label{eq:firstswitchstab}
\tilde{G}^{(\mt)}|_{W_S \times W_S^c} - P|_{W_S \times W_S^c}= [P|_{W_S} - \tilde{G}^{(\mt)}|_{W_S}] \tilde{B}G^{(\mt \cup W_S)} + P|_{W_S} \tilde{B}[P^{(W_S)} - G^{(\mt \cup W_S)}].
\end{equation}

Let us first deal with the second term $(P|_{W_S} \tilde{B}[P^{(W_S)} - G^{(\mt \cup W_S)}])_{ij}$. We have,
\begin{equation}
(P|_{W_S} \tilde{B}[P^{(W_S)} - G^{(\mt \cup W_S)}])_{ij} = \sum_{x,y} (P|_{W_S})_{ix} \tilde{B}_{xy}[P^{(W_S)} - G^{(\mt \cup W_S)}]_{yj}.
\end{equation}
Now,  $(P|_{W_S})_{ix}$ can only be non-zero only when $x$ is one of the indices corresponding to the same vertex as $i$. Since $y$ is connected to $x$, there are finitely many choices for $y$ and all of these $y$ vertices have tree-like neighborhoods. We can finally apply the bound that $[P^{(W_S)} - G^{(\mt \cup W_S)}]_{yj}$ is less than $2\epsilon'$.

We can deal with the first term similarly, by first writing,
\begin{equation}
 [P|_{W_S} - \tilde{G}^{(\mt)}|_{W_S}] \tilde{B}G^{(\mt \cup W_S)} =  [P|_{W_S} - \tilde{G}^{(\mt)}|_{W_S}] \tilde{B}P^{( W_S)}+ [P|_{W_S} - \tilde{G}^{(\mt)}|_{W_S}] \tilde{B}[G^{(\mt \cup W_S)} - P^{(W_S)}].
\end{equation}

The second term is quadratic in the error $(\epsilon')^2$ and can compensate for any factor of the form $(\log N)^{4 \mathfrak{a}}$ that might appear. The first term above can be dealt with the exact same way as the previous bound we considered in this proof.

\textit{Specific Case of $i= c_{\alpha}$ and $j= b_{\beta}$ with $\beta \ne \alpha$:}

We can apply the same identity detailed in \eqref{eq:firstswitchstab}. We do not need to consider the contribution of terms with two factors of $G-P$ (as these will give an error of the form $(\epsilon')^2$.) There are two terms remaining using only a single factor of the form $G-P$ and the analysis of both of these terms is very similar.

Let us consider the term,
\begin{equation}
(P|_{W_S})_{b_{\beta} x} \tilde{B}_{x y} [P^{(W_S)} - G^{(\mt \cup W_S)}]_{y c_{\alpha}}.
\end{equation}
We abuse notation slightly as we let $b_{\beta},c_{\alpha}$ refer to any of the two indices that correspond to the vertex $b_{\beta},c_{\alpha}$ rather than the vertex itself. Since $P|_{W_S}$ is diagonal, $x$ again has to be an index corresponding to $b_{\beta}$. In addition, $y$ has to correspond to a vertex adjacent to $b_{\beta}$ , since $\tilde{B}$ is an adjacency matrix. Finally, in the last term, we use the fact that $y=b_{\beta}$ and $c_{\alpha}$ are no longer Green's function connected to bound $[P^{(W_S)} - G^{(\mt \cup W_S)}]_{y c_{\alpha}}$ by $\phi$. This shows that the above term is bounded by a constant multiple of $\phi$, as desired.
\end{proof}

At this point, we can finally continue and prove bounds of the difference when the indices $i$ and $j$ are both in $W_S^c$. Furthermore, we can prove superior bounds on some Green's function entries if we take into account the Green's function connectivity properties. 

\begin{lemma} \label{lem:ijwsc}
Let $i$ and $j$ be vertices of $W_S^c$ in $\tilde{\GG} \setminus \mt$. We have the following general estimate on Green's function bounds,
\begin{equation}
|\tilde{G}_{ij}^{(\mt)} - P_{ij}| \lesssim \epsilon'.
\end{equation}

Recall the notation $c_{\alpha}, b_{\alpha},a_{\alpha}$ for the vertices that participate in the switching. Let $c_{\alpha}$ be a vertex in $U$ and let $j$ be a vertex that is not Green's function connected to $\{c_{\alpha},b_{\alpha}\}$ in $\GG \setminus \mt$. Then, we have the bound,
\begin{equation}
|\tilde{G}_{ij}^{(\mt)} - P_{ij}| \lesssim \phi.
\end{equation}

\end{lemma}

\begin{proof}

\textit{General Case:}

We see we can  write the difference between $\tilde{G}^{(\mt)}$ and $P$ on $W_S^c \times W_S^c$ as,
\begin{equation} \label{eq:expansionswitch}
\begin{aligned}
&G^{(\mt \cup W_S)} - P^{(W_S)} + (G^{(\mt \cup W_S)} - P^{(W_S)}) \tilde{B}^* P|_{W_S} \tilde{B} P^{( W_S)}\\
&+ P^{(W_S)}\tilde{B}^* (\tilde{G}^{(\mt)}|_{W_S}- P|_{W_S}) \tilde{B} P^{( W_S)} + P^{(W_S)}\tilde{B}^* P|_{W_S} \tilde{B}(G^{(\mt \cup W_S)} - P^{(W_S)})\\
&+ (G^{(\mt \cup W_S)} -P^{(W_S)}) \tilde{B}^* (\tilde{G}^{(\mt)}|_{W_S}- P|_{W_S}) \tilde{B}P^{(W_S)} \\
& +(G^{(\mt \cup W_S)} -P^{(W_S)}) \tilde{B}^*  P|_{W_S} \tilde{B}(G^{(\mt \cup W_S)}-P^{(W_S)} )\\
&+(G^{(\mt \cup W_S}) -P^{(W_S)}) \tilde{B}^* (\tilde{G}^{(\mt)}|_{W_S}- P|_{W_S}) \tilde{B}(G^{(\mt \cup W_S)}-P^{(W_S)}) \\
& +  P^{(W_S)}\tilde{B}^* (\tilde{G}^{(\mt)}|_{W_S}- P|_{W_S}) \tilde{B} (G^{(\mt \cup W_S)} - P^{(W_S)})
\end{aligned}
\end{equation}

The terms that involve two differences of $G- P$ would have an error of at most quadratic order in $\epsilon'^2$. This is already much smaller than $\epsilon'$, even with prefactors of the form $(\log N)^{4 \mathfrak{a}}$. To derive the worst case bound that is of order $\epsilon'$, it suffices to understand the term $[(G^{(\mt \cup W_S)} - P^{(W_S)}) \tilde{B}^* P|_{W_S} \tilde{B} P^{( W_S)}]_{ij}$ (since the other terms are similar). We will write this term as,
\begin{equation}
\sum_{x,y,z,w}(G^{(\mt \cup W_S)} - P^{(W_S)})_{ix} (\tilde{B}^*)_{xy} (P|_{W_S})_{yz} (\tilde{B})_{zw} (P^{(W_S)})_{wj}.
\end{equation}

We remark first that since $P^{(W_S)}$ is a block diagonal matrix $w$ must belong to the same diagonal block as $j$. Furthermore, $\tilde{B}$ is an adjacency matrix connecting $w$ to some element $z$ in $W_S$. In the general case, this must imply that $z$ is the unique element of $W_S$ that is adjacent to the block of $W_S^c$ containing $j$. There are then at most $2d$ choices of $w$ (these are indices corresponding to vertices adjacent to $z$ in $W_S$).   As $P|_{W_S}$ is a block diagonal matrix matrix, with each block corresponding to the two indices mapped to a vertex, there are only $2$ choices for $y$. Finally, there are at most $2d$ choices of $x$ since $\tilde{B}$ is an adjacency matrix.  Thus, we have a finite number of choice for $x,y,z$ and $w$ that give nonzero values. We can bound $(G^{(\mt \cup W_S)} - P^{(W_S)})_{ix}$ by $\epsilon'$ for every $x$ in this finite set. In general, we see this term is bounded by $C \epsilon'$. The analysis of terms involving only a single factor of $G -P$ is similar. Other terms involving two $G-P$'s will be squared and have decay at least $(\epsilon')^2$.


\textit{Specific case of $i= c_{\alpha}$ and $j$ being Green's function separated}

Now, we give the superior analysis in the case that $i$ is the vertex $c_{\alpha}$ for some $\alpha$. These superior estimates are necessary when we reintroduce the tree $\mt$ into the switched graph $\tilde{\GG} \setminus \mt$.

Let us first consider the case that $j$ is not Green's function connected to the vertex $c_{\alpha}$ and that the distance of $j$ to $a_{\alpha}$ is greater than $r$ in $\tilde{\GG} \setminus \mt$. We use the same expansion detailed in equation \eqref{eq:expansionswitch}. As before, we mention that those terms that involve two factors of $G-P$ will incur a factor of $(\epsilon')^2 \ll \phi$. This means that we will only have to deal with the contributions from those terms involving a single $G-P$ factor. It would suffice as before to deal with one such term; the analysis from the contribution of the other terms would be very similar.

Let us consider the term,
\begin{equation}
\sum_{x,y,z,w}(G^{(\mt \cup W_S)} - P^{(W_S)})_{c_{\alpha}x} (\tilde{B}^*)_{xy} (P|_{W_S})_{yz} (\tilde{B})_{zw} (P^{(W_S)})_{wj}.
\end{equation}

Applying the same logic as before , we know that there is a single choice for $z$, the single vertex in $W_S$ that is adjacent to the block of $W_S^c$ containing $j$. This vertex $z= b_{\beta}$ cannot be $b_{\alpha}$, since otherwise $j$ would be Green's function connected to $b_{\alpha}$. Thus, $\beta \ne \alpha$.$y$ must also be an index corresponding to this vertex $b_{\beta}$ and $x$ is an index corresponding to a vertex adjacent to $b_{\beta}$.  Finally, we can use the fact that $c_{\alpha}$ and $b_{\beta}$ are not Green's function connected to bound $|G^{(\mt \cup W_S)} - P^{(W_S)}|$ by $\phi$. This shows that the above term considered is necessarily less than some multiple of $\phi$, as desired.


\end{proof}

\section{Reintroduction of the tree and relating $\tilde{G}^{(\mt)}$ to $\tilde{G}$:Proof of Proposition \ref{prop:introT}} \label{sec:step4}

As we have done previously, our basic step is to apply the Schur complement formula and resolvent estimates in order to treat the removal of the subset $\mt$. We make one important remark before proceeding. At the first step, it is easier to compare $\tilde{G}$ with extensions involving the parameters $Q_I$ and $Q_O$ for the original graph $\mathcal{G}$. Lemmas \ref{lem:Stab} and \ref{lem:step4io} do exactly this. It is only later in Lemma \ref{lem:Qtilde} where we consider the difference between $Q_I$ and $\tilde{Q}_I$. A simple perturbation bound as from Theorem  \ref{thm:pert} will complete the proof. 

We will decompose the adjacency matrix as,
\begin{equation}
\tilde{G}^{-1} = \begin{bmatrix}
A & B\\
B^* & D
\end{bmatrix},
\end{equation}
where $A$ will be the adjacency matrix of the neighborhood $\mt$ and $B$ contains the adjacency information of the boundary of $\mt$ and the vertices they were switched with.

As we have done previously, we will relate the element $\tilde{G}_{ij}$ to the term $P_{ij}$, where $P$ is defined to be $G(\text{Ext}(B_{r}(\{i,j\} \cup \mt, \tilde{\GG}), Q_I,Q_O ))$.
We see that $P^{-1}$ has the adjacency matrix,
\begin{equation}
P^{-1} = \begin{bmatrix}
A & B\\
B^* & D'
\end{bmatrix}.
\end{equation}
The inverse of the matrix $D'$ (will be $G(\text{Ext}(B_{r}(\{i,j\} \cup \mt, \tilde{\GG}^{(\mt)}), Q_I, Q_O))$, which we have analyzed in the previous section. Here, $B_{r}(\mt)$ should be understood as the $r$ neighborhood of $\mt$ in $\GG$ restricted to $\mt^c$.

For vertices $i$ and $j$ inside $\mt$, we can apply the Schur complement formula and derive,
\begin{equation}
\begin{aligned}
&\tilde{G}|_{(\mt)}= ( A + B \tilde{G}^{(\mt)} B^*)^{-1}\\
& P|_{(\mt)} = (A + B P^{(\mt)} B^*)^{-1}.
\end{aligned}
\end{equation}

We can now use the resolvent formula to compute the difference,
\begin{equation} \label{eq:resol4}
\tilde{G}|_{(\mt)} - P|_{(\mt)} = P|_{(\mt)} B (\tilde{G}^{(\mt)} - P^{(\mt)})B^* P + (\tilde{G}|_{(\mt)} - P|_{(\mt)}) B (\tilde{G}^{(\mt)} - P^{(\mt)})B^*.
\end{equation}

Our formal statement on the value of the difference is as follows.
\begin{lemma} \label{lem:Stab}
Let $i$ and $j$ be two points found in our neighborhood $\mt$.
 Let $d_{i}$ denote the distance from $i$ to the root and $d_j$ denote the distance between $j$ and the root.
We have the following estimates,
\begin{equation} \label{eq:StabEst}
|\tilde{G}|_{(\mt)} - P|_{(\mt)}| \lesssim \epsilon' (\log N)^2 \Big(\frac{1}{\sqrt{d-1}} \Big)^{2\ell -d_i -d_j} + (1 +|d_j  +d _i|)  \Big( \frac{1}{d-1}\Big)^{\frac{|d_j -d_i|}{2}} \epsilon'.
\end{equation}

\end{lemma}

\begin{rmk}
The second error term on the right hand side of \eqref{eq:StabEst}  has the issue that it does not have a decay of the power $\left(\frac{1}{d-1}\right)^{\ell}$. By contrast, the first error term on the right hand side of $\eqref{eq:StabEst}$ has the power of $(d-1)^{-2 \ell}$, which can suppress powers of $\log$ as long as we ensure that $d_i$ and $d_j$ are small relative to $\ell$.
\end{rmk}

\begin{proof}

The computations of the previous section show that entries of $(\tilde{G}^{(\mt)} -P^{(\mt)})$ are all bounded by $C\epsilon'$ and a few other terms have superior bounds.

We have for \eqref{eq:resol4}, that
\begin{equation} \label{eq:somesplit}
\begin{aligned}
(P B (\tilde{G}^{(\mt)} - P^{(\mt)})B^* P)_{ij} & \le  \frac{1}{d-1}\sum_{l_{\alpha}, l_{\beta}} P_{i, l_{\alpha}} (\tilde{G}^{(\mt)} - P^{(\mt)})_{\tilde{c}_{\alpha},\tilde{c}_{\beta}} P_{l_{\beta}, j}\\
& \le \frac{1}{d-1} \sum_{\alpha} P_{i,l_{\alpha}} P_{l_{\alpha},j} (\tilde{G}^{(\mt)} - P^{(\mt)})_{\tilde{c}_{\alpha},\tilde{c}_{\alpha}} \\& + \frac{1}{d-1} \sum_{\alpha \ne \beta} P_{i,l_{\alpha}} P_{l_{\beta},j} (\tilde{G}^{(\mt)} - P^{(\mt)})_{\tilde{c}_{\alpha},\tilde{c}_{\beta}}.
\end{aligned}
\end{equation}
$\alpha$ enumerates the vertices on the boundary of $(\mt)$. In a slight abuse of notation, $l_{\alpha}$ and $c_{\alpha}$ actually represent the possible indices corresponding to the vertices on the boundary rather than the vertices themselves; we hope that this usage of notation is clear in context.  The sum is split into two parts that will have different estimates based on whether $\alpha$ is equal to $\beta$ or not.  When $\alpha \ne \beta$ we will be able to get superior estimates when we use the separation of Green's function distance.

Let us now treat the second term on which $\alpha \ne \beta$. This can be divided even further based on whether we are considering an index from the set $U$ from Definition \ref{def:U}. For $\alpha$ corresponding to a vertex in $U$, we can use the fact that $
|(\tilde{G}^{(\mt)} - P^{(\mt)})_{\tilde{c}_{\alpha},\tilde{c}_{\beta}}| \le \phi$. There are $O((\log N)^2)$ that have both of their entries from outside $U$.  Furthermore, we can use the bound $|P_{i,l_{\alpha}}| \le \left(\frac{1}{\sqrt{d-1}}\right)^{\ell -d_i}$.  Applying this, we see that we have the following bound on our cross terms,
\begin{equation}
\frac{\phi}{d-1} \sum_{\alpha \ne \beta} |P_{i,l_{\alpha}}| |P_{l_{\beta},j}| + \frac{\epsilon' (\log N)^2}{d-1} \left(\frac{1}{\sqrt{d-1}}\right)^{\ell -d_i} \left(\frac{1}{\sqrt{d-1}}\right)^{\ell - d_j}
\end{equation}

For the terms that are not cross terms (with $\alpha =\beta$), we merely apply the bound that $|\tilde{G}^{(\mt)} - P^{(\mt)}| \le \epsilon '$.

We see that our full bound on the left hand side of equation \eqref{eq:somesplit}  is the following,
\begin{equation}
\begin{aligned}
&\frac{\epsilon'}{d-1} \sum_{\alpha} |P_{i,l_{\alpha}}||P_{l_{\alpha},j}| + \frac{\phi}{d-1} \sum_{\alpha \ne \beta} |P_{i,l_{\alpha}}||P_{l_{\beta},j}| + \frac{\epsilon'(\log N)^2}{d-1}  \left(\frac{1}{\sqrt{d-1}}\right)^{\ell -d_i} \left(\frac{1}{\sqrt{d-1}}\right)^{\ell - d_j}
\end{aligned}
\end{equation}

  Lemma \ref{lem:sumintree} assures us $\sum_{\alpha} |P_{i,l_{\alpha}}| \le (\log N)^{\mathfrak{j(d)} \mathfrak{a}}$. Thus, we see that $$ \phi \sum_{\alpha \ne \beta}|P_{i,l_{\alpha}}||P_{l_{\beta},j}| \frac{(d-1)^{(2 \ell  -d_i -d_j)/2}}{(d-1)^{(2\ell -d_i -d_j)/2}}\le \frac{\phi (\log N)^{2\mathfrak{j}(d) \mathfrak{a} +4 \mathfrak{a} }}{(d-1)^{(2 \ell - d_i -d_j)/2}}.  $$ Now, since $\phi \ll \epsilon (\log N)^{-K \mathfrak{a}/2}$, we see that for sufficiently large, but still finite, values of $d$ we can show that the factor on the numerator is still less than $\epsilon$.

We finally need to find a bound on $\sum_{\alpha} |P_{i,l_{\alpha}}||P_{l_{\alpha},j}|$ that depends on the distance between  $i$ and $j$ in $\mt$. From the proof of Lemma \ref{lem:pert}, we see that this decay factor is $\lesssim (1 + \text{dist}(i,j))\left(\frac{1}{\sqrt{d-1}}\right)^{\text{dist}(i,j)}$. Just to give a brief summary of the method here, we use the factor that $|P_{i,l_{\alpha}}|$ and $|P_{l_{\alpha},j}|$ can decompose as a product based on the edges of the path connecting $i$ to $l_{\alpha}$ and $l_{\alpha}$ to $j$ respectively. Let $\hat{r}$ be the first common point of these paths.

The product $|P_{i,l_{\alpha}}||P_{l_{\alpha},j}|$ roughly decomposes as $K_{i,\hat{r}} K_{\hat{r},j}|P_{\hat{r},l_{\alpha}}|^2$, where $K_{i,\hat{r}}$ and $K_{\hat{r},j}$ are two quantities that depend on the path between $i$ to $\hat{r}$ and $j$ to $\hat{r}$ respectively. A naive bound on $K_{i,\hat{r}}$ is $\left(\frac{1}{\sqrt{d-1}}\right)^{\text{dist}(i,\hat{r})}$, while a naive bound on $K_{\hat{r},j}$ is $\left( \frac{1}{\sqrt{d-1}} \right)^{\text{dist}(\hat{r},j)}$. The product of these two terms gives the decay $\left( \frac{1}{\sqrt{d-1}} \right)^{\text{dist}(i,j)}$.

If we fix $\hat{r}$ and sum over all $l_{\alpha}$ that would be possible descendants or $\hat{r}$, the computations in Lemma \ref{lem:pert} show that $\sum_{l_{\alpha}} |P_{\hat{r} l_{\alpha}}|^2$ is $O(1)$. Finally, there are $1 + \text{dist}(i,j)$ choices for the first common point $\hat{r}$.  This is the basic strategy that bound $\sum_{\alpha}|P_{i,l_{\alpha}}||P_{l_{\alpha},j}|$ and finishes the proof of our bound.
\end{proof}

To complete this section, we also include estimates on the differences $|\tilde{G}_{ij}-P_{ij}|$ when at least one of $i$ and $j$ are outside of $\mt$.  First consider the case that $i$ is in $\mt$ and $j$ is in $\mt^c$. Applying the resolvent identity, we have $\tilde{G}|_{\mt \times \mt^c} = \tilde{G}|_{\mt} B \tilde{G}^{(\mt)}$.

\begin{lemma} \label{lem:step4io}
Let $i$ be a vertex in $\mt$ and $j$ be a vertex in $\mt^c$. We have the following estimate on $\tilde{G}_{ij} - P_{ij}$.
\begin{equation}\label{eq:step4io}
|\tilde{G}_{ij}-P_{ij}| \lesssim (\log N)\epsilon' \left(\frac{1}{\sqrt{d-1}} \right)^{\ell -d_i}.
\end{equation}
\end{lemma}

\begin{proof}
Applying the resolvent identity and taking the difference, we have,
\begin{equation}
\tilde{G}|_{\mt \times \mt^c} - P|_{\mt \times \mt^c} = (G|_{\mt} - P|_{\mt}) B P^{(\mt)} +P|_{\mt} B (G^{(\mt)} -P^{(\mt)}) + (G|_{\mt} - P|_{\mt}) B (G^{(\mt)} - P^{(\mt)}).
\end{equation}

As before, the term with two factors of the form $(G - P)$ would give an error of the form $(\epsilon')^2$, which is far smaller than any main term error we might have to consider.

Now, let us analyze the first term; writing out the indices, it will be,
\begin{equation}
\sum_{\alpha} (G|_{\mt} - P|_{\mt})_{i,l_{\alpha}} P^{(\mt)}_{\tilde{c}_{\alpha},j}.
\end{equation}
We first remark that $P^{(\mt)}$ is a  block-diagonal matrix and that $j$ would belong to a block containing at most $O(1)$ elements $\tilde{c}_{\alpha}$. Thus, $P^{(\mt)}_{\tilde{c}_{\alpha},j}$ would be zero  for all except for $O(1)$ many elements. Now, we can apply the estimate from equation \eqref{eq:StabEst} in order to bound the first error term $(G|_{\mt} - P|_{\mt})$.

Now, let us analyze the second error term. Again, in terms of coordinates, we have,
\begin{equation}
\sum_{\alpha} (P|_{(\mt)})_{i l_{\alpha}} (G^{(\mt)} - P^{(\mt)})_{\tilde{c}_{\alpha},j}.
\end{equation}
Again, we have to apply the logic of connectedness. We first remark that $j$ cannot be Green's function connected to more than $O(\log N)$ vertices of the form $c_{\alpha}$ by the construction of our switching event. For all $c_{\alpha}$ to which $j$ is not Green's function  connected, we apply the improved bound $|G^{(\mt)} - P^{(\mt)}| \le (\epsilon')^2 + \phi$. For the other $O(\log N)$ terms, we merely use the error bound $\epsilon'$ instead.

Our total error can be bounded by,
\begin{equation}
\sum_{\alpha} |(P|_{(\mt)})_{i l_{\alpha}}|  ((\epsilon')^2 + \phi) + C \log N \epsilon' \left( \frac{1}{\sqrt{d-1}} \right)^{\ell -d_i},
\end{equation}
where $d_i$ is the distance of the vertex $i$ from the center of the tree $\mt$. From Lemma \ref{lem:treesum}, we can bound $\sum_{\alpha}|(P|_{(\mt)})_{i l_{\alpha}}| (d-1)^{\ell/2} \le (\log N)^{\mathfrak{j}(d) \mathfrak{a} + \mathfrak{a}/2}$. Even after multiplying this constant by $\phi$, we obtain a value that is less than $\epsilon$. Therefore, the error coming from this term is still less than $ C \log N \left( \frac{1}{\sqrt{d-1}} \right)^{\ell -d_i} $.
\end{proof}

\subsection{Changing $Q_{I/O}$ to $\tilde{Q}_{I/O}$}

Our stability estimates are finally enough to understand the shift of $\GG$ to $\tilde{ \GG}$ for $Q_I$ and $Q_O$. The proofs of these two quantities are similar, so we will only prove the results for the shift of $Q_I$.
\begin{lemma}\label{lem:Qtilde}
We have that,
\begin{equation} \label{eq:Qtilde}
|Q_{I/O}(\GG) - \tilde{Q}_{I/O}| \lesssim \frac{1}{N^{1-\mathfrak{c}}}  + (2d -1)^{2\ell}\frac{\text{Im}[m_{\TT}^d(z,w)] + \epsilon' + \frac{\epsilon}{S_g^1}}{N\eta}.
\end{equation}

\end{lemma}
\begin{proof}
Clearly, once we remove the neighborhood $\mt$ and switching set $W_S$, the graphs $ \GG$ and $\tilde{\GG}$ are the same. Thus, our goal is to relate terms of the form $G^{(i)}_{jj}$ to those of the form $G^{(\{i\} \cup \mt \cup W_S)}_{jj}$ via resolvent estimates. We will use $i\mt W_S$ as a shorthand for  set $\{i\} \cup \mt \cup W_S$. Let $\chi(\cdot)$ be the characteristic function indicating whether the vertex in question is of at least distance $\frac{\mathfrak{R}}{4}$ away from any of the vertices $a_{\alpha},b_{\alpha},c_{\alpha}$ that take part in the switching and, furthermore, $j$ has a radius $\frac{\mathfrak{R}}{4}$ tree-like neighborhood. We note that there are $\ll N^{\mathfrak{c}}$ vertices $j$ such that $\chi(j) \ne 1$.

Namely, our goal is to prove the following two inequalities,
\begin{equation} \label{eq:firstpair}
\begin{aligned}
& \frac{1}{Nd} \sum_{j \to i} \chi(j)|G_{jj}^{(i)} - G_{jj}^{(i \mt W_S)}| \le (2d-1)^{2\ell}\frac{\text{Im}[m^d_{\TT}(z,w)] +\epsilon' + \frac{\epsilon}{S_g^1}}{N\eta}\\
& \frac{1}{Nd} \sum_{j \to i} \chi(j) |\tilde{G}_{jj}^{(i)} - G_{jj}^{( i \mt W_S)}| \le (2d-1)^{2\ell}\frac{\text{Im}[m^d_{\TT}(z,w)]+ \epsilon' + \frac{\epsilon}{S_g^1}}{N\eta},
\end{aligned}
\end{equation}

as the proof is similar, we only give the details for the first inequality.

For those values $j$ such that $\chi(j) =0$, we use the fact that $G_{jj}^{(i)}$ and $\tilde{G}_{jj}^{(i)}$ are both $O(1)$. Thus,
\begin{equation}
\frac{1}{Nd} \sum_{j \to i} (1- \chi(j)) |G_{jj}^{(i)} - \tilde{G}_{jj}^{(i)}| \lesssim N^{\mathfrak{c}-1},
\end{equation}
since the sum is non-zero for $\ll N^{\mathfrak{c}}$ many vertices.

By the resolvent identity, we have,
\begin{equation}
G_{jj}^{(i)} = G_{jj}^{(i\mt W_S)} + [G^{(i\mt W_S)} B^* G|_{i\mt W_S} B G^{(i \mt W_S)}]_{jj},
\end{equation}
Here, $B$ is the adjacency matrix for the graph $\GG$ between $i \mt W_S$ and the complement.

We would also have,
\begin{equation}
\tilde{G}_{jj}^{(i)} = G_{jj}^{(i \mt W_S)} + [G^{(i \mt W_S)} \tilde{B}^* \tilde{G}^{(i)}|_{i \mt W_S} \tilde{B} G^{(i \mt W_S)}]_{jj}.
\end{equation}
Here, $\tilde{B}$ is the switched adjacency matrix for the matrix $\tilde{\GG}$. We will consider the matrix product that appears in the end (the $G B^* G B G$) to be purely an error term.

We see that we have,
\begin{equation}
\begin{aligned}
&\frac{1}{Nd} \sum_{(j \to i) \in E} \chi(j) |G_{jj}^{(i)} - G_{jj}^{(i \mt W_S)}| \\ & \le \frac{1}{Nd} \sum_{(j \to i) \in E}  \sum_{x, y \in \delta(\mt W_S)} \sum_{ a,b \in \mt W_S} |G_{jx}^{(i \mt W_S)}| B_{xa} |[G^{(i)}|_{\mt W_S}]_{ab}| B_{by} |G_{yj}^{(i \mt W_S)}|.
\end{aligned}
\end{equation}
$\delta(\mt W_S)$ consists of the boundary vertices in $\GG \setminus \{i\}$ that border a vertex of $\mt \cup W_S$.
Now, for given $x$ and $y$, there are at most $O(1)$ choices for $a$ and $b$ for which one would not have $B_{xa} \ne 0$, and $B_{by} \ne 0$. Now, we argue that $|[G^{(i)}|_{\mt W_S}]_{ab}|$ is $O(1)$ for all pairs of vertices $a$ and $b$ in $\mt \cup W_S$.  

An application of the resolvent identity would give,
\begin{equation}
G_{ab} = G^{(i)}_{ab} +G|_{a \times \{i,i+N\}} (G|_{\{i,i+N\}})^{-1}G|_{\{i,i+N\} \times b}.
\end{equation}
By applying our induction hypothesis from Definition \ref{def:IndHyp} on the terms of $G$, we would know that $G_{ab}$ and the entries of  $G_{a \times \{i,i+N\}}$ and $G_{\{i,i+N\} \times b}$ are all $O(1)$ quantities. One can also apply the induction hypothesis and simple perturbation theory to show that the entries of $(G|_{\{i,i+N\}})^{-1}$ are themselves $O(1)$ quantities. By rearrangement, this means that

\begin{equation}
\frac{1}{Nd} \sum_{(j \to i) \in E} \chi(j) |G_{jj}^{(i)} - G_{jj}^{(i \mt W_S)}|  \lesssim \frac{1}{Nd} \sum_{(j \to i) \in E}  \sum_{x,y \in \delta( \mt W_S)} \chi(j) |G^{(i \mt W_S)}_{jx}| |G^{(i \mt W_S)}_{yj}|.
\end{equation}
Now, we apply the Cauchy-Schwarz inequality to bound this by,
\begin{equation} \label{eq:cauchyschwartzapp}
\frac{1}{Nd}  \sum_{x,y \in \delta(\mt W_S)}  \Big[\sum_{(j \to i) \in E} \chi(j) |G_{jx}^{(i \mt W_S)}|^2 \sum_{(j \to i) \in E} \chi(j) |G_{yj}^{( i \mt W_S)}|^2 \Big]^{1/2}.
\end{equation}

We apply the resolvent identity again to relate $G^{(i \mt W_S)}$ to $G^{(\mt W_S)}$, which we do have estimates to control. Namely, we again see that,
\begin{equation}
G_{jx}^{(i \mt W_S)} = G_{jx}^{(\mt W_S)} - G^{(\mt W_S)}|_{j \times \{i, i + N\}} (G^{(\mt W_S)}|_{\{i,i+N\}})^{-1} G^{(\mt W_S)}|_{\{i,i+N\} \times x}.
\end{equation}
We know that the terms of $G^{(TW_S)}|_{j \times \{i,i+N\}}$ are of $O(1)$ by  Proposition \ref{prop:removeWS}. Recall that since $\chi(j)$ is $1$, $j$ must have a tree-like neighborhood of radius $\frac{\mathfrak{R}}{4}$ in the graph $\mathcal{G}\setminus(\mt \cup W_S)$; as $i$ is a vertex adjacent to $j$. Thus, we can apply estimates from perturbation theory to ensure that   the terms of $ (G^{(\mt W_S)}|_{\{i,i+N\}})^{-1}$ are $O(1)$. From these two facts, we now can assert that, 

\begin{equation}
\chi(j)|G_{jx}^{(i \mt W_S)}|^2 \lesssim |G_{jx}^{( \mt W_S)}|^2 + |G_{ix}^{(\mt W_S)}|^2,
\end{equation}

Now, by applying the Ward identity, we see that,
\begin{equation}
\sum_{(j \to i)} |G_{jx}^{(i \mt W_S)}|^2 \lesssim  \frac{\text{Im}[G_{jj}^{(\mt W_S)}]}{\eta} \lesssim \frac{\text{Im}[m_{\TT}^{d}(z,w)]  + \epsilon' + \frac{\epsilon}{S_g^1}}{\eta}.
\end{equation}

We used the following triangle inequality to estimate $G_{jj}^{(\mt W_S)}$,
\begin{equation}
\begin{aligned}
|G_{jj}^{(\mt W_S)} -m_{\TT}^{d}(z,w)|  \le & |G_{jj}^{(\mt W_S)} - G_{jj}(Ext(B_r(j, \GG), Q_O,Q_I))|\\ &+ |G_{jj}(Ext(B_r(j,\GG), Q_O,Q_I)) -G_{jj}(Ext(B_r(j,\GG), m_{\infty},m_{\infty}))|.
\end{aligned}
\end{equation}
The term $G_{jj}(Ext(B_r(\{j\},\GG), m_{\infty},m_{\infty}))$ is the same as $m_{\TT}^{d}(z,w)$. Now, we can estimate the first term on the right hand side of the above inequality by $\epsilon'$ by using Proposition  \ref{prop:removeTest}. The second term on the right hand side of the above inequality is bounded by $\frac{\epsilon}{S_g^1}$ by our inductive hypothesis on the distance between $Q_{I/O}$ and $m_{\infty}$ from Definition \ref{def:IndHyp} .

We can return to the expression in \eqref{eq:cauchyschwartzapp} to bound the quantity inside the square root by the term above. Now, the number of $x$ vertices on the boundary of $\mt W_S$ is less than $(2d-1)^{\ell}$; the same is true of $y$. Thus, the outer sum in equation \eqref{eq:cauchyschwartzapp} will give a factor of no more than $(2d-1)^{2\ell}$. This will give us the desired inequality in equation \eqref{eq:Qtilde}.
\end{proof}

With this estimate on the difference between $Q_I/Q_O$ and $\tilde{Q}_I/\tilde{Q}_O$,  we can finally return to our proof of Proposition \ref{prop:introT} is simple.

\begin{proof}[Proof of Proposition \ref{prop:introT}]

From Lemmas \ref{lem:Stab}and \ref{lem:step4io}, we have a comparison between $\tilde{G}$ and a $P$ whose extension is given by parameters $Q_I$ and $Q_O$. The last Lemma \ref{lem:Qtilde} suggests that the difference between $Q_I/O$ and $\tilde{Q}_{I/O}$ are small enough that we can apply the perturbation estimates from Theorem \ref{thm:pert}. These estimates are sufficient for the proof of Proposition \ref{prop:introT}. 
\end{proof}

\section{Concentration Estimates} \label{sec:concentration}

The main issue with our stability estimate in equation \eqref{eq:StabEst} is the presence of the term whose decay is $\left(\frac{1}{d-1}\right)^{\frac{|d_j-d_i|}{2}}$. As long as $d_j$ and $d_i$ are relatively small, this term will not decay  as $N$  and $\ell$ increase. Recall that this error term came from estimating the following term,
$$
\frac{1}{d-1}\sum_{\alpha} P_{i,l_{\alpha}} P_{l_{\alpha},j} (\tilde{G}^{(\mt)} - P^{(\mt)} )_{c_{\alpha},c_{\alpha}}.
$$

While the triangle inequality will not show an improvement, the main point is that this new graph came from random switching. Thus, we can try to apply a concentration argument to show that that the quantity that appears above has superior bounds. Loosely speaking, the improvement will improve the constant $\sum_{\alpha} |P_{i,l_{\alpha}}||P_{l_{\alpha},j}|$ to $\sqrt{\sum_{\alpha} |P_{i,l_{\alpha}} P_{l_{\alpha},j}|^2}$. Under the $l_2$ norm on $|P_{i,l_{\alpha}}||P_{l_{\alpha},j}|$, the branching factor of the tree no longer cancels out the effect of the decrease. This  is the main improvement that we want to seek.  To derive this improvement, we use the fact that the $c_{\alpha}$'s are randomly chosen, independent from each other. Using this fact, one would expect better concentration results for most switching neighborhoods.

The main result of this section is the following Theorem,
\begin{thm} \label{thm:mainconcentration}
Let $\GG$ be a graph from $\Omega^{1}_{S, \mathfrak{q},o}$ with $\ell$ neighborhood $\mt$. Consider the following weight families.  These weight families will allow us to understand $Q_I$.

\begin{enumerate}
 \item Let $v$ be a vertex in $\mt$ and let $\tilde{v}$ be an index corresponding to $v$ (so it is either $v$ or $v+N$). Recall our indexing of boundary vertices of $\mt$ as $l_{\alpha}$. We define the in-weights $w^{\text{in}}_{\alpha}:= P_{x l_{\alpha}}P_{l_{\alpha} x}$.
 There are $\lesssim |\mt|$ of these types of indices depending on the choice of the index $x$.
\item  If $\mt$ is a truncated $d$-regular tree, fix a vertex $k$ that is connected to $o$; assume for now that the edge is oriented $k \to o$. Define $A$  to be a set of indices as follows. Consider a vertex $l_{\alpha}$ on the boundary of $\mt$; if the path from $o$ to $v_{\alpha}$ does not pass through $k$, then we include the index $\alpha$ in $A$. We then define the weight $w^{\text{in}}_{\alpha}:= P^{(k)}_{o,l_{\alpha}} P^{(k)}_{l_{\alpha},o}$. $P^{(k)}$ is a shorthand for $G_{o,l_{\alpha}}(Ext(B_r(\mt,\GG^{(i)}),Q_I,Q_O))$. If instead the edge were oriented $o \to k$, we would instead choose the weights $P^{(k)}_{o+N,l_{\alpha}} P^{(k)}_{l_{\alpha},o+N}$.  There are $2d$ such weights of this type.
\end{enumerate}

There exists an event $S_{C}(\GG) \subset S_{G}(\GG)$ such that $1-\mathbb{P} (S_{C}(\GG)) = 1- \mathbb{P}(S_{G}(\GG)) - (|T|+d) \exp[-(\log N)^2]$ such that we have the following estimate,
\begin{equation*}
\begin{aligned}
&\left|\sum_{\alpha} w_{\alpha} (\tilde{G}_{\tilde{a}_{\alpha} +N, \tilde{a}_{\alpha}+N}^{(\mt)} - G_{\tilde{a}_{\alpha} +N, \tilde{a}_{\alpha}+N}(Ext_i(B_{r}(\tilde{a}_{\alpha}, \tilde{\GG}^{(\mt)}), Q_I,Q_O) )) - \sum_{\alpha} w_{\alpha} (Q_I - Y_{i,r}(Q_I, Q_O)) \right|\\
& \lesssim (\log N) \epsilon' \sqrt{\sum_\alpha |w_{\alpha}|^2} + \Big((2d-1)^{2\ell}\frac{\text{Im}[m_{\TT}^{d}] + \epsilon' + \frac{\epsilon}{S_g^1}}{N\eta} + N^{\mathfrak{c} -1} \Big) + \frac{1}{N^{1-\mathfrak{c}}}.
\end{aligned}
\end{equation*}

\end{thm}
\begin{rmk}
We have a similar concentration statement to understand the concentration around $Q_O$. We do not write out the whole statement for simplicity, but mention the appropriate changes to the weights and computed function. In the first type of weight family, we change the weights to $w_{\alpha}^{\text{out}}:= P_{x, l_{\alpha} + N} P_{l_{\alpha}+N, x}$. The function which we try to compute using concentration estimates will be of the form $\sum_{\alpha} w_{\alpha} \tilde{G}_{\tilde{a}_{\alpha}, \tilde{a}_{\alpha}} $.

\end{rmk}

The proof of the above statement will be divided into smaller parts. Since the random variables $\tilde{a}_{\alpha}$ are chosen independently of each other given the graph, $\tilde{G}^{(\mt)}_{\tilde{a}_{\alpha}+N, \tilde{a}_{\alpha}+N}$ are random variables that are almost independent of the weights $w_{\alpha}$ and of each other . However, the fact that they are not truly independent prevents us from directly applying standard concentration estimates.  (The value of the Green's function $\tilde{G}^{(\mt)}_{\tilde{a}_{\alpha} + N, \tilde{a}_{\alpha}+N}$ does not depend on only the local geometry, but on the full graph structure of the switched graph. Thus, changing the value of $\tilde{a}_{\alpha}$ for a single $\alpha$, while keeping all others the same, will have subtle effects on $\tilde{G}^{(\mt)}_{\tilde{a}_{\beta},\tilde{a}_{\beta}}$ for $\beta\ne \alpha$.  The distribution of $\tilde{G}^{(\mt)}_{\tilde{a}_{\alpha},\tilde{a}_{\alpha}}$ will also depend on $l_{\alpha}$ as well.) Our first step is to derive concentration estimates for a closely related family of independent random variables. 

\subsection{Concentration estimates for a related family of independent random variables}

For a vertex $v \in \GG$, we use $(v \mt)$ as a short way to express the union $\{v\}  \cup \mt$. Also let $\chi(\cdot)$ is the indicator function of the event that the vertex is question is of distance at least $\mathfrak{R}/4$ from the vertices that bound $\mt$ and that the vertex in question has a radius $\mathfrak{R}/4$ tree-like neighborhood.
As a shorthand for what will follow, we will define $D_{\alpha}$ as a shorthand for the difference,
\begin{equation}
D^{\text{in}}_{\alpha}:=\chi(c_{\alpha}) (G^{(b_{\alpha}\mt)} - P^{(b_{\alpha}\mt)})_{c_\alpha+N,c_{\alpha}+N},
\end{equation}
where recall that $P^{(b_{\alpha}\mt)}$ was defined to be the Green's function of the extension

\noindent $P^{(b_{\alpha} \mt)}_{ij}:= G(Ext(B_r(\{i,j\}, \GG \setminus (b_{\alpha}\mt )),Q_I,Q_O))$, while $P_{ij}^{(\mt)}$ is simply $G_{ij}(Ext(B_{r}(\{i,j\}), \GG \setminus \mt), Q_I,Q_O)) $.  Note that we specifically chose the unswitched graph here.

\begin{lemma} \label{lem:auxconc}

Consider all of the families of weights $w_{\alpha}$ as given in Theorem  \ref{thm:mainconcentration}
Let $S_C(\GG) \subset S_G(\GG)$ be the set of switching events such that the following event holds for all of the weight families.
\begin{equation}
\begin{aligned}
&\Big| \sum_{\alpha} w_{\alpha}D^{\text{in}}_{\alpha} - \sum_{\alpha} w_{\alpha}(Q_I - Y_{i,r}(Q_I,Q_O))\Big|\lesssim\\&
(\log N) \epsilon' \sqrt{\sum_{\alpha} |w_{\alpha}|^2} +   \Big((2d-1)^{2\ell}\frac{\text{Im}[m_{\TT}^{d}] + \epsilon' + \frac{\epsilon}{S_g^1}}{N\eta} + N^{\mathfrak{c} -1} \Big).
\end{aligned}
\end{equation}
Then, $\mathbb{P}(S_C(\GG) = \mathbb{P}(S_G(\GG)) - (|T|+d) \exp[-(\log N)^2]$.
\end{lemma}
\begin{proof}
We have shown in  Proposition \ref{prop:removeTest} that each of the differences satisfy $|(G^{(\mt)} - P^{(\mt)})_{c_{\alpha}+N,c_{\alpha}+N}|\le \epsilon'$. We can apply the resolvent identity to understand $|G^{(b_{\alpha} \mt)}_{c_{\alpha}+N,c_{\alpha}+N} - P^{(b_{\alpha} \mt)}_{c_{\alpha}+N,c_{\alpha}+N}|$. First of all, we see that, for edges $(i \to j)$,
\begin{equation}
\chi(j)G^{(i\mt)}_{j+N,j+N} =\chi(j)(G^{(\mt)}_{j+N,j+N} - G^{(\mt)}|_{j+N \times \{i,i+N\}} (G^{(\mt)}|_{\{i,i+N\}})^{-1} G^{(\mt)}|_{\{i,i+N\} \times j+N}).
\end{equation}

In addition, we have that,
\begin{equation}
\chi(j)P^{(i \mt)}_{j+N, j+N} =\chi(j)( P^{(\mt)}_{j+N,j+N} - P^{(\mt)}|_{j+N \times \{i,i+N\}}(P^{(\mt)}|_{\{i,i+N\}})^{-1} P^{(\mt)}|_{\{i,i+N\} \times j+N}).
\end{equation}

We have the estimate that $|G^{(\mt)} - P^{(\mt)}| \lesssim \epsilon'$ from Proposition \ref{Prop:GTminusPT}. We also know that entries of $(P^{(\mt)}|_{\{i,i+N\}})^{-1}$ are $O(1)$ quantities when $\chi(j) \ne 0$, since, in this case,  $i$ has at least a radius $\mathfrak{R}$ treelike-neighborhood in $\GG \setminus T$. This allows us to apply a simple perturbation argument to assert that,

\begin{equation}
\begin{aligned}
&\chi(j)\bigg|  G^{(\mt)}|_{j+N \times \{i,i+N\}} (G^{(\mt)}|_{\{i,i+N\}})^{-1} G^{(\mt)}|_{\{i,i+N\} \times j+N} \\
& \hspace{1 cm} - P^{(\mt)}|_{j+N \times \{i,i+N\}}(P^{(\mt)}|_{\{i,i+N\}})^{-1} P^{(\mt)}|_{\{i,i+N\} \times j+N} \bigg| \lesssim \epsilon'.
\end{aligned}
\end{equation}

These facts ensure that,
\begin{equation}
\chi(j)|G^{(i \mt)}_{j+N,j+N} - P^{(i \mt)}_{j+N,j+N}| \le \epsilon',
\end{equation}
uniformly. Furthermore, once we have fixed the graph $\GG^{(\mt)}$, the switched edges are chosen independently of each other. This means that the variables $D^{\text{in}}_{\alpha}$ are independent and identically distributed.

By Hoeffding's inequality, we have the following probability bound,
\begin{equation}
\mathbb{P}\Big( \left| w_{\alpha}(D^{\text{in}}_{\alpha}- \mathbb{E}[D^{\text{in}}_{\alpha}]  ) \right| \ge t \sum_{\alpha}|w_{\alpha}|^2 \epsilon' \Big) \le \exp[-t^2/C],
\end{equation}

for some constant $C$. Here, the probability $\mathbb{P}$ and the expectation $\mathbb{E}$ is only over the randomness of the choice of switched sets, only, not on the randomness of the random graph.

Our goal now is to relate $\mathbb{E}[D^{\text{in}}_{\alpha}]$ to the quantities $Q_I$. We first define the set $\mathcal{E}^{(\mt)}$ to be the set of all edges in the complement of $\mt$ in $\GG$.  By definition,
\begin{equation}
\mathbb{E}[D^{\text{in}}_{\alpha}] =  \frac{1}{|\mathcal{E}^{(\mt)}|}\sum_{(i \to j)} \chi(j) (G^{(i\mt)}_{j+N,j+N} - P^{(i\mt)}_{j+N,j+N}).
\end{equation}
We have earlier shown that all of these quantities are less than $\epsilon'$ uniformly. In addition, we have that $|\mathcal{E}^{(\mt)}| \ge Nd - O((2d-1)^{\ell}) \ge Nd - O(N^{\mathfrak{c}})$.
By adjusting the normalization factor, we have that
\begin{equation}
\mathbb{E}[D^{\text{in}}_{\alpha}] = \frac{1}{Nd} \sum_{(i \to j)} \chi(j)(\tilde{G}^{(i\mt)}_{j+N,j+N} - P^{(i\mt)}_{j+N,j+N}) + O\left( \frac{N^{\mathfrak{c}}}{Nd}\right).
\end{equation}

Now, we relate $G^{(i)}_{j+N,j+N}$ to $G^{(i \mt)}_{j+N,j+N}$ via the Schur complement formula.  We can apply the same argument as found in the proof of equation \eqref{eq:firstpair} to show that,
\begin{equation}
\frac{1}{Nd} \sum_{(i \to j)}\chi(j)|G^{(i)}_{j+N, j+N} -  G^{(i \mt)}_{j+N,j+N}| \lesssim (2d-1)^{2\ell}\frac{\text{Im}[m_{\TT}^{d}] + \epsilon' + \frac{\epsilon}{S_g^1}}{N\eta},
\end{equation}
Furthermore,  $\chi(j)=1$ implies that the vertex $j$ has a radius $\mathfrak{R}/4$ tree-like neighborhood ensures. Thus, we have that,
\begin{equation}
\frac{1}{Nd} \sum_{(i,j)} \chi(j) P^{(i \mt)}_{j+N,j+N} = \frac{1}{Nd} \sum_{(i \to j)} \chi(j) Y_{1,r}(Q_I,Q_O).
\end{equation}

For the remaining quantities with $\chi(j) = 0$, we know that $P^{(i )}_{j+N,j+N}$ is $O(1)$ from Lemma \ref{lem:pert}. Using  Proposition \ref{prop:removeTest} would further imply that the entries $G^{(i )}_{j+N,j+N}$ are also $O(1)$. This bounds the sum of terms $\frac{1}{Nd} \sum_{i \to j} (1- \chi(j)) G^{(i)}_{j+N,j+N}$ by $N^{\mathfrak{c} -1}$. 

Combining these facts, we see that,
\begin{equation}
\sum_{\alpha} w_{\alpha}\mathbb{E}[D^{in}_{\alpha}]= \sum_{\alpha} w_{\alpha} (Q_I(z,w) - Y_{i,r}(Q_i,Q_O)) + O \Big((2d-1)^{2\ell}\frac{\text{Im}[m_{\TT}^{d}] + \epsilon' + \frac{\epsilon}{S_g^1}}{N\eta} + N^{\mathfrak{c} -1} \Big).
\end{equation}

In the final line, we used the fact that $\sum_{\alpha} w_{\alpha} \lesssim 1$.
\end{proof}

To use our concentration estimates for the unswitched graph onto the switched graph, we apply the following lemma, which will show that the values of $\tilde{G}^{(\mathbb{T}})_{c_{\alpha}+N,c_{\alpha}+N}$ are close to those of $G^{(b_{\alpha}\mt) }_{c_{\alpha}+N,c_{\alpha}+N}$ as well as similar results for the index $c_{\alpha}$.

\begin{lemma}\label{lem:rmvsingvert}
Assume that $c_{\alpha}$ has distance at least $\mathfrak{R}/4$ from all of the other vertices $\{a_{\beta},b_{\beta},c_{\beta}\}, \beta\ne \alpha$ that participate in the switching and that the pair ${c_{\alpha},b_{\alpha}}$ is not Green's function connected to $\{c_{\beta},b_{\beta}\}$ for $\beta \ne \alpha$. Then, we have that,
\begin{equation}
|G^{(b_{\alpha} \mt)}_{c_{\alpha},c_{\alpha}} - \tilde{G}^{(\mt)}_{c_{\alpha},c_{\alpha}}| \lesssim (2d-1)^{\ell} \phi^2
\end{equation}

\end{lemma}

\begin{proof}

The desired estimate is a consequence of the following intermediate comparison estimates,
\begin{equation}
\begin{aligned}
& |G^{(\mt W_S)}_{c_{\alpha},c_{\alpha}} - G^{(\mt b_{\alpha})}_{c_{\alpha},c_{\alpha}}| \lesssim (2d-1)^{\ell}(\phi)^2\\
&|G^{(\mt W_S)}_{c_{\alpha},c_{\alpha}} - \tilde{G}^{(\mt)}_{c_{\alpha},c_{\alpha}}| \lesssim (2d-1)^{\ell}(\phi^2)
\end{aligned}
\end{equation}

The proof of both of these estimates are similar, so it will suffice to just prove the first one.

By the resolvent identity, we have,
\begin{equation} \label{eq:intermediate1}
|G^{(\mt W_S)}_{c_{\alpha},c_{\alpha}} - G^{(b_{\alpha}\mt )}_{c_{\alpha},c_{\alpha}}| \le \Big| \sum_{x,y \in W_s \setminus b_{\alpha}} G^{(b_{\alpha} \mt )}_{c_{\alpha},x} ( G^{(b_{\alpha}\mt)}|_{W_s \setminus b_{\alpha}})_{xy}^{-1} G^{(b_{\alpha} \mt)}_{y,c_{\alpha}} \Big|
\end{equation}

Now, the proof of Proposition \ref{prop:removeWS} would also show that the elements $G^{(b_{\alpha}\mt)}_{c_{\alpha},x}$ and $G^{(b_{\alpha} \mt)}_{y,c_{\alpha}}$ are both $\lesssim \phi$. Furthermore, Lemma \ref{lem:switchinv} would show that $|(G^{(b_{\alpha}\mt)}|_{W_S \setminus b_{\alpha}})^{-1}_{xy}|$ is  bounded by $\epsilon'$ if $x$ and $y$ are not indices that correspond to the same vertex, while we can bound this by $O(1)$ if they do correspond to the same index. Using these facts, we see that the bound on the right hand side of \eqref{eq:intermediate1} is $(2d-1)^{\ell} \phi^2 + (2d-1)^{2\ell} \epsilon'\phi^2$ which is dominated by the first term.

For the term involving the switched graph, the Schur complement formula gives us,
\begin{equation}
|\tilde{G}^{(\mt)}_{c_{\alpha},c_{\alpha}} - G^{(\mt W_S)}_{c_{\alpha},c_{\alpha}}| \le \Big|\sum_{x,y \in \delta(W_S)}(G^{(\mt W_S)} B)_{c_{\alpha},x} (\tilde{G}^{(\mt)})_{xy} (B' G^{(\mt W_S)})_{y, c_{\alpha}} \Big|,
\end{equation}
where $\delta(W_S)$ correspond to the vertices that neighbor those of $W_S$ in $\GG \setminus W_S$. We have shown earlier in Proposition \ref{prop:removeWS} that we can bound the entries of $(G^{(\mt W_S)} B)_{c_{\alpha},x}$ and $(B'  G^{(\mt W_S)})_{y, c_{\alpha}}$ by $\phi$. Furthermore, we would also know that $\tilde{G}^{(\mt)}_{xy}$ is $O(1)$ if $x$ and $y$ correspond to the same vertex or is bounded by $\epsilon'$ otherwise by Lemma \ref{lem:switchbound}. The same counting gives us the desired bound.
\end{proof}

With these two lemmas in hand, we can return to a proof of Theorem \ref{thm:mainconcentration}.

\begin{proof}[Proof of Theorem \ref{thm:mainconcentration}]

From Lemma \ref{lem:auxconc}, we know that $\sum_{\alpha} w_{\alpha} D_{\alpha}^{in}$ is close to the desired estimate $\sum_{\alpha} w_{\alpha}(Q_I - Y_{i,r}(Q_I,Q_O))$ on the event $S_C{\GG}$. It suffices to understand the differences,
\begin{equation}
 \sum_{\alpha}|w_{\alpha}||D_{\alpha}^{in} - (\tilde{G}_{\tilde{a}_{\alpha} +N, \tilde{a}_{\alpha}+N}^{(\mt)} - G_{\tilde{a}_{\alpha} +N, \tilde{a}_{\alpha}+N}(Ext(B_{r}(\tilde{a}_{\alpha}, \tilde{\GG}^{(\mt)}), Q_1,Q_2) ) |.
\end{equation}

We know that for all terms in which the pair $c_{\alpha},b_{\alpha}$ is not Green's function connected to $c_{\beta},b_{\beta}$ for $\beta \ne \alpha$ that we know $|\tilde{G}_{a_{\alpha}+N,a_{\alpha}+N}^{(\mt)} - G^{(b_{\alpha}\mt) }_{a_{\alpha}+N,a_{\alpha}+N}|\lesssim (2d-1)^{\ell}\phi^2$ via Lemma \ref{lem:rmvsingvert}.   Furthermore, we would also know that since $\tilde{a}_{\alpha}$ would have a radius $\mathfrak{R}/4$ tree-like neighborhood that,
\begin{equation}
 G_{\tilde{a}_{\alpha} +N, \tilde{a}_{\alpha}+N}(Ext_i(B_{r}(\tilde{a}_{\alpha}, \tilde{\GG}\setminus\mt), Q_I,Q_O) ) = G_{\tilde{a}_{\alpha}+N,\tilde{a}_{\alpha}+N}(Ext_i(B_{r}(\tilde{a}_{\alpha},\GG\setminus \mt),Q_I,Q_O)).
\end{equation}

This deals with all except for $O(1)$ many indices $\alpha$. For the remaining indices, we may use the trivial bound $\epsilon'$.

Thus, we see that,
\begin{equation*}
\begin{aligned}
&\sum_{\alpha} w_{\alpha} (\tilde{G}_{\tilde{a}_{\alpha} +N, \tilde{a}_{\alpha}+N}^{(\mt)} - G_{\tilde{a}_{\alpha} +N, \tilde{a}_{\alpha}+N}(Ext_i(B_{r}(\tilde{a}_{\alpha}, \tilde{\GG}^{(\mt)}), Q_1,Q_2) )) - \sum_{\alpha} w_{\alpha} (Q_I - Y_{i,r}(Q_I, Q_O))\\
&= \sum_{\alpha} w_{\alpha}D^{in}_{\alpha} + O(  \sum_{\alpha}|w_{\alpha}|(2d-1)^{\ell} \phi^2 + \max_{\alpha} |w_{\alpha}| \epsilon')
\end{aligned}
\end{equation*}
Manifestly, $\max_{\alpha}|w_{\alpha}| \epsilon'$ is still less than $(\log N) \epsilon'  \sqrt{ \sum_{\alpha} |w_{\alpha}|^2}$.    Furthermore, we can use the fact that $\sum_{\alpha}|w_{\alpha}| \lesssim 1$  and the value of $\phi$ from Definition \ref{defn:phi} to assert that we can bound $\sum_{\alpha}|w_{\alpha}| (2d-1)^{\ell} \phi^2  \ll N^{1- \mathfrak{c}}$
Since $\max_{\alpha} |w_{\alpha}|$ decays as $(d-1)^{- \ell/2}$, we have sufficient control of  this completes the proof of Theorem \ref{thm:mainconcentration}.
\end{proof}

\subsection{Improved Bounds on $\tilde{G}$}

Given the graph $\GG$ in $\Omega^1_{S,\mathfrak{q},o}$, we choose a switching event $S\in \mathbb{S}_{\GG,o}$ that belongs in $S_C(\GG)$ as in Theorem \ref{thm:mainconcentration}.
With these concentration estimates in hand, we are now in position to prove improved bounds on our Green's functions in order to prove Theorem \ref{thm:secondswitching}.
\begin{proof}[Proof of Theorem \ref{thm:secondswitching}]

We consider the event $S_{C}(\GG)$. 
The proof of Theorem \ref{thm:secondswitching} involves proving the three estimates \eqref{eq:secswitchest1}, \eqref{eq:secswitchest2}, and \eqref{eq:secswitchest3}.

We start with the equation \eqref{eq:secswitchest3}. This is a rather simple consequence of the estimates of Section \ref{sec:step4}.  From Lemma  \ref{lem:step4io}, equation \eqref{eq:step4io}, we know that,
$|G_{oi} - P_{oi}| \lesssim \frac{(\log N)^2 \epsilon'}{(d-1)^{\ell/2}}$. Here,
$$
P_{oi} = G_{oi}(Ext(B_r(i\mt,\tilde{\GG}), Q_I, Q_O)),
$$
In the above expression, we have to replace $Q_I$ and $Q_O$ with $\tilde{Q}_I$ and $\tilde{Q}_O$ as well as change the domain to $B_r(\{i,o\}$ instead of $B_r(i \mt)$.

First of all, we have bounds on the difference $Q_I - \tilde{Q}_I$ as well as $Q_O - \tilde{Q}_O$ from Lemma \ref{lem:Qtilde}. Combining this bound with a slight modification of the perturbation bound from Lemma \ref{lem:pert}.

Namely, we have that,
\begin{equation}
\begin{aligned}
&|G_{oi}(Ext(B_r(i\mt,\tilde{\GG}), Q_I, Q_O)) - G_{oi}(Ext(B_r(i\mt,\tilde{\GG}), \tilde{Q}_I,\tilde{Q}_O))|\\ & \lesssim
\Big(\frac{1}{d-1}\Big)^{\text{dist}(i,j)}(1+ \log N)^3 [|Q_I - \tilde{Q}_I| + |Q_O - \tilde{Q}_O|]
\\& \lesssim \Big(\frac{1}{d-1} \Big)^{\text{dist}(i,j)}(\log N)^3 \Big[\frac{1}{N^{1-\mathfrak{c}}}  + (2d -1)^{2\ell}\frac{\text{Im}[m_{\TT}^d(z,w)] + \epsilon' + \frac{\epsilon}{S_g^1}}{N\eta} \Big].
\end{aligned}
\end{equation}
Finally, we can use Lemma \ref{lem:nbdpert} to change the domain used in the Ext operator from $B_r(i \mt,\GG)$ to $B_r( \{i,o\})$. This introduces the terms involving $|\tilde{Q}_I +\tilde{Q}_O - 2 m_{\infty}|$, etc. on the right hand side of equation \eqref{eq:secswitchest3}. Finally, we can combine all of these estimates using a triangle inequality to finish the proof of \eqref{eq:secswitchest3}.

{ We remark here that, for the purposes of an induction argument, it is important to ensure that the term $\frac{(\log N)^2 \epsilon'}{(d-1)^{\ell/2}}$ would be less than the original error term $\epsilon$ . This guarantees that a switching would give a provable improvement to the error bounds. It is very important that we had the constant bound $\mathfrak{j}(d) < \frac{1}{2}$; this constant bound is what allows us to compensate the $(\log N)$ factors that appear in $\epsilon'$ with $(d-1)^{\ell/2}$ in the denominator}. 

The derivation equation \eqref{eq:secswitchest2} from equation \eqref{eq:StabEst} is quite similar to the strategy we used to prove equation \eqref{eq:secswitchest3}. On a high level, we know that $G_{oi}$ should be close to $ G_{oi}(Ext(B_r(i\mt,\tilde{\GG}), Q_I, Q_O))$ due to the estimates from Section \ref{sec:step4}, specifically equation \eqref{eq:StabEst}. We still need to replace the appearances of $Q_I$ and $Q_O$ in this bound with the appropriate $\tilde{Q}_I$ and $\tilde{Q}_O$. These rely on the same perturbation estimates we have used earlier.


However, we now need to find an improvement in the second term of the right hand side of equation \eqref{eq:StabEst}; the first term is sufficiently small.  Tracing through the proof of equation \eqref{eq:StabEst}, This term comes from an attempt to bound the term on the second line of equation \eqref{eq:somesplit}. Namely, the sum of the on-diagonal terms.

\begin{equation} \label{eq:impterm}
\frac{1}{d-1} \sum_{\alpha} P_{i,l_{\alpha}} P_{l_{\alpha},j}(\tilde{G}^{(\mt)}- P^{(\mt)})_{\tilde{c}_{\alpha}, \tilde{c}_{\alpha}}
\end{equation}
We can replace $P^{(\mt)}_{\tilde{c}_{\alpha},\tilde{c}_{\alpha}}$ is defined with respect to the domain $B_r( c_{\alpha}\mt,\GG^{(\mt)})$. This can be replaced with the extension with respect to the domain $B_r(c_{\alpha}, \GG^{(\mt)})$ with an error term of the same order of magnitude as the right hand side of \eqref{eq:secswitchest2}. Now, since $\sum_{\alpha}|P_{i,l_{\alpha}}||P_{l_{\alpha},i}| \lesssim 1$, it causes not issue to apply a triangle inequality.

After the appropriate replacement of the the domain of the extension, the term in \eqref{eq:impterm} is controlled by the estimates of Theorem \ref{thm:mainconcentration}. This completes the desired bound of equation \eqref{eq:secswitchest2}. The proof of \eqref{eq:secswitchest1} requires very similar manipulations as the previous two, but uses the second family of weights $w_{\alpha}$ from Theorem \ref{thm:mainconcentration} rather than the first.
\end{proof}

\appendix

\section{Collected estimates on Green's function of infinite graphs}

\subsection{ Bounds on $G_{ij}$ for tree-like and almost tree-like graphs}
Recall our inductive procedure to compute for a directed tree, all Green's function values 
(via the Schur complement formula), starting from the solution $m_\infty$ of \eqref{eq:self-conc}. 
The success of this procedure requires that the sum of $|G_{ij}|^2$ over all vertices $j$  
of graph-distance $k$ from $i$, be bounded uniformly in $k$. While the naive bound 
$G_{ij} \le \left(\sqrt{d-1}\right)^{-\text{dist}(i,j)}$ is not strong enough to compensate 
the branching factor of the tree, our next bound on $m_{\infty}$, helps 
control the latter branching factor. 



\begin{lemma} \label{lem:minftybnd}
For  some $C_0=C_0(d)$, $d \ge  3$ finite, any $z \in \mathbb{C}_+$, $w \in \mathbb{C}$ 
and any solution $m_{\infty}: \mathbb{C}_+ \to \mathbb{C}_+$ of \eqref{eq:self-conc}, setting 
\begin{equation}\label{dfn:XY}
X:= |m_{\infty}|^2\,, \qquad
Y:= \frac{d}{d-1} |m_{\infty}^{uod}|^2 
= \frac{d}{d-1} \frac{|w|^2 |m_\infty|^2}{\big|z + \frac{d}{d-1} m_{\infty}\big|^2} \,,
\end{equation}
we have that for some finite $C_0=C_0(d)$,
\begin{equation}\label{eq:lem31}
S_g^1 = 1 - X - Y \ge \frac{1}{2} (\Im(z) \wedge 1) > 0  \,, \qquad |m^{sd}_\infty| \le C_0 \, \qquad |m^d_{\mathcal{T}} | \le C_0.
\end{equation}
\end{lemma}
\begin{proof} Taking $(1/\cdot)$ of both sides of \eqref{eq:self-conc}, results with
\begin{equation}\label{eq:ident}
\frac{1}{m_{\infty}} =  \frac{|w|^2}{z+ \frac{d}{d-1} m_{\infty} } -z - m_{\infty} \,.
\end{equation}
The imaginary part of this identity is precisely
\begin{equation*}
-\frac{\Im(m_{\infty})}{|m_{\infty}|^2} = -\frac{|w|^2 \big(\Im(z) + \frac{d}{d-1}\Im(m_{\infty})\big)}{|z + \frac{d}{d-1} m_{\infty}|^2} - \Im(z) - \Im(m_{\infty}).
\end{equation*}
Set $\Gamma:=\frac{\Im(z)}{\Im(m_\infty)}>0$ and multiply the preceding by 
$-\frac{|m_\infty|^2}{\Im(m_\infty)}$,  to arrive after some algebra at 
\begin{equation}\label{eq:Sg-ident}
S_g^1 = 1 - X - Y = \Gamma \Big( X + \frac{d}{d-1} Y \Big). 
\end{equation}
With $\Gamma>0$,  necessarily $X + Y < 1$ and in particular  
$\Im(m_\infty) \le \sqrt{X} \le 1$.   The \abbr{rhs} of \eqref{eq:Sg-ident} is thus at least
$\Im(z) (1-S_g^1)$ and 
the \abbr{lhs} of \eqref{eq:lem31} follows.

Next, the real part of the identity \eqref{eq:ident} amounts to 
\[
\frac{\Re(m_{\infty})}{|m_{\infty}|^2} = \frac{|w|^2 \big(\Re(z) + \frac{d}{d-1}\Re(m_{\infty})\big)}{|z + \frac{d}{d-1} m_{\infty}|^2} - \Re(z) - \Re(m_{\infty}).
\]
With $X,Y \in [0,1]$, it follows after some algebra that 
\[
\frac{\Re(m_{\infty})}{\Re(z)} =  \frac{\frac{d-1}{d} Y -X}{1-Y+X} \ge  - \frac{1}{2} \,.
\]
Since $\Im(m_\infty)>0$ and $\Im(z)>0$, this implies in turn that also $\Re(m_\infty/z) \ge -1/2$. 
With $|m_\infty| \le 1$, we thus deduce from \eqref{eq:m-infty-rel} that for any $d \ge 3$,
\[
|m_{\infty}^{sd}| \le \frac{|1 + m_{\infty}/z|}{|\frac{d-1}{d} + m_{\infty}/z|} 
\le \sup_{\Re(S) \ge -1/2} \Big\{ \frac{|1+S|}{|\frac{d-1}{d} + S|} \Big\} := C_0(d)
\]
is finite, as claimed.

Now, for $m^d_{\mathcal{T}}$, we see that we have,
\begin{equation}
\begin{aligned}
    \frac{1}{m^d_{\mathcal{T}}} &= \frac{|w|^2}{z + \frac{d}{d-1}m_{\infty}} - \left(z + \frac{d}{d-1}m_{\infty} \right) = \frac{1}{m_\infty} - \frac{m_{\infty}}{d-1}\\
    & = \frac{(d-1) - m_{\infty}^2}{(d-1)m_{\infty}}
\end{aligned}
\end{equation}
Thus,
\begin{equation}
|m^d_{\mathcal{T}}| = \frac{(d-1)|m_{\infty}|}{|d-1 - m_{\infty}^2|} \le \frac{d-1}{d-2}.
\end{equation}
\end{proof}

Using Lemma \ref{lem:minftybnd}, we next show that $\sum_{y \in \partial B_k(\hat{r})} |G_{\hat{r},y}|^2$
is bounded, uniformly in $k$, for any one of the directed trees of interest to us. 
This computation from the root is relevant for
the coefficient of various first order terms that will appear in our perturbation series. 
\begin{lemma} \label{lem:treesum}
For the trees $\TT_1$ and $\TT_2$ from Definition \ref{def:digraphtree} 
and any $k \ge 1$,
\begin{equation}\label{eq:treesum}
\begin{aligned}
&A_{1,k}:=\sum_{v \in \partial B_k(\hat{r}) ,o_v\in \{0,\aleph\}} |(G_{\TT_1})_{\hat{r} + \aleph, v +o_v}|^2 \lesssim 1,\\
&A_{2,k}:=\sum_{v \in \partial B_k(\hat{r}), o_v\in \{0,\aleph\}} |(G_{\TT_2})_{\hat{r},v +o_v}|^2 \lesssim 1,
\end{aligned}
\end{equation}
where the implicit constants do not depend on $k$. The corresponding quantities 
for the Green's function $G_{\TT}$ of the `pure' $d$-regular directed tree, are 
also bounded uniformly in $k$.


\end{lemma}
\begin{proof} Suppose $l \in \TT_1$ and $\hat{r}=v_0, v_1,\ldots,v_k=l$ is the (undirected)
path in $\TT_1$ from $r$ to $l$. 
Note that upon removal of the edge from $r$ to $v_1$, the vertex $l$ is in a copy 
$\TT^{(\hat{r})}_1$ or $\TT^{(\hat{r})}_2$, respectively, 
of $\TT_1$ or $\TT_2$ rooted at $v_1$.  Thus,
similarly to the derivation of \eqref{eq:zeroselfconc}-\eqref{eq:firstselfconc}, but now 
utilizing \eqref{eq:secondGreen}, we get in case of an out-edge from $\hat{r}$ leading to $v_1$, that
\begin{equation}\label{eq:recurs1}
\begin{aligned}
& (G_{\TT_1})_{\hat{r} + \aleph, l +o_l} = \frac{-1}{\sqrt{d-1}} (G_{\TT_1})_{\hat{r} + \aleph, \hat{r}} (G_{\TT^{(\hat{r})}_1})_{v_1 +\aleph, l+o_l}, \\
& (G_{\TT_2})_{\hat{r}, l +o_l} = \frac{-1}{\sqrt{d-1}} (G_{\TT_2})_{\hat{r},\hat{r}} (G_{\TT^{(\hat{r})}_1})_{v_1 +\aleph, l+o_l},
\end{aligned}
\end{equation}
whereas if an in-edge of $\hat{r}$ is connected to $v_1$, we get instead that 
\begin{equation}\label{eq:recurs2}
\begin{aligned}
&(G_{\TT_1})_{\hat{r} + \aleph, l +o_l}= \frac{-1}{\sqrt{d-1}} (G_{\TT_1})_{\hat{r} + \aleph, \hat{r}+\aleph} (G_{\TT^{(\hat{r})}_2})_{v_1, l+o_l} \,,\\
& (G_{\TT_2})_{\hat{r}, l +o_l}= \frac{-1}{\sqrt{d-1}} (G_{\TT_2})_{\hat{r},\hat{r}+\aleph} (G_{\TT^{(\hat{r})}_2})_{v_1, l+o_l} \,.
\end{aligned}
\end{equation}
We take $|\cdot|^2$ of both sides, sum over $o_l$ and all vertices $l \in \partial B_k(\hat{r})$ 
and substitute our expressions \eqref{eq:self-conc}--\eqref{eq:m-infty-rel}
for  $(G_{\TT_1})_{\hat{r} + \aleph, \hat{r}}$, $(G_{\TT_1})_{\hat{r}+\aleph,\hat{r}+\aleph}$, 
$(G_{\TT_2})_{\hat{r},\hat{r}+\aleph}$ and $(G_{\TT_2})_{\hat{r},\hat{r}}$, to arrive at
the following recursion relations 
\begin{equation}\label{eq:recurs}
\begin{aligned}
&A_{1,k} = Y  A_{1,k-1} + X A_{2,k-1}, \\
& A_{2,k} = X A_{1,k-1} + Y A_{2,k-1},
\end{aligned}
\end{equation}
in terms of $(X,Y)$ of \eqref{dfn:XY}. Solving the recursion \eqref{eq:recurs} leads to 
\begin{equation}
\begin{bmatrix}
A_{1,k}\\ A_{2,k}
\end{bmatrix}
=\begin{bmatrix}
Y & X  \\
X & Y
\end{bmatrix}^k
\begin{bmatrix}
A_{1,0}\\
A_{2,0}
\end{bmatrix}.
\end{equation}
The central matrix has eigenvalues $Y \pm X$. From \eqref{eq:lem31} we know
that $|Y \pm X| < 1$, which upon transforming into the basis of 
eigenvectors, suffices for the uniform in $k$ boundedness of $A_{1,k}$ and $A_{2,k}$.

In case of the tree $\TT$, we derive the corresponding recursion via the Schur 
complement formula after the removal of the root. Indeed, we have that,
\begin{equation}
\begin{aligned}
\sum_{l \in \partial B_{k}(\hat{r}),o_l}|(G_{\TT})_{r + o_r , l+ o_l}|^2 &= \frac{d}{d-1} |(G_{\TT})_{\hat{r} + o_{\hat{r}},\hat{ r}}|^2  \sum_{l \in \partial B_{k-1}(\hat{r}),o_l} |(G_{\TT_1})_{\hat{r}+\aleph, l  +o_l }|^2\\&  + \frac{d}{d-1} |(G_{\TT})_{\hat{r}+o_{\hat{r}},\hat{r}+\aleph}|^2 \sum_{l \in \partial B_{k-1}(\hat{r}),o_l}|(G_{\TT_2})_{\hat{r},l+o_l}|^2.
\end{aligned}
\end{equation}
We conclude by observing that both terms are bounded, uniformly in $k$, per \eqref{eq:treesum}. 
\end{proof}

Proceeding to generalize the computation in Lemma \ref{lem:treesum} to other higher
order terms in perturbation expansions, we first prove a general factorization property 
of the Green's functions on directed trees. 


\begin{lemma}\label{lem:Factorization} 
Denote by $p$ the path in $\TT$, $\TT_1$, or $\TT_2$,
between two indices $x$ and $y$ (each being either $u$ or $u+\aleph$ for some vertex $u$).
For $\TT_1$ and $\TT_2$, specify the root as $\hat{r}$ and denote by $x \wedge y$ the 
common ancestor of $x$ and $y$ of the largest depth. 
Let $E_{v}$ denote the types of the pair of edges in $p$ adjacent to each 
vertex $v \in p$ (i.e., whether they are in-edges or out-edges). For some explicit 
function $K(E) \le \frac{1}{\sqrt{d-1}}$ depending only on $z,w,d$ and the
edge types and some universal constant $C$, we have the following bound
on the corresponding Green's function
\begin{equation} \label{eq:fact}
|G_{x,y}| \le C \prod_{\substack{v \in p\\ v \ne x,y,x \wedge y}} K(E_v),
\end{equation}
uniformly over $w \in \mathbb{D}$, a fixed compact region in $\mathbb{C}$  
(where $v \ne x \wedge y$ means that $v$ is not the ancestor of the points $x$ and $y$ on the path to the root in $\mathcal{T}_1$ and $\mathcal{T}_2$.)

\end{lemma}
\begin{rmk} While this shows that $G_{x,y}$ has `nice' factorization properties, one has
to be careful here. Indeed, \eqref{eq:fact} is more complicated due to the dependence of 
$G_{x,y}$ on the specific orientations along the path connecting $x$ and $y$.
\end{rmk}
\begin{proof} In principle, this factorization is a consequence of the Schur complement formula. However, 
as mentioned, one has to be aware of the orientations along $p$ while performing a case by case analysis. 
First, if either the index $x$ or $y$ corresponds to the root vertex, then \eqref{eq:fact} is fairly straightforward.
Indeed, as we saw in the recursions \eqref{eq:recurs1}-\eqref{eq:recurs2}
 leading to Lemma \ref{lem:treesum}, in this case each of the 
$K(E_v)$ is the absolute value of one of 
$(G_{\TT_1})_{v+\aleph,v}, (G_{\TT_1})_{v+\aleph,v+\aleph}, (G_{\TT_2})_{v,v}$ 
or $(G_{\TT_2})_{v,v+\aleph}$ times $\frac{1}{\sqrt{d-1}}$. Further, for any intermediate vertex 
$v \in p$,  $v \ne x,y$, the specific Green's function entry to choose among these
four options depends only on the orientation of the edges along $p$ 
immediately before it and immediately after it. Each of the factors corresponding to our pair of
fixed indices $x$ and $y$ is covered by a universal  finite constant $C_0$ 
as long as $|w|$ is uniformly bounded.

Since $\TT$ is regular, in that case we shall assume with no loss of generality that $x$ is the root. 
The more delicate case is when neither of $x$ or $y$ are at the root in either $\TT_1$ 
or $\TT_2$. Due to duality, it suffices to proceed with the detailed analysis only for 
$\TT_1$, and we assume first that $x \wedge y \ne \hat{r}$ and denote by $\ell \ge 1$ the length of
the path $p_{x \wedge y}$ from $\hat{r}$ to $x \wedge y$. Denoting by $c$ the child of $\hat{r}$ which is 
on this path, upon removing the edge from $\hat{r}$ to $c$, we get from \eqref{eq:thirdGreen} that
\begin{equation}\label{eq:recurs-xy}
G_{x,y} = G^{(\hat{r})}_{x,y} + \frac{1}{d-1} G^{(\hat{r})}_{x,c+\tilde{o}} (G_{\TT_{1}})_{c+o,c+o} G^{(\hat{r})}_{c+\tilde{o},y}
\end{equation}
$o$ and $\tilde{o}$ are either $\aleph$ or $0$ depending on the orientation of the 
edge connecting $c$ to $\hat{r}$. The quantity $(G_{\TT_{1}})_{c+o,c+o}$ is either
$m_\infty$ or $m_\infty^{sd}$ and in either case bounded in absolute value by the 
 universal, finite $C_0$. Note that the tree $\TT_1^{(\hat{r})}$ after the 
removal of $r$ is either a $\TT_1$ rooted at $c$ or a $\TT_2$ rooted at $c$. 
Thus, following \eqref{eq:recurs1}-\eqref{eq:recurs2} yields the product bound of \eqref{eq:fact} 
for both $G^{(\hat{r})}_{x,c+\tilde{o}}$ and $G^{(\hat{r})}_{c+\tilde{o},y}$. Combining the latter 
two bounds gives the terms $K(E_v)$ for all vertices $v \in p$ except possibly 
$x$, $y$ and $x \wedge y$, as well as the product of $K(E_u)^2 \le \frac{1}{d-1}$ over
all $u \in p_{x \wedge y}$, $u \ne \hat{r}$. Consequently, 
\[
|G_{x,y}| \le |G^{(\hat{r})}_{x,y}| + C_0^3 (d-1)^{-\ell} \prod_{\substack{v \in p\\ v \ne x,y,x \wedge y}} K(E_v) \,.
\]
By the preceding reasoning, when starting with $G^{(\hat{r})}_{x,y}$ instead of $G_{x,y}$, 
we get the same bound on the second term on the \abbr{rhs} of \eqref{eq:recurs-xy}, 
except for reducing $\ell$ by one. Thus, setting $\widehat{C}_i=C_0^3 \sum_{k \ge i} (d-1)^{-k}$, $i  \ge 0$,
and iterating till all that remains from $p_{x \wedge y}$
is the common ancestor $x \wedge y$, we conclude that 
\begin{equation}\label{eq:bound-star}
|G_{x,y}| \le |(G_{\TT_{(\star)}})_{x,y}| + \widehat{C}_2 \prod_{\substack{v \in p\\ v \ne x,y,x \wedge y}} K(E_v) \,,
\end{equation}
where $\TT_{(\star)}$ is either a $\TT_1$ tree or a $\TT_2$ tree, now rooted
at $x \wedge y$. Turning to deal with $|(G_{\TT_{(\star)}})_{x,y}|$, we appeal once more to 
\eqref{eq:thirdGreen}, where since now $\hat{r}=x \wedge y$, the first term on the \abbr{rhs} is zero 
at $x,y$. Assuming \abbr{wlog} that $\TT_{(\star)}$ is a $\TT_1$ tree and 
with $c_x \ne c_y$ denoting the children of $x \wedge y$ along the path $p_x$ and $p_y$
leading in that tree to $x$ and $y$ respectively, we find that 
\begin{equation}\label{eq:recurs-wedgexy}
(G_{\TT_{(\star)}})_{x,y} = \frac{1}{d-1} G^{(\star)}_{x, c_x + \tilde{o}_1} 
(G_{\TT_1})_{r+ o_1, r + o_2} G^{(\star)}_{c_y + \tilde{o}_2, y},
\end{equation}
where $G^{(\star)}$ denotes the Green function of $\TT_{(\star)} \setminus (x \wedge y)$,
and $\tilde{o}_i$, $o_i$ are $0$ or $\aleph$, depending on the 
types of edges connecting $c_x$ and $c_y$ to $x \wedge y$. We use 
as before the factorization properties of the disjoint sub-trees of 
$\TT_{(\star)} \setminus (x \wedge  y)$ containing $x$ and $y$, to see that 
\[
|G^{(\star)}_{x, c_x + \tilde{o}_1}| \le C_0  \prod_{\substack{v \in p_x \\ v \ne x, x \wedge y}} K(E_v) \,, \qquad 
|G^{(\star)}_{c_y+\tilde{o}_2,y}| \le C_0 \prod_{\substack{v \in p_y  \\ v \ne y, x \wedge y}} K(E_v) \,.
\] 
Plugging the latter bounds and our usual bound $|(G_{\TT_1})_{\hat{r} + o_1, r + o_2}| \le  C_0$
into \eqref{eq:recurs-wedgexy} then combining with \eqref{eq:bound-star}, we finally arrive at
\eqref{eq:fact} (with $C=\widehat{C}_1$ a universal constant).
\end{proof}

Building on the preceding factorization property and bounding the expressions $\{K(E_v)\}$, we now
generalize Lemma \ref{lem:treesum}, starting at any $x  \in \TT_1$ (and not only at
its root).

\begin{lemma} \label{lem:factsum}
Let $\mathcal{K}_{xy}$ denote the \abbr{rhs} of \eqref{eq:fact}. Then,  for the tree 
$\TT_1$ with root $r$, uniformly in $k \ge 1$, $w \in \mathbb{D}$ compact 
and  $x \in \TT_1$,
\begin{equation}
\sum_{y \in B_k(x)} |\mathcal{K}_{xy}|^2 \lesssim  k.
\end{equation}
\end{lemma}
\begin{proof}
Let $p_x$ denote the path from $x$ to the root $\hat{r}$. Let $v_l$ denote the vertex on $p_x$ that is of distance $l$ from $x$. We let $C(v_l)$ denote the children of $v_l$ that are of distance less than $k-l$ from $v_l$ aside from those  that can only reached using an edge of $p_x$. The product expression \eqref{eq:fact} gives us,
\begin{equation} \label{eq:somesum}
\sum_{y \in B_k(x)} |\mathcal{K}_{xy}|^2 = \sum_{l=0}^k  \left(\frac{1}{d-1}\right)^{l-1} \sum_{z \in C(v_l)} |\mathcal{K}_{v_l z}|^2.
\end{equation}

Now, our main observation is that $|\mathcal{K}_{v_l z}|^2$ can be roughly understood as $|P_{v_l z}|^2$ in a tree that uses $v_l$ as a root. More formally, let $ch^l_1,\ldots , ch^l_{2d-2}$ denote the direct children of $v_l$ in $C(v_l)$ and let $C(ch^l_m)$ denote the children of $ch^l_m$ in $C(v_l)$, while we let $T_{ch^l_m}$ denote the tree with root $ch^l_m$.
We see we have,
\begin{equation}
\sum_{z \in C(v_l)}|\mathcal{K}_{v_l z}|^2 = \sum_{m=1}^{2d-1} \sum_{z \in C(ch^l_m)} |G_{ch^l_m, z}|^2.
\end{equation}
Each of the internal sums on the right hand side can be explicitly computed and shown to be $\lesssim k-l$ by Lemma \ref{lem:treesum}.  The outer sum will just give a constant factor outside. Finally, we can evaluate the full  sum in equation \eqref{eq:somesum}. The summability of the exponential series will show that this will be bounded by some universal constant times $k$, as desired.
\end{proof}

The following lemma allows us to generalize our bounds on the decay of Green's function values with distance to the case of almost tree-like graphs with at most one cycle. The most important part is the relation of the Green's function of the graph to the Green's function of the simply connected cover, which is the infinite tree.

\begin{lemma} \label{lem:excess1}
Consider a digraph $\GG$ with excess at most 1; let $i$ and $j$ be vertices in $\GG$. Let $p_1$ be the shortest path connecting $i$ and $j$ in $\GG$ and let $p_2$ be the second shortest non-backtracking path connecting $i$ and $j$ in $\GG$. Furthermore, let $sp_1$ be a sequence of symbols $fw$ or $rv$ indicating whether each edge in the path $p_1$ were traversed in the forward or reverse direction, and let $sp_2$ be the same object for the path $p_2$. Given these sequences $sp_1$ and $sp_2$, consider the `pure' infinite tree $\TT$ and vertices $v_1$ and $v_2$ in $\TT$ that are reachable from the root $\hat{r}$ using a sequences of edges that match $sp_1$ and $sp_2$, respectively. For example, if $sp_1$ involved only forward edges, then $v_1$ should be a vertex that is reachable from $\hat{r}$ using only forward edges.

 We can derive the following estimate on the Green's function extension,
\begin{equation}
\begin{aligned}
|G(Ext(\GG,m_{\infty},m_{\infty}))_{i + o_i, j+o_j}| &\lesssim |(G_{\TT})_{\hat{r}+o_i,v_1+o_j}| + |(G_{\TT})_{\hat{r}+o_i,v_2+o_j}|\\
&\lesssim\left[ \frac{\max\{|m_{\infty}|,|m^{uod}_{\infty}|\}}{\sqrt{d-1}} \right]^{\text{dist}(i,j)}.
\end{aligned}
\end{equation}
Here, $o_i$ and $o_j$ can be $0$ or $\aleph$ to distinguish the two indices that correspond to the same vertex.

\end{lemma}

\begin{rmk}
In the top line of the preceding bound, we
distinguish the specific edges that are used in the path connecting $i$ and $j$
in order to be able to derive sharp bounds when large neighborhoods are considered. 
\end{rmk}

\begin{proof}

The extension of $Ext(\GG,m_{\infty},m_{\infty})$ involves adding an appropriate infinite tree to every boundary vertex of the digraph $\GG$ that does not yet have full degree. Since we have assumed that the digraph $\GG$ has excess at most 1, this would imply that there is at most one cycle in the extension graph.

Now, we can always relate the Green's function of a graph to its infinite cover as follows. First, we remark that the Green's function is the unique solution to the following system of equations,
\begin{equation} \label{eq:definingeq}
\begin{aligned}
&\delta_{ik} + z G_{ik} = w G_{\aleph + i,k} + \frac{1}{\sqrt{d-1}} \sum_{(i \to j) \in E} G_{\aleph +j,k}\\
& z G_{i, \aleph +k}= w G_{\aleph + i, \aleph +k} + \frac{1}{\sqrt{d-1}} \sum_{(i \to j) \in E} G_{\aleph+j, \aleph+k}\\
& z G_{\aleph +i,k} = \overline{w} \, G_{ik} + \frac{1}{\sqrt{d-1}} \sum_{(j \to i) \in E} G_{j, k}\\
& \delta_{ik} + z G_{\aleph +i, \aleph+k} = \overline{w} \, G_{i,\aleph +k} + \frac{1}{\sqrt{d-1}} \sum_{(j \to i) \in E}G_{\aleph +j, \aleph+k}.
\end{aligned}
\end{equation}

We argue that if we have a new graph $\hat{H}$, with Green's function $\hat{G}$, that is a lift of the graph $H$, with covering map $\pi$. Then, we would have the following relationship between the Green's function $G$ of the original graph and the Green's function of the covering graph.
\begin{equation} \label{eq:propcover}
\begin{aligned}
&G_{ij} = \sum_{y:\pi(y) = j} G_{x y}\\
&G_{\aleph + i, j} = \sum_{y:\pi(y) = j} G_{\aleph+x,y}\\
& G_{i,j+\aleph} = \sum_{y: \pi(y) = j} G_{x,\aleph+y}\\
& G_{i+\aleph,j+\aleph} = \sum_{y:\pi(y)=j} G_{x+\aleph, y+\aleph}.
\end{aligned}
\end{equation}
Here, $x$ is a vertex in $\hat{H}$ that maps to $i$ under the map $\pi$; the value on the right hand side does not depend on the choice of $x$ due to translation invariance. It is implicitly implied that the above results hold as long as the quantities that appear on the right hand side are convergent. In order to check the above relation, we substitute the proposed equalities in equation \eqref{eq:propcover} into the defining equations of equation \eqref{eq:definingeq} and check that equality holds. This mainly involves realizing that if we choose $\hat{i}$ so that $\pi(\hat{i})=i$, then we can parameterize the vertices $j$ such that $i \to j$ is an edge by some lifted vertex $\hat{j}$ such that $\pi(\hat{j}) = j$ and $\hat{i} \to \hat{j}$ is a edge in $\hat{H}$. This confirms the first two equations in equation \eqref{eq:definingeq}.  We can do the same thing for edges $j \to i$ as well, and this verifies the last two equations of \eqref{eq:definingeq}.

We can use the above formula to compute the Green's function of $Ext(\GG,m_{\infty},m_{\infty})$ by setting $H=Ext(\GG,m_{\infty},m_{\infty})$ . This graph has only one cycle due to our assumption on the excess; thus, the covering graph $\hat{H}$ is an infinite tree. The preimage of any vertex $i \in H$ will correspond to the set of non-backtracking paths in $\hat{H}$.  $\sum_{y:\pi(y)=j}$ now merely correspond to a `sum' over non-backtracking paths.

 For points $i$ and $j$ that are of distance $k$ in $H$, the non-backtracking paths all have lengths at least $k$ and the different non-backtracking paths have length that differ by $|C|$ where $|C|$ is the length of the unique simple cycle in $H$. Recall that $p_1$ was the shortest path between $i$ and $j$ while $p_2$ was the second-shortest non-backtracking path between $i$ and $j$.  We have the following structural properties on the paths $p_j$. We can divide $p_2$  as $Q_i \cup C \cup Q_o$ where $C$ is the cycle and further shortest backtracking paths $p_j$ are $Q_i \cup C \cup \ldots \cup C \cup Q_o$ where there are $j-1$ copies of $C$. The path $p_1$ must be either included completely in $Q_o$ or equal to the union $Q_i \cup Q_o$. 


We also have the following property, let $v$ be a fixed lift of $i$ in $\hat{H}$ and $w_1$ be the lift of $j$ in $\hat{H}$ that is closest to $v$. Then, if $\hat{p}_1$ is the path between $v$ and $w_1$, and $\hat{sp}_1$ is the set of labels ( either $f$ or $b$) of the edges of $\hat{p}_1$, then $\hat{sp}$ is the same sequence as $sp_1$. In addition, let $w_2$ is the lift of $j$ in $\hat{H}$ that is the second closest to $v$. If $\hat{p}_2$ is the path between $v$ and $w_2$, and $\hat{sp}_2$ is the set of labels of this path, then $\hat{sp}_2$ is the same as $sp_2$.

Thus, we see that
$$
\begin{aligned}
&|G_{ij}(Ext(\GG, m_{\infty}, m_{\infty})| \le \sum_{n=1}^{\infty} |(G_{\TT})_{v,w_n}|
 \le  |(G_{\TT})_{v,w_1}| +  |(G_{\TT})_{v,w_2}| \frac{1}{1 - 
\frac{|m_{\infty}| \vee |m_\infty^{uod}|}{\sqrt{d-1}}}.
\end{aligned}
$$
Here, $w_n$ is the $n$th closest lift of $j$ to $v$. To derive the final inequality, we used the decomposition of $G$ found in Lemma \ref{lem:factsum}. The paths $p_j$ have $(j-1)|C|$ more $K$ factors than $p_2$. We can thus apply a geometric sum for these terms. 
%
%
\end{proof}

As a simple corollary of this estimate, we can readily show that we have the same type of estimate for the first order term in perturbation series as in Lemma \ref{lem:treesum}. Later, we will show more general estimates.
\begin{cor} \label{cor:firstcor}
Let $\GG$ be a graph with excess at most $1$. Fix a vertex $i$ in $\GG$ and let $R_k$ denote all of the vertices in $\GG$ of distance exactly $k$ from $i$ in $\GG$. Then, we have the following estimate,
\begin{equation}
\sum_{j \in R_k} |G_{i j}(Ext(\GG, m_{\infty}, m_{\infty})|^2 \lesssim 1.
\end{equation}
\end{cor}
\begin{proof}
A graph with excess $0$ would correspond to an infinite tree after the expansion, so this estimate is already known.

We can understand graphs of excess $1$ by, again, referring to the infinite cover of the tree.
Recall that we have,
\begin{equation}
G_{i j}(Ext(\GG, m_{\infty}, m_{\infty}))= \sum_{l=1}^{\infty} (G_{\TT})_{v w_l},
\end{equation}
where $v$ is a fixed lift of $i$ and $w_l$ is a set of lifts of $w$ in order from closest to most distant.
Furthermore if $j$ is of distance $k$ from $i$ then $w_1$ is of distance $k$ from $v$ and $w_2$ is of distance $\ge k+|C|$ from $v$.

The Cauchy-Schwartz inequality will show that,
\begin{equation}
 |\sum_{l=1}^{\infty} (G_{\TT})_{v w_l}|^2 \le \sum_{l=1}^{\infty} |(G_{\TT})_{v w_l}|^2 e^{l \theta} \sum_{l=1}^{\infty} e^{- l \theta} \lesssim |(G_{\TT})_{v w_1}|^2 + |(G_{\TT})_{v w_2}|^2.
\end{equation}
Setting $\theta =- \frac{1}{2} \log
\frac{|m_{\infty}| \vee |m_\infty^{uod}|}{\sqrt{d-1}}$,
makes sure that the infinite sum  $\sum_{l=1}^{\infty} e^{-l \theta}$ is finite. Furthermore, $|(G_{\TT})_{v w_l}|^2 e^{l \theta}$ will still have exponential decay since the factor $e^{\theta}$ is chosen to be smaller than the ratio $\frac{(G_{\TT})_{v w_{l-1}}}{(G_{\TT})_{v w_l}}$.

By applying this inequality to all vertices $j$ that are in $R_k$, we see that we can bound,
\begin{equation}
\sum_{j \in R_k}|G_{i j}(Ext(\GG, m_{\infty}, m_{\infty}))|^2 \lesssim \sum_{ w \in S_{k} \cup S_{k+|C|}} |(G_{\TT})_{v, w}|^2 \lesssim 1.
\end{equation}
Here, $S_k$ denotes the elements that are of distance $k$ from $v$ in $\TT$.
\end{proof}

\subsection{Comparison of Stieltjes transforms}

\begin{lemma}\label{lem:compareStieltjes}
 We show the following equality between the Stieltjes transform proposed in equations \eqref{eq:m-infty-rel} and those coming from free probability computations as in \cite{BD13}:
$$m_\TT^d(z,w)=\sqrt{d-1} m_\star(\sqrt{d-1}z,\sqrt{d-1}w) .
$$
\end{lemma}

\begin{proof} First, notice that if we remove the scaling of edges in the adjacency matrix by $\frac{1}{\sqrt{d-1}}$ , then the corresponding equations for $m_{\infty}$ and $m_{\TT}^d$ would be,
\begin{equation} \label{eq:unrescaled}
\begin{aligned}
& m_{\infty} = \frac{z + d m_{\infty}}{|w|^2-\big( z+ d m_{\infty} \big)(z+ (d-1)m_{\infty})},\\
&m_{\TT}^d = \frac{z+ d m_{\infty}}{|w|^2 - \big(z + d m_{\infty} \big)^2}.
\end{aligned}
\end{equation}

 Consider $\mathcal{T}_d$, our infinite $d$-regular digraph tree; from this, pick a subgraph $S_{d-1}$  in $\mathcal{T}_d$ that is isomorphic to the infinite $d-1$ regular tree $\mathcal{T}_{d-1}$. For every vertex $v$ that is part of $S_{d-1}$, there will be one edge, $i_v$, into it and one edge, $o_v$, out of it that will be part of the original graph $\mathcal{T}_d$, but will not be part of the subgraph $S_{d-1}$. If we remove the edge $o_v$, the connected component remaining that is totally disjoint from $S_{d-1}$ will be isomorphic to $\mathcal{T}_1$ (the $2d-1$-ary tree). If we remove $i_v$, the connected component remaining that is totally disjoint from $S_{d-1}$ will be isomorphic to $\mathcal{T}_2$.  For $|z| > (d-1) + |w|^2$, we compute the Green's function of the Hermitized (unscaled) adjacency matrices of $\mathcal{T}_d$ by counting paths and carefully tracking where such paths uses edges of  $S_{d-1}$. (We will naturally consider $w$ and $w^*$ to correspond to a weighted self-loops at the appropriate vertex.)

Consider trying to compute  $(G_{\TT_d})_{\hat{r}\hat{r}}$ at the root $\hat{r}$. For its first step, it could either take an edge that is part of $S_{d-1}$ or else it could take the outedge $o_{\hat{r}}$. If the path takes the outedge $o_{\hat{r}}$, the path will perform a walk in $\mathcal{T}_1$ before returning to the root by traversing $o_{\hat{r}}$ in the reverse direction. The weighted contribution of paths that take this option would be $\frac{m_{\infty}(z)}{z}$, corresponding to the weighted contribution of walks that return to the root in $\mathcal{T}_1$. Before the walk takes an edge that is part of $S_{d-1}$, it could repeat this process indefinitely. Thus, the total weighted contribution of paths that could occur before the walk's first step in $S_{d-1}$ is,
$$
\sum_{k=0}^{\infty} \left(  \frac{m_{\infty}}{z}\right)^k = \frac{1}{1- \frac{m_{\infty}}{z}}. 
$$

If the path takes an edge that is part of $S_{d-1}$ to get to the vertex $p$, we can repeat a similar argument. Either the next step after this is also a part of $S_{d-1}$ or it could take the inedge $i_p$. (Note that this is still true if we took the weighted edge corresponding to the self loop.) We can count the contribution of all paths that use the inedge $i_p$ as $\frac{1}{1- \frac{m_{\infty}}{z}}$. If we repeat this argument as we traverse edges of $S_{d-1}$ we see that $m_{\TT}^d$ can be expressed as follows. Let $\mathcal{S}_k$ denote the total  weight of all paths in $S_{d-1}$ that use $k$ edges. 
\begin{equation}
m_{\TT}^d = \frac{1}{z} \sum_{k=0}^{\infty}  \frac{\mathcal{S}_k}{z^k}  \left(\frac{1}{1 - \frac{m_{\infty}(z)}{z}} \right)^{k+1}.
\end{equation} 

Thus,
\begin{equation}\label{eq:minusonerecur}
m^{d}_{\TT} = m^{d-1}_{\TT} \left( z - m_{\infty}(z) \right).
\end{equation}

If we return to the equations in \eqref{eq:unrescaled}, we find that,
\begin{equation} \label{eq:invrel}
\frac{1}{m^{d}_{\TT}} = \frac{|w|^2}{z+ d m_{\infty}} - (z + d m_{\infty}) = \frac{|w|^2}{z + d m_{\infty}(z)} - (z + (d-1) m_{\infty}) - m_{\infty}(z) = -m_{\infty}(z) + \frac{1}{m_{\infty}}.
\end{equation}

Solving the quadratic gives,
\begin{equation}
m _{\infty} = \frac{-1 + \sqrt{1+ 4 m_d^2}}{2 m_d}.
\end{equation}
Returning to equation \eqref{eq:minusonerecur}, we see that,
\begin{equation}
m^{d}_{\TT}(z) = m^{d-1}_{\TT }\left( z - \frac{2 m^d_{\TT}}{1+ \sqrt{1 + 4 (m^d_{\TT})^2}} \right).
\end{equation}

This is the same recursion as in equation (4.6) of \cite{BD13}. Notice that though we only confirmed the identity for $|z| \ge (d-1) +|w|^2$, the identity would hold for all values of $z$ by analytic continuation. To fully confirm the equality of Green's functions, it suffices to show the equality of Green's functions when $d=1$.

When $d=1$, we have to solve the quadratic,
\begin{equation}
m_{\infty}^2 + m_{\infty} \left[ z+ \frac{1}{z}  - \frac{|w|^2}{z} \right] +1 = 0.
\end{equation}
Setting the quantity $A= z+ \frac{1}{z} - \frac{|w|^2}{z}$, we see that we have,
\begin{equation}
m_{\infty} = \frac{-A - \sqrt{A^2 - 4}}{2},
\end{equation}
and, using \eqref{eq:invrel}, we see that,
\begin{equation}
\begin{aligned}
m^1_{\TT} &= - \frac{1}{\sqrt{A^2 - 4}} = - \frac{1}{\sqrt{ \left( z+ \frac{1}{z} - \frac{|w|^2}{z}\right)^2 -4}}\\
& = - \frac{z}{\sqrt{ \left( z^2+ 1 - |w|^2 \right)^2 -4z^2}}\\
&=  - \frac{z}{\sqrt{ \left( z^2+ 2z +  1 - |w|^2)(z^2 - 2z +1 - |w|^2\right)}}\\
&= - \frac{z}{\sqrt{ \left( (z+1)^2 - |w|^2) ((z-1)^2 - |w|^2\right)}}\\
&= - \frac{z} {\sqrt{ (z+1 -|w|)(z+1 + |w|)(z-1 - |w|)(z-1 +|w|)}}\\
&= -\frac{z}{\sqrt{(z^2- (1-|w|)^2) (z^2 - (1+|w|)^2)}}.
\end{aligned}
\end{equation}

This Green's function can be checked to correspond to the density of equation (4.1) in \cite{BD13}. \end{proof}

\bibliographystyle{abbrv}
\bibliography{dgraph}       

\begin{thebibliography}{10}

\bibitem{AGZ10}
G.~Anderson, A.~Guionnet, and O.~Zeitouni.
\newblock {\em An introduction to random matrices}, volume 118 of {\em
  Cambridge Stud. Adv. Math.}
\newblock Cambridge Univ.Press , Cambridge, UK, 2010.

\bibitem{BCZ14}
A.~Basak, N.~Cook, and O.~Zeitouni.
\newblock Circular law for the sum of random permutation matrices.
\newblock {\em Electron. Jour. Probab.}, 23(33):1--51, 2018.

\bibitem{BD13}
A.~Basak and A.~Dembo.
\newblock Limiting spectral distribution of sums of unitary and orthogonal
  matrices.
\newblock {\em Electron. Commun. Probab.}, 18(69):1--19, 2013.

\bibitem{BPZ19}
A.~Basak, E.~Paquette, and O.~Zeitouni.
\newblock Regularization of non-normal matrices by {G}aussian noise -- the
  banded {T}oeplitz and twisted {T}oeplitz cases.
\newblock {\em Forum Math. Sigma}, 7(e3):72, 2019.

\bibitem{BPZ20}
A.~Basak, E.~Paquette, and O.~Zeitouni.
\newblock Spectrum of random perturbations of {T}oeplitz matrices with finite
  symbols.
\newblock {\em Trans. Amer. Math. Soc.}, 373(7):4999–5023, 2020.

\bibitem{RHY19}
R.~Bauerschmidt, J.~Huang, and H.-T. Yau.
\newblock Local {K}esten-{McK}ay law for random regular graphs.
\newblock {\em Comm. Math. Phys.}, 369(2):523 --636, 2019.

\bibitem{RKY17}
R.~Bauerschmidt, A.~Knowles, and H.-T. Yau.
\newblock Local semicircle law for random regular graphs.
\newblock {\em Comm. Pure Appl. Math}, 70(10):1898 --1960, 2017.

\bibitem{BC12}
C.~Bordenave and D.~Chafai.
\newblock Around the circular law.
\newblock {\em Probability Surveys}, 9:1--89, 2012.

\bibitem{Cook}
N.~Cook.
\newblock The circular law for random regular digraphs.
\newblock {\em Ann. Inst. H. Poincar\'e Probab. Statist.}, 55(4):2111–2167,
  2019.

\bibitem{Coste21}
S.~Coste.
\newblock The spectral gap of sparse random digraphs.
\newblock {\em Ann. Inst. H. Poincar\'e Probab. Statist.}, 57(2):644--684,
  2021.

\bibitem{CLZ}
S.~Coste, G.~Lambert, and Y.~Zhu.
\newblock The characteristic polynomial of sums of random permutations and
  regular digraphs.
\newblock {\em Int. Math. Res. Notices}, 2024:2461–2510, 2024.

\bibitem{EY17}
L.~Erdos and H.-T. Yau.
\newblock {\em Dynamical Approach to Random Matrix Theory}.
\newblock AMS and Courant Institute of Mathematical Sciences, 2017.

\bibitem{Girko84}
V.~Girko.
\newblock The circular law.
\newblock {\em Teor. Veroyatnost. i Primenen}, 29(4):669--679, 1984.

\bibitem{GKZ11}
A.~Guoinnet, M.~Krishnapur, and O.~Zeitouni.
\newblock The single ring theorem.
\newblock {\em Ann. of Math.}, 174(2):1189--1217, 2011.

\bibitem{GWZ14}
A.~Guoinnet, P.~Wood, and O.~Zeitouni.
\newblock Convergence of spectral measure of non normal matrices.
\newblock {\em Proc. Amer. Math. Soc.}, 142(2):667--679, 2014.

\bibitem{HL00}
U.~Haagerup and F.~Larsen.
\newblock Brown's spectral distribution measure for {R}-diagonal elements in
  finite von {N}eumann algebras.
\newblock {\em J. Funct.Anal.}, 176:331--367, 2000.

\bibitem{Huang21}
J.~Huang.
\newblock Invertibility of adjacency matrices for random d-regular graphs.
\newblock {\em Duke Math. Jour.}, 170(18):3977 --4032, 2021.

\bibitem{HuangYau}
J.~Huang and H.-T. Yau.
\newblock Spectrum of random d-regular graphs up to the edge.
\newblock {\em Comm. Pure Appl. Math}, 77:1635--1723, 2024.

\bibitem{Janson95}
S.~Janson.
\newblock Random regular graphs:asymptotic distributions and contiguity.
\newblock {\em Combin. Probab. Comput.}, 4(4):369--405, 1995.

\bibitem{LLTTY18}
A.~Litvak, A.~Lytova, K.~Tikhomirov, N.~Tomczak-Jaegarmann, and P.~Youssef.
\newblock The smallest singular value of a shifted d-regular random square
  matrix.
\newblock {\em Prob.Thr. Rel. Fields}, 173:1301--1347, 2019.

\bibitem{LLTTY}
A.~Litvak, A.~Lytova, K.~Tikhomirov, N.~Tomczak-Jaegermann, and P.~Youssef.
\newblock Circular law for sparse random regular digraphs.
\newblock {\em J. Euro. Math. Soc.}, 23(2):467–501, 2020.

\bibitem{LT22}
A.~Litvak and K.~Tikhomirov.
\newblock Singularity of sparse {B}ernoulli matrices.
\newblock {\em Duke Math. Jour.}, 171(5):1135--1233, 2022.

\bibitem{Male11}
C.~Male.
\newblock Traffic distributions and independence: Permutation invariant random
  matrices and the three notions of independence.
\newblock {\em Memoirs of the {AMS}}, 267:1300, 2020.

\bibitem{MNR}
F.~L. Metz, I.~Neri, and T.~Rogers.
\newblock Spectral theory of sparse non-{H}ermitian random matrices.
\newblock {\em J. Phys. A: Math. Theo.}, 52(43):434003, 2019.

\bibitem{NP}
H.~H. Nguyen and A.~Pan.
\newblock A note on the singularity probability of random directed d-regular
  graphs.
\newblock {\em Eur. J. Combinatorics}, 122:104039, 2024.

\bibitem{RV14}
M.~Rudelson and R.~Vershynin.
\newblock Invertibility of random matrices: unitary and orthogonal
  perturbations.
\newblock {\em Bull. Amer. Math. Soc.}, 27(2):293--338, 2014.

\bibitem{SSS}
A.~Sah, J.~Sahasrabudhe, and M.~Sawhney.
\newblock The limiting spectral law for sparse {iid} matrices.
\newblock {\em Forum Math. Pi {(}or arXiv:2310.17635{)}}, 13, 2025.

\bibitem{SV}
J.~Sj\"ostrand and M.~Vogel.
\newblock General {T}oeplitz matrices subject to {G}aussian perturbations.
\newblock {\em Ann. Henri Poincar\'e}, 22:49–81, 2021.

\bibitem{Sniady02}
P.~Sniady.
\newblock Random regularization of {B}rown spectral measure.
\newblock {\em J. Funct. Anal.}, 193:291--313, 2002.

\bibitem{TV10}
T.~Tao and V.~Vu.
\newblock Random matrices: universality of {ESD}s and the circular law,.
\newblock {\em Ann. Probab.}, 38(5):2023--2065, 2010.

\end{thebibliography}


\end{document}